\documentclass[11pt,a4paper]{article}
\usepackage[utf8]{inputenc}
\usepackage{amsmath}
\usepackage{amsfonts}
\usepackage{hyperref}
\hypersetup{colorlinks=true, linkcolor=blue, citecolor=blue,
    urlcolor=blue}
\usepackage{amssymb}
\usepackage{extarrows}
\usepackage{amsthm}
\usepackage{hyperref}
\usepackage{cite}
\usepackage{enumerate}
\PassOptionsToPackage{normalem}{ulem}
\usepackage{ulem}
\usepackage[margin=0.8in]{geometry} 
\usepackage{array}
\usepackage{cases}
\usepackage{stmaryrd}
\usepackage{mathrsfs}
\usepackage{longtable}
\usepackage{algorithm}
\usepackage{algpseudocode}
\usepackage{booktabs}
\allowdisplaybreaks[4]

\usepackage{color}
\usepackage{dsfont}
\usepackage{marginnote}
\usepackage{enumitem}
\usepackage{verbatim}
\usepackage{graphicx}

\theoremstyle{plain}
\newtheorem{theorem}{\protect Theorem}[section]
\newtheorem{prop}[theorem]{\protect Proposition}
\newtheorem{definition}[theorem]{\protect Definition}

\newtheorem{lemma}[theorem]{\protect Lemma}
\newtheorem{remark}[theorem]{\protect Remark}

\newtheorem{corollary}[theorem]{\protect Corollary}

\newtheorem{assumption}[theorem]{\protect Assumption}
\newtheorem{condition}[theorem]{\protect Condition}

\newcommand{\Rb}{\mathbb{R}}
\newcommand{\Eb}{\mathbb{E}}
\newcommand{\Xc}{\mathcal{X}}

\newcommand{\Dc}{\mathcal{D}}
\newcommand{\Fb}{\mathbb{F}}
\newcommand{\Pb}{\mathbb{P}}
\newcommand{\Pc}{\mathcal{P}}

\newcommand{\Mc}{\mathcal{M}}

\newcommand{\Lc}{\mathcal{L}}

\newcommand{\Ab}{\mathbb{A}}

\newcommand{\Bc}{\mathcal{B}}
\newcommand{\Ac}{\mathcal{A}}
\newcommand{\Xb}{\mathbb{X}}

\newcommand{\proj}{\Pi_{\overline{\Xb}}}
\newcommand{\indicator}[1]{\mathbf{1}_{\{#1\}}}
\title{Mean Field Game with Reflected Jump Diffusion Dynamics:\\ A Linear Programming Approach
}

\author{Zongxia Liang\thanks{Department of Mathematical Sciences, Tsinghua University, Beijing, China. \url{ liangzongxia@tsinghua.edu.cn}}
\and Xiang Yu\thanks{Department of Applied Mathematics,  The Hong Kong Polytechnic University, Kowloon, Hong Kong. \url{xiang.yu@polyu.edu.hk}}
	\and Keyu Zhang\thanks{
	Department of Mathematical Sciences, Tsinghua University, Beijing, China. \url{Zhangky21@mails.tsinghua.edu.cn}}}
\date{}
\begin{document}
	\maketitle
	\vspace{-0.4in}
		\begin{abstract}
			This paper develops a linear programming approach for mean field games with reflected jump-diffusion dynamics.  We first prove the equivalence between the mean field equilibria in the linear programming formulation and those in the weak relaxed control formulation under some measurability and growth conditions on model coefficients. Building upon the characterization of the occupation measure in the equivalence result, we further establish the existence of linear programming mean field equilibria under fairly general conditions on model coefficients. Finally, a numerical example is presented to illustrate the computation of a mean field equilibrium using the linear programming formulation.

			\vspace{0.1in}	
			\noindent\textbf{Keywords}: Mean field games, linear programming approach, reflected diffusion with jumps, mean field equilibrium, occupation measure 
            
    \vspace{0.1in}
	\noindent \textbf{MSC2020 subject classifications}: 49N80, 60H30, 60G07\\
		\end{abstract}
		
	\vspace{-0.2in}
		\section{Introduction}
         	Mean field game (MFG), introduced independently by Lasry and Lions \cite{lasry_mean_2007} and Huang, Malham\'e, and Caines \cite{caines_large_2006}, has become a powerful and tractable paradigm for analyzing approximate Nash equilibria in large population games with homogeneous players, where all players interact with the population only through the distribution of the aggregated states or strategies. The core question in MFG theory is the existence of a mean field equilibrium (MFE), for which different approaches have been developed over the past decades. For instance, the analytical approach, initially proposed by Lasry and Lions \cite{lasry_mean_2007}, characterizes the MFE by solving a coupled system of Hamilton–Jacobi–Bellman (HJB) and Fokker–Planck (FP) equations: the backward HJB equation determines the optimal value function of a representative agent given the population distribution, while the forward FP equation describes the evolution of the population distribution under the optimal control. The second approach, introduced by Carmona and Delarue \cite{carmona_probabilistic_2013,carmona_mean_2013}, is probabilistic based on the stochastic maximum principle, where the HJB equation is replaced by a backward stochastic differential equation (BSDE). The other probabilistic approach, introduced by Lacker \cite{lacker_mean_2015}, employs the idea of relaxed control formulation in El
		Karoui et al. \cite{nicole_el_compactification_1987} and Haussmann and Lepeltier \cite{haussmann_existence_1990} and develops some compactification arguments to establish the existence of Markovian MFEs in general settings. Recently, this compactification method has
		been generalized and refined to analyze the existence of MFE for MFG problems in various
		context such as: MFG with Brownian common noise in \cite{common_2016}; MFG with absorption in \cite{campiFis}, MFG  with controlled jumps in \cite{benazzoli_mean_2020}; MFG with singular controls in \cite{fu_mean_2017}, MFG with reflected diffusion state dynamics in \cite{bo2025meanfieldgamecontrols,jarni2025meanfieldgamesreflected}, and MFG with Poissonian common noise and pathwise formulation in \cite{BWWY25},
        just to name a few. 
		
		Another well-known relaxation technique to ensure the existence of solution for stochastic control and game problems is the linear programming (LP) formulation. This approach recasts a stochastic control problem as an LP problem over the space of occupation measures (see, e.g., \cite{stockbridge_time-average_1990,kurtz_existence_1998,kurtz_stationary_2001,bhatt_occupation_1996,helmes_linear_2007,dufour_existence_2012,kurtz2017linearprogrammingformulationssingular}). Comparing with the compactification method, the underlying mechanisms to establish the existence of optimal control are different: the compactification arguments work directly on the law of the controlled state process, whereas the LP method focuses on the occupation measures, and thus their convexity and compactness arguments differ substantially.  Moreover, the existence theorems obtained via compactification arguments are rather abstract and cannot be exploited for the development of numerical algorithms. In contrast, the LP formulation is well-suited for constructing numerical schemes, see, e.g,  \cite{cho2002linear,mendiondo1998approximation, bouveret2022technological, aid2020entry,dumitrescu2023linear,Stockbridge2018}.  In the literature of MFGs, however, the LP approach has received relatively less attention. To name a few related studies in this direction, the LP formulation was utilized in MFG of stopping in \cite{benazzoli_mean_2020}, and was later extended to a more general setting including mixed optimal stopping/control in \cite{dumitrescu_control_2021}.  \cite{guo2024mf} introduced a primal--dual formulation for discrete-time MFGs via the complementarity condition in linear programming. In a similar spirit, \cite{guo2022optimization} examined a discrete-time  Stackelberg MFG and established its equivalence with a minimax optimization problem. \cite{yu2025majorminormeanfieldgame} investigated a discrete-time major–minor MFG by combining the LP formulation with the entropy regularization in the major's problem.    More recently, \cite{guo2025continuoustimemeanfieldgames} developed a primal-dual formulation for continuous
		time MFGs based on the LP formulation, where the dual problem is formulated as a maximization problem over smooth subsolutions of the Hamilton-Jacobi-Bellman (HJB) equation.

        In the present paper, we focus on a class of continuous-time MFG with two key features: (i) the state dynamics are constrained to a convex domain via multi-dimensional reflections, and they are subject to sudden shocks modeled by a point process with a control-dependent jump intensity; (ii) the cost functional involves a terminal cost. 
        To the best of our knowledge, MFGs in this setting have not yet been studied using the probabilistic approach. In the presence of jumps, the state process no longer takes values in the space of continuous functions, but in the Skorokhod space $\mathbb{D}([0,T],\Rb^d)$ of all c\`adl\`ag functions. When applying the conventional compactification method, the convergence $x^n \to x$ in the Skorokhod space under the usual topologies (for example, the $J_1$ topology) does not imply $x^n(T) \to x(T)$ as $n \to \infty$. This limitation renders the compactification method inadequate in tackling the terminal cost in our model setting with jumps. In contrast, the LP approach circumvents these issues when we work on the space of occupation measures, which avoids direct engagement with the  topology of the path space, thereby simplifying the analysis. Therefore, we develop the LP approach in the present paper to address the existence of MFE of the MFG in the context of reflected jump-diffusion dynamics. Our main theoretical contributions can be summarized as two-fold:

\begin{enumerate}
	\item Under mild measurability and growth conditions on the coefficients, we establish the equivalence between LP-mean field equilibria (LPMFEs) and the weak mean field equilibria (MFEs) in the relaxed control formulation in the model with reflected jump-diffusion dynamics (see Proposition \ref{thm:prob-rep-R0} and Theorem \ref{theorem:equi}), which is new to the literature. A key challenge is that, in the presence of jumps, the reflection process is not necessarily continuous. To handle the associated discontinuities, we reformulate the cost functional to depend only on the continuous part of the reflection, absorbing the jump effects into a modified running cost (see Remark \ref{remark:jump}). This reformulation naturally motivates our LP approach. Establishing this equivalence requires a two-fold argument. One direction is relatively standard: a weak MFE naturally induces an admissible occupation measure, thus yielding an LPMFE. The main difficulty lies in the other direction.  In a classical diffusion model, this direction has been established in \cite{guo2025continuoustimemeanfieldgames} by exploiting the superposition principle for diffusion processes (see \cite{trevisan_well-posedness_2016}), which ensures that any admissible occupation measure for the linear programming induces a weak solution of the controlled state process. However, in the presence of reflected dynamics, it becomes challenging to consider the controlled Fokker-Planck equation and to generalize the superposition principle for reflected SDE, let alone the jump term in our setting will further complicate the analysis. Instead,
    we resort to generalizing the result developed by  Kurtz and Stockbridge in \cite{kurtz_stationary_2001} to our new setting, which essentially asserts that for each occupation measure $\mu$ satisfying a linear constraint, there exists a stationary solution to a (singular) martingale problem with marginal distribution $\mu$. With the aid of this result, we show that for a given measurable flow of measures $\mu$, any occupation measure $(\nu, m, \lambda) \in \Dc_{P}(\mu)$ (see Definition~\ref{def:lpocm}) corresponds to a solution of a controlled martingale problem. Furthermore, by a martingale representation argument, the occupation measure induces an admissible relaxed control (see Definition~\ref{weak_relax}) in the context of reflected jump-diffusion dynamics, thereby drawing the equivalence between two formulations.
	
	\item  We establish the existence of an LPMFE (see Definition \ref{lpdef}) under general conditions comparable to those using the compactification approach, e.g., \cite{lacker_mean_2015}. This marks a significant improvement over the existing MFG literature using the LP approach, which typically relies on more restrictive assumptions. For instance, \cite{benazzoli_mean_2020} considered only measure-independent coefficients. While \cite{dumitrescu_control_2021} extended the analysis to scenarios with measure-dependent coefficients, this dependence was limited to a finite number of bounded moments. More recently, \cite{dumitrescu2023linear} proved the existence result with general measure dependence in the cost functional, but required the dynamic coefficients to be measure-independent and the control space to be compact. In this paper, we successfully relax all these assumptions. To this end, our proof, inspired by the approximation procedure in \cite{lacker_mean_2015}, consists of two main steps: Firstly, establishing the existence result in the case with bounded coefficients and a compact control space; Secondly, extending the existence result to the original setting with general model assumptions via the approximation argument. This demonstrates that the LP approach is as powerful as the compactification method for establishing the existence result. Moreover, the existence of an LPMFE directly yields the existence of an MFE with Markovian relaxed controls, without resorting to the Markovian projection technique as in the compactification approach. We also present a showcase example of an inventory management problem to exemplify the advantage of the linear programming approach in the numerical computation of an LPMFE in the context of jump-diffusion dynamics.  
\end{enumerate}

The rest of the paper is organized as follows.  Section \ref{sect:2} introduces the MFG problem with reflected jump-diffusion dynamics and its linear programming formulation based on occupation measures. Section \ref{sec:eqv}
establishes the crucial equivalence between the LPMFEs and MFEs in the weak relaxed control formulation. Section \ref{sect:4} proves the existence of an LPMFE in the setting of reflected jump-diffusion dynamics. Section \ref{sect:5} presents a numerical implementation to illustrate the application of the LP approach.   Finally, the paper ends with an Appendix, collecting some technical results used in previous sections.

\subsection{Notations and Preliminaries}
For a topological space $(E, \tau)$, we denote by $\mathcal{B}(E)$ its Borel $\sigma$-algebra, by $\mathcal{P}(E)$ the set of probability measures on $E$, and by $\mathcal{M}_+(E)$ (resp.\ $\mathcal{M}(E)$) the set of positive (resp.\ signed) finite Borel measures on $E$. For any $x \in E$, $\delta_x \in \mathcal{P}(E)$ denotes the Dirac measure at $x$. If $(E, d)$ is a metric space and $p \geq 1$, we denote by $\mathcal{P}_p(E)$ the set of $\mu \in \mathcal{P}(E)$ with a finite $p$-th moment, i.e., $\int_E d(x, x_0)^p \mu(\mathrm{d}x) < \infty$ for some $x_0 \in E$. The spaces $\mathcal{M}^p_{+}(E)$ and $\mathcal{M}^p(E)$ are defined analogously using the total variation measure $|\mu|$. For $\mu, \nu \in \mathcal{P}_p(E)$, the $p$-Wasserstein distance is given by
\[
    W_{E, p}(\mu, \nu) := \left( \inf_{\gamma \in \Pi(\mu, \nu)} \int_{E \times E} d(x, y)^p \, \gamma(dx, dy) \right)^{1/p},
\]
where $\Pi(\mu, \nu)$ is the set of couplings of $\mu$ and $\nu$. Unless otherwise specified, $\mathcal{P}_p(E)$ is equipped with the metric $W_{E, p}$, and we abbreviate this to $W_p$ when the context is clear. If $E$ is a Polish space, then so is $(\mathcal{P}_p(E), W_{E, p})$.

We denote by $M(E)$ the space of Borel measurable functions, by $B_b(E)$ the space of bounded Borel measurable functions and by $C(E)$ the space of continuous functions. The subspace of bounded continuous functions, $C_b(E)$, is endowed with the supremum norm $\|\varphi\|_\infty := \sup\limits_{x \in E} |\varphi(x)|$. For metric spaces $E_1, E_2$, $C(E_1, E_2)$ (resp. $M(E_1,E_2)$) denotes the space of $E_2$-valued continuous (resp. measurable) functions on $E_1$. 
 We denote by $\mathrm{BL}(E)$ the space of bounded Lipschitz functions on $E$, equipped with the norm 
\begin{align*}
\|\varphi\|_{BL}:=\|\varphi\|_{\infty}+\sup_{x\neq y}\frac{|\varphi(x)-\varphi(y)|}{|x-y|}.
\end{align*}
Fix a finite time horizon $T > 0$. For a Polish space $E$, 
let $\mathbb{D}([0, T], E)$ denote the space of càdlàg (RCLL) paths, equipped with the Skorokhod topology. We denote by $C_b^{1,2}([0,T] \times E)$  the space of functions in $C_b([0,T] \times E)$ whose partial derivatives $\partial_t u, \partial_x u$, and $\partial_{xx} u$ are also in $C_b([0,T] \times E)$.

Let $\mathbb{X} \subset \mathbb{R}^d$ be the state space, and let the action space $\mathbb{A}$ be a closed subset of a Euclidean space.   Denote by $\overline{\mathbb{X}}$ the closure of $\mathbb{X}$.  We consider the space $M_p := M([0, T],\mathcal{P}_p(\overline{\mathbb{X}}))$ of measurable flows of probability measures. For any flow $\mu\in M_p$, its essential supremum $p$-th moment is defined as
\[
    \|\mu\|^p_{T} := \operatorname*{ess\,sup}_{t\in[0,T]} \int_{\overline{\mathbb{X}}}|z|^p\,\mu_t(dz).
\]
We denote by $B_p:=\left\{ \mu \in M_p : \|\mu\|_T^p < \infty \right\}.$ 
The space $M_p$ is endowed with the topology of convergence in measure, denoted by $\tilde{\tau}_p$, which is induced by the metric
\begin{align*}
    d^M_{p}(\mu^1,\mu^2)=\int_0^T 1\wedge W_p(\mu^1_t,\mu^2_t)\,dt, \quad \forall \  \mu^1, \mu^2 \in M_p.
\end{align*}
A sequence $(\mu^n)_n \subset M_p$ converges to $\mu \in M_p$ in this topology if for any $\epsilon>0$,
\[
    \lim_{n\to\infty} \int_0^T \mathbf{1}_{\{t:\,W_p(\mu^n_t, \mu_t)>\epsilon\}}\,dt = 0.
\]
A standard result is that if $\mu^n \to \mu$ in $M_p$, then there exists a subsequence $(\mu^{n_k})_k$ such that $W_p(\mu_t^{n_k}, \mu_t)\to 0$ for almost every $t \in [0,T]$.

Let $\tilde{V}$ denote the set of measurable flows $m = \{m_t\}_{t \in [0,T]}$ of finite signed Borel measures on $\overline{\mathbb{X}} \times \mathbb{A}$ such that
$\int_0^T |m_t|(\overline{\mathbb{X}} \times \mathbb{A}) \, dt < \infty$. 
Define $\tilde{V}_0 = \tilde{V} / \!\sim$, where $m^1 \sim m^2$ if $m^1_t = m^2_t$ for almost every $t$. Then $\tilde{V}_0$ is a vector space under the usual sum and scalar multiplication, with the null flow as its zero element. We identify each $m \in \tilde{V}_0$ with the signed measure $m_t(\mathrm{d}x,\mathrm{d}a)\mathrm{d}t$ on $[0,T] \times \overline{\mathbb{X}} \times \mathbb{A}$ and endow $\tilde{V}_0$ with the corresponding weak topology, $\tau_0$.  
For $p \geq 1$, define the subsets of $\tilde{V}_0$,
$$
\tilde{V}_p:=\left\{m \in \tilde{V}_0: \int_0^T \int_{\overline{\Xb}\times\Ab}(|x|^p+|a|^p)|m|_t(\mathrm{d}x,\mathrm{d}a) \mathrm{d} t<\infty\right\}.
$$
 This space is
endowed with the weak topology with respect to continuous functions with $p$-polynomial growth, denoted by $\tau_p$, of the associated measures. We denote by $V_0$ (resp. $V_p$ ) the set of measure flows $\left(m_t\right)_{t \in[0, T]} \in \tilde{V}_0$ (resp. $\tilde{V}_p$) such that $\mathrm{d}t$-a.e. $m_t$ is a probability measure.  The space $\tilde{V}_p$ is a Hausdorff locally convex topological vector space and $V_p$ is metrizable. (cf. \cite[Appendix A]{dumitrescu_control_2021} for more details).

We will often work on the product space $\mathcal{P}_2(\overline{\mathbb{X}}) \times V_2 \times \mathcal{M}_{+}([0,T]\times\partial\mathbb{X})$, which we endow with the product topology $\tau_2 \otimes \tau_2 \otimes \tau_0$. It is a standard result (see, e.g., \cite[Theorem 7.12]{Villani2003Topics}) that on $\mathcal{P}_2(\overline{\mathbb{X}})$, the topology induced by the 2-Wasserstein metric, $W_2$, coincides with $\tau_2$. Because this equivalence extends to spaces of finite measures sharing the same total mass, the topology induced by the $W_2$ metric on $V_2$ is also equivalent to $\tau_2$. We also consider the topology of stable convergence on $V_p$, denoted $\overline{\tau}_p$. A sequence $(m^n)_n \subset V_p$ converges to $m \in V_p$ in this topology if for all test functions  $\varphi:[0,T]\times\overline{\Xb}\times\Ab\rightarrow\Rb$ jointly measurable, continuous in $(x,a)$ for each $t\in[0,T]$ and with $p$-polynomial growth, we have $\lim\limits_{n\rightarrow\infty}\int\int_{[0, T]\times\overline{\Xb}\times\Ab}\varphi(t,x,a)m^n_t(dx,da)dt=\int\int_{[0,T]\times\overline{\Xb}\times\Ab}\varphi(t,x,a)m_t(dx,da)dt$.
Moreover, it follows from \cite[Lemma A.3]{lacker_mean_2015} that on the set $V_p$, convergence in the topology $\tau_p$ implies convergence in the stable topology $\overline{\tau}_p$; hence, the two topologies coincide on $V_p$.  Consequently, for our purposes on $V_2$, we can endow the product space $\mathcal{P}_2(\overline{\mathbb{X}}) \times V_2 \times \mathcal{M}_{+}([0,T]\times\partial\mathbb{X})$ with the equivalent topology $\tau_2 \otimes \overline{\tau}_2 \otimes \tau_0$.

Finally, let \(\mathbb{M}^{d\times d}_{+}\) denote the space of nonnegative definite \(d\times d\) matrices, and let \(I_{d\times d}\) be the \(d\times d\) identity matrix.
For a closed convex set $D \subset \mathbb{R}^d$ and $y \in \mathbb{R}^d$, let $\Pi_D(y)$ denote the projection of $y$ onto $D$. Then, for all $y_1, y_2 \in \mathbb{R}^d$,
\begin{align}\label{lipPi}
    |\Pi_D(y_1) - \Pi_D(y_2)| \le |y_1 - y_2|.
\end{align}

\section{Problem Formulation}\label{sect:2}
\subsection{Mean field game problem}
Let us first introduce the strong formulation of MFG problem with state reflections. Consider the state-dynamics coefficients
\begin{align*}
&b : [0,T] \times \mathbb{R}^d \times \mathcal{P}_2(\overline{\mathbb{X}}) \times \mathbb{A} \to \mathbb{R}^d, 
\quad 
\sigma : [0,T] \times \mathbb{R}^d \times \mathcal{P}_2(\overline{\mathbb{X}}) \times \mathbb{A} \to \mathbb{R}^{d \times d},\\
&\beta : [0,T] \times \mathbb{R}^d \times \mathcal{P}_2(\overline{\mathbb{X}}) \times \mathbb{A} \to \mathbb{R}^d,\quad {\pi:[0,T] \times \mathbb{A}\rightarrow(0,\infty),}
\end{align*}
and the cost functions
\[
f : [0,T] \times \mathbb{R}^d \times \mathcal{P}_2(\overline{\mathbb{X}}) \times \mathbb{A} \to \mathbb{R}, 
\quad
h : [0,T] \times \mathbb{R}^d  \to \mathbb{R}, 
\quad
g : \mathbb{R}^d \times \mathcal{P}_2(\overline{\mathbb{X}}) \to \mathbb{R}.
\]
\begin{assumption}\label{c1} 
Let {$q\geq 4$} be a given real number.
	\begin{itemize}
		\item[(i)] The state space $\mathbb{X}$ is a convex open domain in $\mathbb{R}^d$ with a $C^3$ boundary. We focus on two cases: (1) $\mathbb{X}$ is a bounded set, or (2) $\mathbb{X} = (0,\infty)$.

       \item[(ii)] The control space $\mathbb{A}$ is a closed subset of a Euclidean space.

        \item[(iii)] The initial distribution $m^*_0\in\Pc_{q}(\overline{\Xb})$.
		
		\item[(iv)]  The functions $b,\sigma,\beta,\pi,f$ and $g$ of $(t,x,\mu,a)$ are measurable in $t$ and continuous in $(x,\mu,a)$. The function \(h: [0, T] \times \partial\Xb \to \mathbb{R}\)  is bounded and continuous.

         \item [(v)] There exist constants $c_1, c_2 > 0$ such that for all  $(t,x,\mu,a)\in[0,T] \times \mathbb{R}^d \times \mathcal{P}_2(\overline{\mathbb{X}}) \times \mathbb{A}$, we have
        \begin{align*}
             &-c_1\left(1+|x|^2+\int_{\overline{\Xb}}|z|^2\mu(dz)\right)+c_2|a|^{q}\leq f(t,x,\mu,a)\leq c_1\left(1+|x|^2+\int_{\overline{\Xb}}|z|^2\mu(dz)+|a|^q\right),\\
            &|g(x,\mu)|\leq c_1\left(1+|x|^2+\int_{\overline{\Xb}}|z|^2\mu(dz)\right).
        \end{align*}

        \item [(vi)] There exists a constant $c_3>0$ such that for all $(t,\mu,a)\in[0,T]\times\Pc_2(\overline{\Xb})\times\Ab$ and $x,y\in\overline{\Xb}$, we have
        \begin{align*}
            |b(t,x,\mu,a)-b(t,y,\mu,a)|&+|\sigma(t,x,\mu,a)-\sigma(t,y,\mu,a)|+|\beta(t,x,\mu,a)-\beta(t,y,\mu,a)|\leq c_3|x-y|,
        \end{align*}		 
        and
\begin{align}\label{bc}
            &|b(t,x,\mu,a)|+|\sigma\sigma^{\top}(t,x,\mu,a)|+|\beta(t,x,\mu,a)|\leq c_3 \left(1+|x|+\left(\int_{\overline{\Xb}}|z|^2\mu(dz)\right)^{\frac{1}{2}}+|a|\right).
    \end{align}  
	\end{itemize}
\end{assumption}
For a given pair $(\mu,\rho)\in B_2\times\mathcal{P}_2(\overline{\mathbb{X}})$, consider a filtered probability space $(\Omega, \mathcal{F}, \mathbb{F}, \mathbb{P})$ that supports  a $d$-dimensional standard Brownian motion $W = (W_t)_{t \in [0,T]}$ and  a point process $N = (N_t)_{t \in [0,T]}$. For an $\mathbb{A}$-valued, $\mathbb{F}$-predictable process $\alpha$, the point process $N$ is assumed to have a stochastic intensity given by $\pi(t, \alpha_t)$. The compensated process is denoted by $\tilde{N}=\big( N_t - \int_0^t \pi(s, \alpha_s)ds\big)_{t\in [0,T]}$.
 The dynamics of the controlled, reflected state process $X = (X_t)_{t \in [0,T]}$ are governed by the following stochastic differential equation (SDE): 
\begin{equation}\label{sde_strong}
	dX_t 
	= b\bigl(t, X_t, \mu_t, \alpha_t\bigr)\,dt 
	+ \sigma\bigl(t, X_t, \mu_t, \alpha_t\bigr)\,dW_t
	+ \beta\bigl(t, X_{t-}, \mu_t, \alpha_t\bigr)\,d\tilde{N}_t
	+ m(X_t)\,dR_t,
	\quad
	\mathbb{P} \circ X_0^{-1} = m_0^*,
\end{equation}
where  $m(x)$ denotes the inward normal vector at $x \in \partial \mathbb{X}$, and $R = (R_t)_{t \in [0,T]}$ is a non-decreasing process serving as the reflection term, ensuring that $X_t \in \overline{\mathbb{X}}$ for all $t \in [0,T]$ and satisfying
\begin{equation}\label{Skorokhod1}
R_t = \int_0^t \mathbf{1}_{\{\partial \mathbb{X}\}}(X_s) \, dR_s.
\end{equation}
The objective function of the representative agent is defined by
\[
\mathbb{E}^{\mathbb{P}}\Biggl[
\int_{0}^{T} f\bigl(t, X_{t}, \mu_{t}, \alpha_{t}\bigr)dt
+
\int_{0}^{T} h\bigl(t, X_{t}\bigr)d R_{t}
+
g\bigl(X_{T}, \rho\bigr)
\Biggr].
\]
The goal of the representative agent is to minimize the above objective function over all admissible controls, subjecting to the \emph{mean-field consistency condition}.
\begin{definition}[Mean field equilibrium with strict control]\label{def:mfe}
	For a given measurable flow of measures $\mu \in B_2$,  define $\mathcal{S}(\mu)$ as the set of tuples $U = (\Omega, \mathcal{F}, \mathbb{F}, \mathbb{P}, W, N, X, R, \alpha)$
	such that $(\Omega, \mathcal{F}, \mathbb{F}, \mathbb{P})$ is a filtered probability space, $W$ is a $d$-dimensional standard $\mathbb{F}$-Brownian motion,  $\alpha$ is an $\mathbb{F}$-predictable process with values in $\mathbb{A}$,
     $N$ is a point process such that the compensated process $\tilde{N}$ is an $\mathbb{F}$-martingale,
    and $(X,R)$ is a pair of $\mathbb{F}$-progressively measurable processes such that $X$ is $\overline{\mathbb{X}}$-valued, $R$ is non-decreasing and RCLL, and the pair $(X,R)$ is a solution to the reflected SDE~\eqref{sde_strong} satisfying the Skorokhod condition~\eqref{Skorokhod1}.  Moreover, the following integrability condition holds:
    \begin{align}\label{integ1}
        \Eb^{\Pb}\left[\int_0^T(|X_s|^2+|\alpha_s|^2)ds+|X_T|^2+R^c_T\right]<\infty.
    \end{align}
    For a given pair $(\mu,\rho)\in B_2\times\mathcal{P}_2(\overline{\mathbb{X}})$,
	let $\mathcal{J}^S[\mu,\rho] : \mathcal{S}(\mu) \to \mathbb{R}$ be defined by 
	\[
	\mathcal{J}^S[\mu,\rho](U)
	= \mathbb{E}^{\mathbb{P}}\Biggl[
	\int_{0}^{T} f\bigl(t, X_{t}, \mu_{t}, \alpha_{t}\bigr)\,dt
	+
	\int_{0}^{T} h\bigl(t, X_{t}\bigr)\,dR_{t}
	+
	g\bigl(X_{T}, \rho\bigr)
	\Biggr].
	\]
	The value function of the representative agent's optimal control problem under strict controls given the aggregation $(\mu,\rho)$ is defined by
	\[
	V^{S}(\mu,\rho) := \inf_{U \in \mathcal{S}(\mu)} \mathcal{J}^S[\mu,\rho](U).
	\]
	Then $U^* = (\Omega, \mathcal{F}, \mathbb{F}, \mathbb{P}, W, N, X, R, \alpha)$ is called a mean field equilibrium with strict controls (in the weak sense) if $U^* \in \mathcal{S}(\mu^*)$ and $\mathcal{J}^S[\mu^*,\rho^*](U^*) = V^{S}(\mu^*,\rho^*)$, where $\mu^*_t = \mathbb{P} \circ (X_t)^{-1}$ for a.e. $t \in [0,T]$ and $\rho^*=\mathbb{P} \circ (X_T)^{-1}$.
\end{definition}
\begin{remark}\label{remark:jump}
Notably, the reflection process $R$ is not continuous in general. Potential discontinuities arise when jumps driven by the martingale $\tilde{N}$ push the state process $X$ to exit the domain $\overline{\mathbb{X}}$.
 Let $\Delta R_t := R_t - R_{t-}$ denote the jump size at time $t$, which can be identified as
\begin{equation*}\label{jumpofR}
    \Delta R_t = \left| \Pi_{\overline{\mathbb{X}}}\bigl(X_{t-} + \beta_t\bigr) - \bigl(X_{t-} + \beta_t\bigr) \right| \Delta N_t,
\end{equation*}
where $\beta_t := \beta(t,X_{t-},\mu_t,\alpha_t)$, $\Delta N_t := N_t - N_{t-}$, and $\Pi_{\overline{\mathbb{X}}}$ is the projection operator onto $\overline{\mathbb{X}}$.

This characterization allows us to reformulate the cost functional. By decomposing $R$ into its continuous part, $R^c$, and its jump  component $R^d:=(R^d_t)_{t\in[0,T]}$ with $R^d_t:=\sum_{s\leq t}\Delta R_s$, we have
\begin{align*} 
    &\mathbb{E}\left[
	\int_{0}^{T} f\bigl(t, X_{t}, \mu_{t}, \alpha_{t}\bigr)\,dt
	+
	\int_{0}^{T} h\bigl(t, X_{t}\bigr)\,dR_{t}
	+
	g\bigl(X_{T}, \rho\bigr)
	\right] \nonumber \\
    &\quad= \mathbb{E}\left[
	\int_{0}^{T} f\bigl(t, X_{t}, \mu_{t}, \alpha_{t}\bigr)\,dt
	+
	\int_{0}^{T} h\bigl(t, X_{t}\bigr)\,dR^c_{t}
	+ \sum_{0 < t \le T} h(t,X_t)\Delta R_t +
	g\bigl(X_{T}, \rho\bigr)
	\right] \nonumber \\
    &\quad= \mathbb{E}\left[
	\int_{0}^{T} f\bigl(t, X_{t}, \mu_{t}, \alpha_{t}\bigr)\,dt
	+
	\int_{0}^{T} h\bigl(t, X_{t}\bigr)\,dR^c_{t}
	+ \int_{0}^{T} G(t,X_{t-},\mu_t,\alpha_t) \, dN_t +
	g\bigl(X_{T}, \rho\bigr)
	\right] \nonumber \\
    &\quad= \mathbb{E}\left[
	\int_{0}^{T} \tilde{f}\bigl(t, X_{t}, \mu_{t}, \alpha_{t}\bigr)\,dt
	+
	\int_{0}^{T} h\bigl(t, X_{t}\bigr)\,dR^c_{t}
	+
	g\bigl(X_{T}, \rho\bigr)
	\right],
\end{align*}
where 
\begin{align}\label{tf}
    &G(t,x,\mu,a) := h\left(t, \Pi_{\overline{\mathbb{X}}}\bigl(x + \beta(t,x,\mu,a)\bigr)\right) \left| \Pi_{\overline{\mathbb{X}}}\bigl(x + \beta(t,x,\mu,a)\bigr) - \bigl(x + \beta(t,x,\mu,a)\bigr) \right|,\nonumber\\
    &\tilde{f}(t,x,\mu,a) := f(t,x,\mu,a) + G(t,x,\mu,a)\pi(t,a).
\end{align}
Because $X$ is an RCLL process, the distinction between $X_t$ and $X_{t-}$ vanishes in the resulting Lebesgue integral. Moreover, as $
    0=\int_0^t \mathbf{1}_{\mathring{\Xb}}(X_s)dR_s=\int_0^t \mathbf{1}_{\mathring{\Xb}}(X_s)dR^c_s+\sum_s \mathbf{1}_{\mathring{\Xb}}(X_s)\Delta R_s=\int_0^t \mathbf{1}_{\mathring{\Xb}}(X_s)dR^c_s$,
 we obtain that $R^c_t = \int_0^t \mathbf{1}_{\{\partial \mathbb{X}\}}(X_s)  dR^c_s$.
\end{remark}
\begin{remark}\label{tfremark}
It is clear that under Assumption~\ref{c1}, the function $\tilde{f}$ defined by \eqref{tf} satisfies the following growth condition:
\begin{equation*}
    -c_4\left(1+|x|^2+\int_{\overline{\mathbb{X}}}|z|^2\mu(\mathrm{d}z)+|a|^2\right) + c_2|a|^{q} \le \tilde{f}(t,x,\mu,a) \le c_4\left(1+|x|^2+\int_{\overline{\mathbb{X}}}|z|^2\mu(\mathrm{d}z)+|a|^q\right)
\end{equation*}
for some new constant $c_4$ that depends on $c_1$, $c_3$, $\|\pi\|_{\infty}$, and $\|h\|_{\infty}$. This follows from the fact that for $(t,x,\mu,a) \in [0,T]\times\overline{\mathbb{X}}\times\mathcal{P}_2(\overline{\mathbb{X}})\times\mathbb{A}$,
\begin{align*}
    |G(t,x,\mu,a)\pi(t,a)| 
    &\le \|\pi\|_{\infty}\|h\|_{\infty} \left| \Pi_{\overline{\mathbb{X}}}\bigl(x + \beta(t,x,\mu,a)\bigr) - \bigl(x + \beta(t,x,\mu,a)\bigr) \right| \\
    &\le \|\pi\|_{\infty}\|h\|_{\infty} \left( \left| \Pi_{\overline{\mathbb{X}}}\bigl(x + \beta(t,x,\mu,a)\bigr) - x \right| + |\beta(t,x,\mu,a)| \right) \\
    &\le 2\|\pi\|_{\infty}\|h\|_{\infty} |\beta(t,x,\mu,a)|,
\end{align*}
where we have used the property \eqref{lipPi} in the last inequality. 
\end{remark}

\subsection{Linear programming formulation}\label{sect:3}
We next introduce the LP formulation of the previous MFG problem. To this end, we begin with a preliminary definition.
\begin{definition}\label{def:lpocm}
	Fix a measurable flow of measures $\mu\in B_2$.
	Let \(\Dc_{P}(\mu)\) be the set of triples \((\nu,m,\lambda) \in  {\Pc_2(\overline{\mathbb{X}}) \times V_2 \times \Mc_{+}([0,T]\times \partial\mathbb{X})}\) such that, for all {\(u\in C^{1,2}_b([0,T]\times \overline{\mathbb{X}})\)},
	\begin{align}\label{lpconstraint}
		\int_{\overline{\mathbb{X}}}u(T,x)\,\nu(dx)
		= \int_{\overline{\mathbb{X}}}u(0,x)\,m^*_{0}(dx)
		&+ \int_{[0,T]\times \overline{\mathbb{X}}\times\mathbb{A}} \bigl(\partial_t u+\Lc u\bigr)(t,x,\mu_t,a)\,m_t(dx,da)dt \\
		& + \int_{[0,T]\times \partial \mathbb{X}}\mathcal{A}u(t,x)\,\lambda(dt,dx), \nonumber
	\end{align}
	where
	\begin{align*}
		\Lc u(t,x,\mu,a)
		&:= b(t,x,\mu,a)\cdot\partial_x u(t,x)
		+ \frac{1}{2}\mathrm{tr}\bigl(\sigma\sigma^{\top}(t,x,\mu,a)\,\partial^2_{xx} u(t,x)\bigr) \\
		& + \bigl[ u(t,\Pi_{\overline{\mathbb{X}}}\bigl(x + \beta(t,x,\mu,a)\bigr)) - u(t,x) - \beta(t,x,\mu,a)\cdot\partial_x u(t,x) \bigr] \pi(t,a),
	\end{align*}
	and $
	\mathcal{A}u(t,x) := m(x)\cdot \partial_x u(t,x).
	$
\end{definition}
Then, we have
\begin{definition}[LP formulation of the MFG problem]\label{lpdef} 
		 For a given pair $(\mu,\rho)\in B_2\times\mathcal{P}_2(\overline{\mathbb{X}})$, we define the linear programming objective functional $\mathcal{J}^{L}_{P}[\mu,\rho]:\Dc_{P}(\mu)\to\mathbb{R}$ by
		\begin{align*}
			\mathcal{J}^{L}_{P}[\mu,\rho](\nu,m,\lambda)
			&= \int_{[0,T]\times \overline{\mathbb{X}}\times\mathbb{A}} \tilde{f}(t,x,\mu_t,a)\,m_t(dx,da)\,dt
			+ \int_{[0,T]\times \partial\Xb} h(t,x)\,\lambda(dt,dx) + \int_{\overline{\mathbb{X}}} g(x,\rho)\,\nu(dx), 
		\end{align*}
        where $\tilde{f}$ is defined by \eqref{tf}.
		We say that \((\mu^*,\rho^*,\nu^*,m^*,\lambda^*)\) is an \emph{linear programming mean field equilibrium} (LPMFE), if it satisfies the following two conditions:
		\begin{itemize}
			\item[(i)] \((\nu^*,m^*,\lambda^*) \in \Dc_{P}(\mu^*)\), and for every \((\nu,m,\lambda)\in \Dc_{P}(\mu^*)\),
			\[
			\mathcal{J}^{L}_{P}[\mu^*,\rho^*](\nu^*,m^*,\lambda^*)
			\;\le\;
			\mathcal{J}^{L}_{P}[\mu^*,\rho^*](\nu,m,\lambda).
			\]
			\item[(ii)] The following consistency condition holds
			\[
			m^{*\mathbb{X}}_{t} = \mu^*_{t}, \,\ a.e.\,\ t\in[0,T],
			\quad
			\nu^* = \rho^*,
			\]
			where \(m^{*\mathbb{X}}_{t}\) denotes the marginal of \(m^*_t\) on \(\overline{\mathbb{X}}\).
		\end{itemize}
         Because the consistency condition implies that the mean-field terms $(\mu^*,\rho^*)$ are determined by $(\nu^*, m^*, \lambda^*)$, we will often omit them and simply refer to the triple $(\nu^*, m^*, \lambda^*)$ as an LPMFE.
	\end{definition}
{\begin{remark}\label{remark:s_to_r}
	Here, let us discuss the intuition behind the LP formulation.  
	Let \(U \in \mathcal{S}(\mu)\) with $\mu \in B_2$, and define the associated occupation measures by
	\begin{align*}
		&\nu := \mathbb{P} \circ X_T^{-1}, \\
		&m_t(B\times C) := \mathbb{E}^\mathbb{P}\!\left[\mathbf{1}_{B}(X_t)\mathbf{1}_{C}(\alpha_t)\right] = \mathbb{E}^\mathbb{P}\!\left[\mathbf{1}_{B}(X_{t-})\mathbf{1}_{C}(\alpha_t)\right], \quad B \in \mathcal{B}(\overline{\mathbb{X}}),\,\,C \in \mathcal{B}(\mathbb{A}),\,\ t-a.e., \\
		&\lambda(D) := \mathbb{E}^\mathbb{P}\!\left[\int_0^T \mathbf{1}_{D}(t,X_t)\,dR^c_t\right] = \mathbb{E}^\mathbb{P}\!\left[\int_0^T \mathbf{1}_{D}(t,X_{t-})\,dR^c_t\right], \quad D \in \mathcal{B}([0,T]\times\partial\mathbb{X}),
	\end{align*}
	where we have used the fact that \(X_t(\omega) = X_{t-}(\omega)\) for a.e. \(t\in[0,T]\) and for all \(\omega\in\Omega\), together with the continuity of $R^c$.   From the integrability condition in \eqref{integ1}, it follows that $(\nu,m,\lambda)\in \Pc_2(\overline{\Xb})\times V_2\times\Mc_{+}([0,T]\times\partial\Xb)$.  Then, the cost functional can be expressed as
	\[
	\mathcal{J}^S[\mu,\rho](U)
	= \int_{[0,T]\times\overline{\mathbb{X}}\times\mathbb{A}} \tilde{f}(t,x,\mu_t,a)\,m_t(dx,da)dt
	+ \int_{0}^{T}\!\!\int_{\partial\mathbb{X}} h(t,x)\,\lambda(dt,dx)
	+ \int_{\overline{\mathbb{X}}} g(x,\rho)\,\nu(dx).
	\]
	On the other hand, applying Itô’s formula to any test function \(u \in C_b^{1,2}([0,T]\times\overline{\mathbb{X}})\) yields
	\begin{align*}
		u(T,X_T) &= u(0,X_0)
		+ \int_{0}^{T} \bigl(\partial_t u + \Lc u\bigr)(t,X_{t-},\mu_t,\alpha_t)\,dt
		+ \int_{0}^{T} m(X_t)\cdot \partial_x u(t,X_{t-})\,dR^c_t \\
		&\quad + \int_{0}^{T} (\sigma\cdot\partial_x u)(t,X_t,\mu_t,\alpha_t)\,dW_t
		+ \int_{0}^{T}\bigl[u(t,\proj (X_{t-}+\beta(t,X_{t-},\mu_t,\alpha_t))) - u(t,X_{t-})\bigr]\,d\tilde{N}_t.
	\end{align*}
     Taking expectations and using the growth condition on \(\sigma\), together with the boundedness of \(u\) and \(\partial_x u\), we get $(\nu,m,\lambda) \in \Dc_{P}(\mu)$.
	
	It is clear that the LP formulation is a relaxation of the weak formulation with strict controls, as every strict control in \(\mathcal{S}(\mu)\) induces a triple in \(\Dc_{P}(\mu)\). In the sequel, we shall show that the LP formulation is, in fact, equivalent to the weak formulation with relaxed controls (see Definition~\ref{weak_relax}).
\end{remark}}

		\section{Equivalence with the Weak Relaxed Control Formulation}\label{sec:eqv}
		In this section, we establish the equivalence between LPMFEs and MFEs in the weak formulation defined below. 
        \begin{definition}[Mean field equilibrium  with relaxed control]\label{weak_relax}
        	For a given measurable flow of measures $\mu \in B_2$, define  \(\mathcal{R}(\mu)\) the set of  tuples $U=(\Omega,\mathcal{F},\mathbb{F},\mathbb{P},M,\mathcal{N}, \Lambda, X, R)$ 
        such that \((\Omega,\mathcal{F},\mathbb{F},\mathbb{P})\) is a filtered probability space, \(\Lambda\)  is an \(\mathbb{F}\)-predictable process with values in $\Pc(\Ab)$, \(M = (M^1,\dots,M^d)\) is a \(d\)-dimensional orthogonal, continuous, \(\mathbb{F}\)-adapted martingale measures on \([0,T]\times\mathbb{A}\) with intensity \(\Lambda_t(da)\,dt\), \(\mathcal{N}\) is a random counting measure on \([0,T]\times\mathbb{A}\) with compensator \(\pi(t,a)\Lambda_t(da)\,dt\),  and $(X,R)$ is a pair of $\mathbb{F}$-progressively measurable processes such that $X$ is an $\overline{\mathbb{X}}$-valued process, and $R$ is a non-decreasing, RCLL process with $R_0=0$.  $(X,R)$ satisfies the following equations:
        \begin{eqnarray*}
        	dX_t 
        	&= &\left(\int_{\Ab} b\bigl(t, X_t, \mu_t, a\bigr)\Lambda_t(da)\right)dt 
        	+ \int_{\Ab}\sigma\bigl(t, X_t, \mu_t, a\bigr)M(dt,da) \\
        	&&+ \int_{\Ab}\beta\bigl(t, X_{t-}, \mu_t, a\bigr)\tilde{\mathcal{N}}(dt,da)
        	+ m(X_t)\,dR_t,
        	\quad
        	\mathbb{P} \circ X_0^{-1} = m_0^*,
        \end{eqnarray*}
        and $R_t = \int_0^t \mathbf{1}_{\{\partial \Xb\}}(X_s)dR_s$, where $\tilde{\mathcal{N}}$ denotes the compensated random measure, i.e., $\tilde{\mathcal{N}}(dt,da)=\mathcal{N}(dt,da)-\pi(t,a)\Lambda_t(da)dt$. Moreover, the following integrability condition holds:
        \begin{align}\label{integ2}
            \Eb^{\Pb}\left[\int_0^T\int_{\Ab}(|X_t|^2+|a|^2)\Lambda_t(da)dt+|X_T|^2+R^c_T\right]<\infty.
        \end{align} 
        For a given pair $(\mu,\rho)\in B_2\times\mathcal{P}_2(\overline{\mathbb{X}})$,
        define $\mathcal{J}^R[\mu,\rho]:\mathcal{R}(\mu)\rightarrow\Rb$  by 
        \[
        \mathcal{J}^R[\mu,\rho](U)
        = \mathbb{E}^{\mathbb{P}}\Biggl[
        \int_{0}^{T}\int_{\Ab} f\bigl(t, X_{t}, \mu_{t}, a\bigr)\Lambda_{t}(da)dt
        +
        \int_{0}^{T} h\bigl(t, X_{t}\bigr)\,dR_{t}
        +
        g\bigl(X_{T}, \rho\bigr)
        \Biggr].
        \]
        The value of the optimization problem  under relaxed control associated to $(\mu,\rho)$ is 
        \[
        V^{R}(\mu,\rho) :=\; \inf_{U\in\mathcal{R}(\mu)} \mathcal{J}^R[\mu,\rho](U).
        \]
        Moreover, we say that $U^*=(\Omega,\mathcal{F},\mathbb{F},\mathbb{P},M,\mathcal{N}, \Lambda, X, R)$ is an mean field equilibrium in the weak relaxed control formulation if 
        \(U^*\in\mathcal{R}(\mu^*)\) and $\mathcal{J}^R[\mu^*,\rho^*](U^*) = V^{R}(\mu^*,\rho^*)$, where $\mu^*_t = \mathbb{P} \circ (X_t)^{-1}$ for a.e. $t \in [0,T]$ and $\rho^*=\mathbb{P} \circ (X_T)^{-1}$.
        \end{definition}
{In what follows, we show that any \((\nu, m, \lambda) \in \mathcal{D}_{P}(\mu)\) admits a probabilistic representation in terms of a reflected  diffusion process with jump under a relaxed control. 
		\begin{prop}\label{thm:prob-rep-R0}
			 Suppose that the coefficients $b, \sigma, \beta, \pi$ are measurable functions such that $\pi$ is bounded, while $b, \sigma, \beta$ satisfy the growth condition \eqref{bc}. Let $(\nu, m, \lambda) \in \mathcal{D}_{P}(\mu)$ for some $\mu \in B_2$.
          Furthermore, let \(v_{t,x}(da)\) be a stochastic kernel such that
			 \[
			 m_t( dx, da)dt = v_{t,x}(da)\, m^{\mathbb{X}}_t(dx)\, dt,
			 \]
			 where \(m^{\mathbb{X}}_t\) denotes the marginal of \(m_t\) on \(\overline{\mathbb{X}}\).  
			 Then, there exists a filtered probability space 
			 \((\Omega, \mathcal{F}, \mathbb{F}, \mathbb{P})\) supporting  \(d\)-dimensional  orthogonal continuous \(\mathbb{F}\)-martingale measures 
			 	\(M=(M^1,\ldots,M^d)\) on \([0,T]\times\mathbb{A}\) 
			 	with intensity \(v_{t, X_{t-}}(da)\,dt\),
			 	 a random counting measure \(\mathcal{N}\) on \([0,T]\times\mathbb{A}\)  
			 	with the compensator \(\pi(t,a)v_{t, X_{t-}}(da)\,dt\),
			 	a pair of \(\mathbb{F}\)-progressively measurable processes \((X,R)\) 
			 	such that \(X\) takes values in \(\overline{\mathbb{X}}\) and 
			 	\(R\) is non-decreasing, RCLL, with \(R(0)=0\).  
			 \((X,R)\) satisfies the following equations
			\begin{eqnarray*}
				dX_t &=& \left(\int_{\mathbb{A}} b(t, X_t, \mu_t, a)\, v_{t, X_t}(da)\right) dt 
				+ \int_{\mathbb{A}} \sigma(t, X_t, \mu_t, a)\, M(dt,da) \\
				&& + \int_{\mathbb{A}} \beta(t, X_{t-}, \mu_t, a)\, \tilde{\mathcal{N}}(dt,da) 
				+ m(X_t)\, dR_t, 
				\,\, \, \, \mathbb{P} \circ X_0^{-1} = m_0^*,
			\end{eqnarray*}
            and $R_t = \int_0^t \mathbf{1}_{\{\partial\mathbb{X}\}}(X_s) dR_s$,
			 where \(\tilde{\mathcal{N}}(dt,da) := \mathcal{N}(dt,da) -\pi(t,a) v_{t, X_{t-}}(da)dt\) denotes the compensated random measure.  Moreover, 
			 \begin{align*}
			 	&\nu = \mathbb{P}\circ X_T^{-1}, \\
			 	&m_{t}(B\times C)
			 	= \mathbb{E}^{\mathbb{P}}\bigl[\mathbf{1}_{B}(X_t)\,v_{t,X_t}(C)\bigr]=\mathbb{E}^{\mathbb{P}}\bigl[\mathbf{1}_{B}(X_{t-})\,v_{t,X_{t-}}(C)\bigr], 
			 	\quad B\in\mathcal{B}(\overline{\mathbb{X}}),\; C\in\mathcal{B}(\mathbb{A}),\; t-a.e.,\\
			 	&\lambda(D)
			 	= \mathbb{E}^{\mathbb{P}}\Bigl[\int_0^T \mathbf{1}_{D}(t,X_t)\,dR^c_t\Bigr]= \mathbb{E}^{\mathbb{P}}\Bigl[\int_0^T \mathbf{1}_{D}(t,X_{t-})\,dR^c_t\Bigr], 
			 	\quad D\in\mathcal{B}([0,T]\times\partial\mathbb{X}).
			 \end{align*}
		\end{prop}
		\begin{proof}
			\textit{Step-1: Construction of the operator and the stationary measure}.  We extend the kernel \(v_{t,x}\) to \((\mathbb{R}_{+} \times \mathbb{R}^d) \setminus ([0, T] \times \overline{\mathbb{X}})\) by setting $v_{t,x} = \delta_{a_0}$
			for an arbitrary \(a_0 \in \mathbb{A}\). Given a measurable flow $\mu\in B_2$, we define the functions \(\bar{b}: \mathbb{R}_+ \times \mathbb{R}^d \to \mathbb{R}^d\), \(\bar{a}: \mathbb{R}_+ \times \mathbb{R}^d \to \mathbb{M}_{+}^{d\times d}\), $\bar{\pi}:\mathbb{R}_+ \times \mathbb{R}^d \to \mathbb{R}^d$ and $\overline{\pi\beta}:\mathbb{R}_+ \times \mathbb{R}^d \to \mathbb{R}^d$ by
			\begin{align*}
				\bar{b}(t, x) &= 
				\begin{cases}
					\displaystyle \int_{\mathbb{A}} b\bigl(t, x, \mu_t, a\bigr)\, v_{t,x}(da), & \text{if } (t, x) \in [0, T] \times \mathbb{R}^d, \\
					0, & \text{otherwise,}
				\end{cases} \\
				\bar{a}(t, x) &=
				\begin{cases}
					\displaystyle  \int_{\mathbb{A}} \sigma\sigma^{\top}\bigl(t, x, \mu_t, a\bigr)\, v_{t,x}(da), & \text{if } (t, x) \in [0, T] \times \mathbb{R}^d, \\
					I_{d\times d}, & \text{otherwise,}
				\end{cases}\\
				\bar{\pi}(t, x) &= 
				\begin{cases}
					\displaystyle \int_{\mathbb{A}} \pi\bigl(t, a\bigr)\, v_{t,x}(da), & \text{if } (t, x) \in [0, T] \times \mathbb{R}^d, \\
					0, & \text{otherwise,}
				\end{cases}
                \\
				\overline{\pi\beta}(t, x) &=
				\begin{cases}
					\displaystyle  \int_{\mathbb{A}}\pi(t,a) \beta\bigl(t, x, \mu_t, a\bigr)\, v_{t,x}(da), & \text{if } (t, x) \in [0, T] \times \mathbb{R}^d, \\
					0, & \text{otherwise.}
				\end{cases}
			\end{align*}
		    Define the measures
		    \begin{align*}
		    	&\tilde{\nu}(B \times C):=\delta_{\{T\}}(B)\nu(C\cap \overline{\Xb}), \quad B \in \mathcal{B}\left(\mathbb{R}_{+}\right), \, C \in \mathcal{B}(\mathbb{R}^d), \\
		    	&\tilde{m}(B \times C):=\int_{B \cap[0, T]} m^{\Xb}_t(C\cap \overline{\Xb} ) d t, \quad B \in \mathcal{B}\left(\mathbb{R}_{+}\right), \, C \in \mathcal{B}(\mathbb{R}^d), \\
		    	&\tilde{\lambda}(B\times C):=\lambda((B \cap[0, T])\times(C\cap \partial\Xb)),  \quad B \in \mathcal{B}\left(\mathbb{R}_{+}\right), \, C \in \mathcal{B}(\mathbb{R}^d),\\
		    	&\tilde{m}_0(B):=m_0^*(B\cap\overline{\Xb}), \quad B \in \mathcal{B}(\mathbb{R}^d).
		    \end{align*}
		    Then, it follows that {$\left(\tilde{\nu}, \tilde{m}, \tilde{\lambda},\tilde{m}_0\right) \in \mathcal{P}\left(\mathbb{R}_{+} \times \mathbb{R}^d\right) \times \mathcal{M}_{+}\left(\mathbb{R}_{+} \times \mathbb{R}^d\right) \times \mathcal{M}_{+}\left(\mathbb{R}_{+} \times \mathbb{R}^d\right)\times \mathcal{P}(\mathbb{R}^d)$}. We also define the transition  kernel \(\bar{Q} : \mathbb{R}_{+} \times \mathbb{R}^d \to \mathcal{P}(\mathbb{R}^d)\) by
            \begin{align*}
		    	&\bar{Q}(t,x,dy)=\begin{cases}
		    		\displaystyle \frac{\int_{\mathbb{A}} \pi(t,a)\delta_{\proj(x+\beta(t,x,\mu_t,a))}(dy)\, v_{t,x}(da)}{\bar{\pi}(t,x)}, & \text{if } (t, x) \in [0, T] \times \mathbb{R}^d, \\
		    		\delta_{x_0}, & \text{otherwise,}
		    	\end{cases} 
		    \end{align*}
		  where \(x_0 \in \mathbb{R}^d\) is arbitrary. Finally, we define the operators
		    \begin{align*}
		    	&\hat{\mathcal{L}}(\gamma \phi)(t, x)=\gamma^{\prime}(t) \varphi(x)+\gamma(t)\left[(\bar{b}(t, x)-\overline{\pi\beta}(t,x))\cdot\nabla_x \varphi(x)+\frac{1}{2}\mathrm{tr}(\bar{a}(t, x) \nabla^2_{xx}\varphi(x))\right]\\&\qquad\qquad\qquad+\bar{\pi}(t,x)\int_{\Rb^d}\gamma(t)\left(\varphi(y)-\varphi(x)\right)	\bar{Q}(t,x,dy),\\
		    	&\hat{\Ac}(\gamma \phi)(t, x)=\gamma(t)m(x)\cdot\nabla_x\varphi(x),
		    \end{align*}
		    for all $\gamma \in C_b^1\left(\mathbb{R}_{+}\right), \varphi \in C_b^2(\mathbb{R}^d)$. Then, based on the definition of $\left(\tilde{\nu}, \tilde{m}, \tilde{\lambda},\tilde{m}_0\right)$, we have
		    \begin{align*}
		    	\int_{\mathbb{R}_{+} \times \mathbb{R}^d} \gamma(t) \varphi(x) \tilde{\nu}(d t, d x)=\gamma(0) \int_{\mathbb{R}^d} \varphi(x) \tilde{m}_0(d x)+\int_{\mathbb{R}_{+} \times \mathbb{R}^d} \hat{\mathcal{L}}(\gamma \varphi)(t, x) \tilde{m}(d t, d x)\\+\int_{\mathbb{R}_{+} \times \mathbb{R}^d} \hat{\mathcal{A}}(\gamma \varphi)(t, x) \tilde{\lambda}(d t, d x).
		    \end{align*}
		    Let $U=\{0,1\}$ and define new operators $\overline{\mathcal{L}}$ and $\overline{\mathcal{A}}$ on the domain$$
		   \mathcal{D}=\left\{\zeta \gamma \varphi: \zeta \in C_b^1\left(\mathbb{R}_{+}\right), \gamma \in C_b^1\left(\mathbb{R}_{+}\right), \varphi \in C_b^2(\mathbb{R}^d)\right\}
		    $$ such that
		    \begin{align*}
		    	\overline{\mathcal{L}}(\zeta \gamma \varphi)(r, s, x, u):= & u \zeta(r) \hat{\mathcal{L}}(\gamma \varphi)(s, x) \\
		    	& +(1-u)\left[\zeta(0) \gamma(0) \int_{\mathbb{R}^d} \varphi(x) \tilde{m}_0(d x)-\zeta(r) \gamma(s) \varphi(x)+\zeta^{\prime}(r) \gamma(s) \varphi(x)\right],\\
		    	\overline{\mathcal{A}}(\zeta \gamma \varphi)(r, s, x):=&\zeta(r)\hat{\Ac}(\gamma \phi)(s, x).
		    \end{align*}
		    Define $\bar{m} \in \mathcal{P}\left(\mathbb{R}_{+} \times \mathbb{R}_{+} \times \mathbb{R}^d \times U\right)$ and $\bar{\lambda}\in \Mc_{+}(\mathbb{R}_{+} \times \mathbb{R}_{+} \times \mathbb{R}^d )$ by
		    \begin{align*}
		    	&\bar{m}(d r, d s, d x, d u):=K^{-1}\left[\delta_1(d u) \delta_0(d r) \tilde{m}(d s, d x)+\delta_0(d u) e^{-r} \indicator{(0,\infty)}(r) d r \tilde{\nu}(d s, d x)\right],\\
		    	&\bar{\lambda}(dr,ds,dx):=K^{-1}\delta_{\{0\}}(r)\tilde{\lambda}(ds,dx),
		    \end{align*}
		    where $K=\tilde{m}\left(\mathbb{R}_{+} \times \mathbb{R}^d\right)+1$ is a normalizing constant. The conditional distribution of $u$ given $(r, s, x)$ under $\bar{m}$ is
		    $$
		    \bar{\eta}(r, s, x, d u)=\delta_1(d u) \indicator{0}(r)+\delta_0(d u) \indicator{(0,\infty)}(r).
		    $$
		   Then, the integration by parts readily yields  
           $$\int \overline{\mathcal{L}}(\zeta \gamma \varphi) d \bar{m}+\overline{\mathcal{A}}(\zeta \gamma \varphi) d \bar{\lambda}=0,$$ 
           for all $\zeta \gamma \varphi \in \mathcal{D}$.

		\textit{Step-2: Verification of the conditions to apply \cite[Corollary 1.10]{kurtz2017linearprogrammingformulationssingular}.} 
		Define the operators $\mathcal{L}_0$ and $\mathcal{A}_0$ on the domain $\mathcal{D}$ by
		\begin{align*}
			\mathcal{L}_0(\zeta \gamma \varphi)(r, s, x, u, v) 
			&= u \, \zeta(r) \bigg( \gamma^{\prime}(s) \varphi(x) 
			+ \gamma(s) \bigg[ v_1 \cdot \nabla_x \varphi(x) 
			+ \frac{1}{2} \mathrm{tr}\!\big(v_2 \nabla^2_{xx} \varphi(x)\big) \\
			&\quad\quad + v_3 \int_{\mathbb{R}^d} \big( \varphi(y) - \varphi(x) \big) \, v_4(dy) \bigg] \bigg) \\
			&\quad + (1-u) \bigg[ \zeta(0) \gamma(0) \int_{\mathbb{R}^d} \varphi(x) \, \tilde{m}_0(dx) 
			- \zeta(r) \gamma(s) \varphi(x) + \zeta^{\prime}(r) \gamma(s) \varphi(x) \bigg], \\
			\mathcal{A}_0(\zeta \gamma \varphi)(r, s, x, z) 
			&= \zeta(r) \gamma(s) \, z \cdot \nabla_x \varphi(x),
		\end{align*}
		where $v = (v_1, v_2, v_3, v_4) \in \mathbb{V}:=\mathbb{R}^d \times \mathbb{M}^{d\times d}_{+} \times \mathbb{R}_{+} \times \mathcal{P}(\mathbb{R}^d)$  
		and $z \in \mathbb{R}^d$. In order to apply \cite[Corollary 1.10]{kurtz2017linearprogrammingformulationssingular}, let us define the transition kernels
		\[
		\eta_0: \mathbb{R}_+ \times \mathbb{R}_+ \times \mathbb{R}^d \times U 
		\to \mathcal{P}\big(\mathbb{V}\big),
		\quad
		\eta_1: \mathbb{R}_+ \times \mathbb{R}_+ \times \mathbb{R}^d \to \mathcal{P}(\mathbb{R}^d),
		\]
		by
		\begin{align*}
			&\eta_0(r, s, x, u, dv) 
			= \delta_{\tilde{b}(s, x)}(d v_1) \,
			\delta_{\bar{a}(s, x)}(d v_2) \,
			\delta_{\bar{\pi}(s,x)}(d v_3) \,
			\delta_{\bar{Q}(s, x, \cdot)}(d v_4), \\
			&\eta_1(r, s, x, dz) 
			= \delta_{m(x)}(dz),
		\end{align*}
		where $\tilde{b}(s, x) = \bar{b}(s, x) - \overline{\pi\beta}(s, x)$. Then
		\begin{align*}
			&\overline{\mathcal{L}}(\zeta \gamma \varphi)(r, s, x, u) 
			= \int_{\mathbb{V}}
			\mathcal{L}_0(\zeta \gamma \varphi)(r, s, x, u, v) \,
			\eta_0(r, s, x, u, dv), \\
			&\overline{\mathcal{A}}(\zeta \gamma \varphi)(r, s, x) 
			= \int_{\mathbb{R}^d} \mathcal{A}_0(\zeta \gamma \varphi)(r, s, x, z) \, 
			\eta_1(r, s, x, dz).
		\end{align*}
			{Let $\psi_0(r, s, x, u, v)=1+u\left(|v_1|+|v_2|+v_3\right)$} and $\psi_1(r,s,x,z)=1+|z|$. Using the growth conditions on $b$, $\sigma$, $\beta$, and $\pi$, we have 
			\begin{align*}
					&\int_{\mathbb{R}_{+} \times \mathbb{R}_{+} \times \mathbb{R}^d \times U \times\mathbb{V}} \psi_0(r, s, x, u, v) \eta_0(r, s, x, u, d v) \bar{m}(d r, d s, d x, d u)\\
                    &\qquad\leq 1+K^{-1}\int_{\Rb_{+}\times\Rb^d}\left(|\tilde{b}(s, x)|+|\bar{a}(s,x)|+\|\pi\|_{\infty}\right)\tilde{m}(ds,dx)\\
                    &\qquad\leq 1+C\left(\int_{[0,T]\times\overline{\Xb}\times\Ab}(1+|x|^2+|a|^2+\|\mu\|^2_{T})m_t(dx,da)dt\right)<\infty,
                    \\
					&\int_{\mathbb{R}_{+} \times \mathbb{R}_{+} \times \mathbb{R}^d \times \Rb^d} \psi_1(r, s, x,z) \eta_1(r, s, x,  d z) \bar{\lambda}(d r, d s, d x) \leq 2 K^{-1}\lambda([0,T]\times\partial\Xb)<\infty,
			\end{align*}
where $C$ is a constant that depends on $K$, $\|\pi\|_{\infty}$, and the constant $c_3$ in \eqref{bc} of Assumption \ref{c1}.
The remaining task is to verify that 
\((\psi_0, \psi_1, \mathcal{L}_0, \mathcal{A}_0)\) 
fulfills\cite[Condition~1.3]{kurtz2017linearprogrammingformulationssingular} or equivalently,  \cite[Condition~1.2]{kurtz_stationary_2001}. For the reader’s convenience, we restate these as Condition \ref{cond:B1} in Appendix \ref{sec:Condition}.
The verification of \textup{(i)} and \textup{(ii)} is straightforward. To show \textup{(iii)}, it suffices to show that one can find a countable subset of 
\(C_b^1(\mathbb{R}_{+})\) 
approximating any function in \(C_b^1(\mathbb{R}_{+})\) 
in the topology of bounded pointwise convergence 
(the same holds for \(C_b^2(\mathbb{R}^d)\))\footnote{A sequence \(\{\varphi_n\}_{n=1}^\infty \subset C_b^2(\mathbb{R}^d)\) 
	is said to converge to \(\varphi \in C_b^2(\mathbb{R}^d)\) in the topology of 
	\emph{bounded pointwise convergence} if, for every \(x \in \mathbb{R}^d\) and 
	for each \(j \in \{0,1,2\}\), the \(j\)-th derivatives satisfy 
	\(\varphi_n^{(j)}(x) \to \varphi^{(j)}(x)\) as \(n \to \infty\), 
	and there exists a constant \(M<\infty\) such that 
	\(\sup_{n} \|\varphi_n\|_{C^2} \le M\), where 
	\(\|\varphi_n\|_{C^2} := \|\varphi_n\|_{\infty} + \|\nabla\varphi_n\|_{\infty} + \|\nabla^2\varphi_n\|_{\infty}\).}. 
The proof is deferred to Lemma~\ref{lemma:separable}. 

		   Let us next prove that, for each \((u, v, z) \in U \times \mathbb{V} \times \mathbb{R}^d\),  
		   the operators  
		   \[
		   L_{u,v}(\zeta \gamma \varphi)(r,s,x) := \mathcal{L}_0(\zeta \gamma \varphi)(r,s,x,u,v),
		   \quad
		   A_{z}(\zeta \gamma \varphi)(r,s,x) := \mathcal{A}_0(\zeta \gamma \varphi)(r,s,x,z),
		   \]
		   are pre-generators. 
		    We first treat the operator \(A_{z}\).  For \(i = (r, s, x) \in \mathbb{R}_{+} \times \mathbb{R}_{+} \times \mathbb{R}^d\), define the processes
		   \[
		   \mathcal{R}^i_t := r, \qquad S^i_t := s, \qquad \text{and} \qquad X^i_t := x + zt.
		   \]
		   For each \(t \geq 0\) and \(i \in \mathbb{R}_{+} \times \mathbb{R}_{+} \times \mathbb{R}^d\), define the measures
		   $$
		   \nu_t^i(B \times C \times D)=\delta_{\mathcal{R}_t^i}(B) \delta_{S_t^i}(C) \delta_{X_t^i}(D), \quad B \in \mathcal{B}\left(\mathbb{R}_{+}\right), \quad C \in \mathcal{B}\left(\mathbb{R}_{+}\right), \quad D \in \mathcal{B}(\mathbb{R}^d).
		   $$
		Because  \(\nu^i = \left(\nu_t^i\right)_{t \geq 0}\) solves the forward equation associated with \((A_{z}, \delta_i)\) and is right-continuous at zero due to time-continuity of each process, it follows from Proposition \ref{prop:B2} that \(A_{z}\) is a pre-generator.
		   
For the operator $L_{u,v}$, suppose first that $u = 1$. Then
\[
L_{1,v}(\zeta \gamma \varphi)(r,s,x)
= \zeta(r) \bigg( \gamma^{\prime}(s) \varphi(x) 
+ \gamma(s) \Big[ v_1 \cdot \nabla_x \varphi(x) 
+ \frac{1}{2} \mathrm{tr}\!\big(v_2 \nabla^2_{xx} \varphi(x)\big) 
+  v_3\int_{\mathbb{R}^d} \big( \varphi(y) - \varphi(x) \big) \, v_4(dy) \Big] \bigg).
\]
By  \cite[Proposition~10.2 (p.~256)]{ethier_markov_1986}, for any initial probability distribution 
$\eta$ on $\mathbb{R}_{+} \times \mathbb{R}_{+} \times \mathbb{R}^d$,  
the martingale problem for $(L_{1,v}, \eta)$ with c\`adl\`ag paths admits a solution,  
provided that for any such $\eta$ there exists a solution to the martingale problem for 
$(\tilde{L}, \eta)$ with c\`adl\`ag paths, where
\[
\tilde{L}(\zeta \gamma \varphi)(r,s,x)
= \zeta(r) \Big( \gamma^{\prime}(s) \varphi(x) 
+ \gamma(s) \big[ v_1 \cdot \nabla_x \varphi(x) 
+ \frac{1}{2} \mathrm{tr}\!\big(v_2 \nabla^2_{xx} \varphi(x)\big) \big] \Big).
\]
It is clear that a solution to the martingale problem for $(\tilde{L}, \eta)$ with c\`adl\`ag paths exists.  
Indeed, let $(\Omega, \mathcal{F}, \mathbb{P})$ be a probability space supporting a standard $d$-dimensional Brownian motion $(W_t)_{t \geq 0}$ and an initial random vector $(\mathcal{R}_0, S_0, X_0)$ with law $\eta$, independent of $W$. Define the process $Y_t := (\mathcal{R}_t, S_t, X_t)$ as the solution to the SDE system
\begin{equation*}
    \begin{cases}
	d\mathcal{R}_t = 0, \
	dS_t = dt, \\
	dX_t = v_1 \, dt + v_2^{1/2} \, dW_t,
\end{cases}
\end{equation*}
with initial condition $Y_0 = (\mathcal{R}_0, S_0, X_0)$, where $v_2^{1/2}$ is the principal square root of $v_2$, i.e., $v_2^{1/2} (v_2^{1/2})^{\top} = v_2$.  By construction, $\mathbb{P} \circ Y^{-1}$ solves the martingale problem for $(\tilde{L}, \eta)$.  
Consequently, the martingale problem for $(L_{1,v}, \eta)$ with c\`adl\`ag paths admits a solution.  
This ensures the existence of a solution to the forward equation for $(L_{1,v}, \delta_i)$ that is right-continuous at zero, for any $i \in \mathbb{R}_{+} \times \mathbb{R}_{+} \times \mathbb{R}^d$, which by Proposition \ref{prop:B2} implies that $L_{1,v}$ is a pre-generator. By a similar argument, one can also show that $L_{0,v}$ is a pre-generator.

		  Finally, the set \(\mathcal{D}\) is closed under multiplication and separates points by working with bump functions.

			\textit{Step-3: Characterization of the  controlled martingale problem}. By  \cite[Corollary 1.10]{kurtz2017linearprogrammingformulationssingular}, there exist a complete filtered probability space \((\Omega, \mathcal{\tilde{F}}, \tilde{\Fb}, \mathbb{Q})\), a $\tilde{\Fb}$-progressive  stationary \(\mathbb{R}_+ \times \mathbb{R}_+ \times \mathbb{R}^d\)-valued process \((\mathcal{R}, S, Y)\), and a $\tilde{\Fb}$-progressive random measure \(\overline{\Gamma}\) having stationary increments such that, for all \(\zeta \gamma \varphi \in \mathcal{D}\), the process 
			\begin{align*}
					\zeta\left(\mathcal{R}_t\right) \gamma\left(S_t\right) \varphi\left(Y_t\right)&-\int_0^t \int_U \overline{\mathcal{L}}(\zeta \gamma \varphi)\left(\mathcal{R}_s, S_s, Y_{s}, u\right) \bar{\eta}\left(\mathcal{R}_s, S_s, Y_{s}, d u\right) d s\\&-\int_{\Rb_+\times\Rb_{+}\times\Rb^d\times[0,t]}\overline{\mathcal{A}}(\zeta \gamma \varphi)(r,s,x)\overline{\Gamma}(dr,ds,dx,dv), \  \ t\in [0, T] 
			\end{align*}
		is an \(\tilde{\Fb}\)-martingale. Moreover, it follows from the properties of \(\overline{\Gamma}\) established in   \cite[Theorem~2.2, Lemma~2.4, and Remark~2.5]{kurtz_stationary_2001} that \(\overline{\Gamma}\) is continuous. In particular, by  \cite[Theorem~2.2(d)]{kurtz_stationary_2001}, \(\overline{\Gamma}\) admits the representation
		\[
		\overline{\Gamma}\bigl(H \times [0,t]\bigr)
		= \int_{0}^{\gamma_0^{-1}(t)} \mathbf{1}_{H}\bigl(Z(s)\bigr)\,d\gamma_1(s),
		\]
		where \(\gamma_0\) and \(\gamma_1\) are continuous increasing processes, and \(Z\) is an \(\mathbb{R}_{+} \times \mathbb{R}_{+} \times \mathbb{R}^d\)-valued process.  
		To prove the continuity of \(\overline{\Gamma}\), it suffices to show that \(\gamma_0\) is strictly increasing. We shall use   \cite[Lemma~2.4]{kurtz_stationary_2001} for this purpose, which entails showing that every solution \(\bigl((\overline{\mathcal{R}},\overline{S},\overline{Y}),\xi\bigr)\) of the stopped martingale problem for \(\overline{\mathcal{A}}\) satisfies
		\[
		\xi \wedge \inf \bigl\{ t : (\overline{\mathcal{R}}_t,\overline{S}_t,\overline{Y}_t) \notin K_{1} \bigr\} = 0
		\quad\text{a.s.},
		\]
		where \(K_{1} := \{0\} \times [0,T] \times \partial\Xb\).
		
		Let \(\bigl((\overline{\mathcal{R}},\overline{S},\overline{Y}),\xi\bigr)\) be a solution of the stopped martingale problem for \(\overline{\mathcal{A}}\). That is, there exists a filtration \(\{\tilde{\mathcal{F}}_t\}\) such that \((\overline{\mathcal{R}},\overline{S},\overline{Y})\) is an  \(\{\tilde{\mathcal{F}}_t\}\)-adapted process, \(\xi\) is a \(\{\tilde{\mathcal{F}}_t\}\)-stopping time, and for each \(\zeta\gamma\phi \in \mathcal{D}\), the process
		\[
		\zeta(\overline{\mathcal{R}}_{\cdot\wedge\xi})\,\gamma(\overline{S}_{\cdot\wedge\xi})\,\phi(\overline{Y}_{\cdot \wedge\xi})
		- \zeta(\overline{\mathcal{R}}_{0})\,\gamma(\overline{S}_{0})\,\phi(\overline{Y}_{0})
		- \int_{0}^{\cdot \wedge\xi} \overline{\mathcal{A}}(\zeta\gamma\phi)(\overline{\mathcal{R}}_s,\overline{S}_s,\overline{Y}_s)\,ds
		\]
		is an \(\{\tilde{\mathcal{F}}_t\}\)-martingale. Without loss of generality, we assume \((\overline{\mathcal{R}}_0,\overline{S}_0,\overline{Y}_0) \in K_1\) with positive probability. Choosing \(\zeta = \gamma = 1\), we obtain that the process 
		\[
		\phi(\overline{Y}_{\cdot \wedge\xi}) - \phi(\overline{Y}_{0})
		- \int_{0}^{\cdot \wedge\xi} \,m(\overline{Y}_s) \cdot \nabla_x\phi(\overline{Y}_s)\,ds
		\]
		is an \(\{\tilde{\mathcal{F}}_t\}\)-martingale. Following the discussion in  \cite[Remark~2.5]{kurtz_stationary_2001}, we deduce that \(\overline{Y}\) satisfies
		\[
		\overline{Y}_{t\wedge\xi} 
		= \overline{Y}_{0} + \int_0^{t\wedge\xi} \,m(\overline{Y}_s)\,ds.
		\]
		Note  that \(m\) is the inward normal vector, it is clear that for any \(\omega\) such that \(\overline{Y}_{0}(\omega) \in \partial\Xb\) and \(\xi(\omega) > 0\),
		\[
		\xi(\omega) \wedge \inf \{t : (\overline{\mathcal{R}}_t(\omega),\overline{S}_t(\omega),\overline{Y}_t(\omega)) \notin K_{1} \}
		\;\leq\; \xi(\omega) \wedge \inf \{ t : \overline{Y}_t(\omega) \notin \partial\Xb \}
		= \inf \{ t : \overline{Y}_t(\omega) \notin \partial\Xb \}
		= 0,
		\]
		which completes the proof.

			Following the same proof as  \cite[Theorem 2.4]{helmes_linear_2007}, we can conclude the existence of a complete filtered probability space $(\Omega, \mathcal{F}, \mathbb{F}, \mathbb{P})$, where $\mathbb{F}$ satisfies the usual conditions, an $\mathbb{F}$-stopping time $\tau$ with values in $\mathbb{R}_{+}$, a process $\tilde{\mathcal{R}}$ with values in $\mathbb{R}_{+}$,  a process $\tilde{S}$ with values in $\mathbb{R}_{+}$ such that $\tilde{S}_t \indicator{t \leq \tau}=t \indicator{t \leq \tau}$, an $\mathbb{F}$-progressively measurable c\`adl\`ag process  $X$ with values in $\mathbb{R}^d$ such that $\mathbb{P} \circ X_0^{-1}=\tilde{m}_0$ and a  continuous  $\mathbb{F}$-progressively measurable  random measure $\tilde{\Gamma}$ on $\Rb_{+}\times\Rb_{+}\times\Rb^d\times\Rb_{+}$. Furthermore, it holds that
			\begin{align}\label{repre1}
				&\tilde{\nu}=\mathbb{P} \circ\left(\tau, X_\tau\right)^{-1}, \quad \tilde{m}(G)=\mathbb{E}^{\mathbb{P}}\left[\int_0^\tau \indicator{G}\left(t, X_t\right) d t\right], \quad \forall G \in \mathcal{B}\left(\mathbb{R}_{+} \times \mathbb{R}^d\right),\nonumber\\
				&\tilde{\lambda}(G)=\mathbb{E}^{\mathbb{P}}\left[\int_{\Rb_{+}\times\Rb_{+}\times\Rb^d\times [0,\tau]} \indicator{G}\left(s, x\right)  \tilde{\Gamma}(dr,ds,dx,dv)\right], \quad  \ \forall G \in \mathcal{B}\left(\mathbb{R}_{+} \times \mathbb{R}^d\right),
			\end{align}
			and the process 
			\begin{align}\label{proof:martingale1}
					\gamma\left(\tilde{S}_{\cdot \wedge \tau}\right) \varphi\left(X_{\cdot \wedge \tau}\right)-\int_0^{\cdot \wedge \tau} \hat{\mathcal{L}}(\gamma \varphi)\left(\tilde{S}_s, X_s\right) d s-\int_{\Rb_{+}\times\Rb_{+}\times\Rb^d\times[0,\cdot \wedge\tau]}\hat{\Ac}(\gamma \varphi)(s,x)\tilde{\Gamma}(dr,ds,dx,dv)
			\end{align}
			is an $\Fb$-martingale for all $\gamma \in C_b^1\left(\mathbb{R}_{+}\right), \varphi \in C_b^2(\mathbb{R}^d)$.
			Here, we have the same construction of \((\tilde{R},\tilde{S},X,\tilde{\Gamma})\)  as in \cite{helmes_linear_2007} that 
			\begin{align*}
			&\tilde{\mathcal{R}}(s) = \mathcal{R}\bigl(\sigma^0_1 + s\bigr),
			\quad
			\tilde{S}(s) = S\bigl(\sigma^0_1 + s\bigr),
			\quad
			X(s) = Y\bigl(\sigma^0_1 + s\bigr),\\
			&\tilde{\Gamma}\bigl(H\times[s,t]\bigr)
			\;=\;
			\overline{\Gamma}\bigl(H\times [\sigma^0_1 + s,\;\sigma^0_1 + t]\bigr), \quad H\in\Bc(\mathbb{R}_+\times\mathbb{R}_{+}\times\mathbb{R}^d),\,\,\ 0 \le s \le t,
			\end{align*}
		 where	$\sigma^0_1 \;=\;\inf\{\,s\ge0: \mathcal{R}(s) = 0\}$. Then, by the representation of \(\overline{\Gamma}\) in  \cite[Theorem 2.2 (f)]{kurtz_stationary_2001}, we obtain
		 \begin{align}\label{representation_skorokhod}
		 \tilde{\Gamma}\bigl(H\times[0,t]\bigr)=	\overline{\Gamma}\bigl(H\times [\sigma^0_1,\sigma^0_1 + t]\bigr)&=\int_{\sigma^0_1}^{\sigma^0_1 + t}\indicator{H}(\mathcal{R}_s,S_s,Y_s)d\bar{\lambda}(s)\nonumber\\
		 	&=\int_{0}^{t}\indicator{H}(\tilde{\mathcal{R}}_s,\tilde{S}_s,X_s)d\tilde{\lambda}(s),
		 \end{align}
	where \(\bar{\lambda}(s) := (\gamma_1 \circ \gamma_0^{-1})(s)\) and \(\tilde{\lambda}(s) := \bar{\lambda}(\sigma^0_1 + s)\) are continuous increasing processes.    A key property inherited from \cite[Theorem~2.2 (e)–(f)]{kurtz_stationary_2001} is that the random measure \(d\bar{\lambda}\) is supported on the set of times 
\(\{s : (\mathcal{R}_s, S_s, Y_s) \in K_1\}\). 
Consequently, the construction of \((\tilde{\mathcal{R}}, \tilde{S}, X, \tilde{\lambda})\) as time-shifted processes guarantees that this property is preserved; namely,
\(d\tilde{\lambda}\) is supported on the set 
\(\{s : (\tilde{\mathcal{R}}_s, \tilde{S}_s, X_s) \in K_1\}\).

			Note that
			\[
			1 = m_0^*(\overline{\mathbb{X}}) = \tilde{m}_0(\overline{\mathbb{X}}) = \mathbb{P}\bigl(X_0 \in \overline{\mathbb{X}}\bigr),
			\]
			which implies  \(X_0 \in \overline{\mathbb{X}}\)\, ,  \(\mathbb{P}\)-a.s., and therefore
			$m_0^* = \mathbb{P} \circ X_0^{-1}$.
			On the other hand, in view that
			$$
			1=\nu(\overline{\mathbb{X}})=\tilde{\nu}(\{T\} \times \overline{\mathbb{X}})=\mathbb{P}\left(\tau=T, X_\tau \in \overline{\mathbb{X}}\right)=\mathbb{P}\left(\tau=T, X_T \in \overline{\mathbb{X}}\right),
			$$
			we conclude $\tau= T, X_T \in \overline{\mathbb{X}}$, \  $ \mathbb{P}$-a.s. and
			\begin{equation*}\label{nu}
	       \nu=\mathbb{P} \circ X_T^{-1}.
			\end{equation*}
			 Also note that
			$$
			\int_{G} m_t^{\Xb}(d x) d t=\tilde{m}(G)=\mathbb{E}^{\mathbb{P}}\left[\int_0^T\indicator{G}\left(t, X_t\right) d t\right]=\mathbb{E}^{\mathbb{P}}\left[\int_0^T\indicator{G}\left(t, X_{t-}\right) d t\right], \quad \forall G \in \mathcal{B}([0, T] \times \overline{\mathbb{X}}),
			$$
			where we have used the fact that $X_t(\omega) = X_{t-}(\omega)$ for a.e. $t \in [0,T]$ for any fixed $\omega \in \Omega$ due to that $X$ is a c\`adl\`ag process.
			Then, for $B \in \mathcal{B}([0, T]), C \in \mathcal{B}(\overline{\mathbb{X}}), D \in \mathcal{B}(\Ab)$,
			\begin{align*}
					\int_B \int_{C \times D} m_t(d x, d a) d t=\int_{B \times C} v_{t, x}(D) m^{\Xb}_t(d x) d t&=\int_B \mathbb{E}^{\mathbb{P}}\left[\indicator{C}\left(X_t\right) v_{t, X_t}(D)\right] d t\\&=\int_B \mathbb{E}^{\mathbb{P}}\left[\indicator{C}\left(X_{t-}\right) v_{t, X_{t-}}(D)\right] d t,
			\end{align*}
			which implies
			\begin{equation*}\label{mt}
				m_t(B \times C)=\mathbb{E}^{\mathbb{P}}\left[\indicator{B}\left(X_t\right) v_{t, X_t}(C)\right]=\mathbb{E}^{\mathbb{P}}\left[\indicator{B}\left(X_{t-}\right) v_{t, X_{t-}}(C)\right], \  B \in \mathcal{B}(\overline{\mathbb{X}}),\  C \in \mathcal{B}(\Ab),\ t-a.e..
			\end{equation*}
			By the definition of $\tilde{m}$, we have
			$$
			0=\tilde{m}\left(\mathbb{R}_{+} \times \overline{\mathbb{X}}^c\right)=\mathbb{E}^{\mathbb{P}}\left[\int_0^T \indicator{\overline{\mathbb{X}}^c}\left(X_t\right) d t\right],
			$$
			which implies 
			\begin{equation}\label{C2}
          (\mathbb{P} \otimes \lambda)\left(\left\{(\omega, t) \in \Omega \times[0, T]: X_t(\omega) \in \overline{\mathbb{X}}^c\right\}\right)=0 .
			\end{equation}
           It follows from \eqref{repre1} and \eqref{representation_skorokhod} that 
			\begin{align*}
				\lambda(G)=\tilde{\lambda}(G)&=\mathbb{E}^{\mathbb{P}}\left[\int_{\Rb_{+}\times\Rb_{+}\times\Rb^d\times [0,T]} \indicator{G}\left(s, x\right)  \tilde{\Gamma}(dr,ds,dx,dv)\right] \\ &=\mathbb{E}^{\mathbb{P}}\left[\int_{0}^T \indicator{G}\left(\tilde{S}_s,X_s\right)  d\tilde{\lambda}(s)\right],
                \quad \forall G \in \mathcal{B}([0, T] \times \partial\Xb).
			\end{align*}
			Using the fact that $\tilde{S}_s \indicator{s \leq T}=s \indicator{s \leq T}$, we deduce  that 
            \begin{equation*}
                \lambda(G)=\mathbb{E}^{\mathbb{P}}\left[\int_{0}^T \indicator{G}\left(s,X_s\right)  d\tilde{\lambda}(s)\right],\quad \forall G \in \mathcal{B}([0, T] \times \partial\Xb),
            \end{equation*}
            and $\tilde{\lambda}$ increases only when $X$ in $\partial\Xb$.
            Then,  thanks to \eqref{proof:martingale1}, we obtain that the process 
				\begin{align*}
				\gamma\left(\cdot \wedge T\right) \varphi\left(X_{\cdot \wedge T}\right)-\int_0^{\cdot \wedge T} \hat{\mathcal{L}}(\gamma \varphi)\left(s, X_s\right) d s-\int_0^{\cdot \wedge T}\hat{\Ac}(\gamma \varphi)(s,X_s)d\tilde{\lambda}(s)
			\end{align*}
			is an $\mathbb{F}$-martingale for all $\gamma \in C_b^1\left(\mathbb{R}_{+}\right), \varphi \in C_b^2(\mathbb{R}^d)$. Moreover, setting \(\gamma \equiv 1\), 
			we draw the conclusion that, for every \(\varphi \in C_b^2(\mathbb{R}^d)\), 
			$$
			\varphi\left(X_{\cdot \wedge T}\right)-\int_0^{\cdot  \wedge T} \hat{\mathcal{L}}(\varphi)\left(s, X_s\right) d s-\int_0^{\cdot  \wedge T}\hat{\Ac}(\varphi)(X_s)d\tilde{\lambda}(s)
			$$
			is an $\mathbb{F}$-martingale. Equivalently, for all \(\varphi \in C_b^2(\mathbb{R}^d)\) and $t\in[0,T]$,
		{\small\begin{align*}
			\varphi\left(X_{t}\right)-\varphi\left(X_0\right)&-\int_0^{t} \int_{\Ab}\left(b\left(s, X_s,\mu_{s}, a\right)\cdot \nabla_x\varphi\left(X_s\right)+\frac{1}{2}\mathrm{tr}\left(\sigma\sigma^{\top}(s, X_s,\mu_s, a) \nabla^2_{xx}\varphi\left(X_s\right)\right)\right) v_{s, X_s}(d a) d s\nonumber\\&-\int_{0}^{t}\int_{\Ab}\pi(s,a)\left[\varphi(\proj(X_{s}+\beta(s,X_{s},\mu_s,a)))-\varphi(X_{s})-\beta\left(s, X_{s},\mu_{s}, a\right)\cdot \nabla_x\varphi\left(X_{s}\right)\right]v_{s, X_s}(d a) d s\nonumber\\&-\int_{0}^{t}m(X_s)\cdot\nabla_x\varphi(X_s)dR^c_s
		\end{align*}}
		is an \(\mathbb{F}\)-martingale, where
		\[
		R^c_t := \int_0^t d\tilde{\lambda}(s).
		\]
		Clearly, $R^c$ is nondecreasing, continuous, and satisfies $R^c_0 = 0$, with the property
		\begin{equation*}\label{reflection1}
		    R^c_t = \int_0^t \indicator{\partial \mathbb{X}}(X_s) \, dR^c_s.
		\end{equation*}
		Finally, as $X_t(\omega) = X_{t-}(\omega)$ for a.e. $t \in [0,T]$ for any fixed $\omega \in \Omega$,  
		we can rewrite the martingale property as: for all $\varphi \in C_b^2(\mathbb{R}^d)$ , the process 
	\begin{align}\label{martingale_jump_reflected}
			\varphi(X_t) - \varphi(X_0)
		- \int_0^t \mathcal{L}\varphi(s, X_{s-}, \mu_s, a) \, v_{s, X_{s-}}(da) \, ds
		- \int_0^t m(X_{s}) \cdot \nabla_x \varphi(X_{s}) \, dR^c_s, \ \ t \in [0,T]
		\end{align}
		is an $\mathbb{F}$-martingale.

		\textit{Step-4: SDE representation of the controlled martingale problem}. 
			As $\Pb$ solves the martingale problem \eqref{martingale_jump_reflected}, it follows from arguments analogous to those in \cite[Lemma~2.1]{benazzoli_mean_2020}
          that  on an extension of the filtered probability space \(\bigl(\Omega,\mathcal{F},\mathbb{F},\mathbb{P}\bigr)\), denoted by \(\bigl(\Omega',\mathcal{F}',\mathbb{F}',\mathbb{P}'\bigr)\), there exist $d$-dimensional  orthogonal $\Fb'$-adapted continuous martingale measures \(M=(M^1,\cdots,M^d)\)  with intensity \(v_{t,X_{t-}}(da)\,dt\),   a random counting measure \(\mathcal{N}\) on \([0,T]\times\mathbb{A}\) with compensator \(\pi(t,a)v_{t,X_{t-}}(da)\,dt\) such that
       \begin{align}\label{eq:sde_intermediate}
       		 X_t&= X_0 + \int_0^t\int_{\Ab} b\left(s, X_s,\mu_s, a\right) v_{s, X_s}(d a) d s ,+\int_{0}^{t}m(X_s)dR^c_s\\&\quad\quad\,\,\,\,+\int_0^t\int_{\Ab}\pi(s,a)[\proj(X_{s-}+\beta(s,X_{s-},\mu_s,a))-(X_{s-}+\beta(s,X_{s-},\mu_s,a))] v_{s,X_{s-}}(da)ds\nonumber\\&\quad\quad\,\,\,\,+\int_0^t\int_{\Ab} \sigma\left(s, X_s,\mu_s, a\right) M(d s, d a)+\int_0^t\int_{\Ab} \bigg(\proj(X_{s-}+\beta(s,X_{s-},\mu_s,a))-X_{s-}\bigg)\tilde{\mathcal{N}}(ds,da).\nonumber
       \end{align}
      Here, $\tilde{\mathcal{N}}$ denotes the compensated random measure associated with $\mathcal{N}$
      and the drift term on the second line arises from the discrepancy between the compensator of the true jump, $\proj(X_{s-}+\beta(\cdot)) - X_{s-}$, and the jump term $\beta(\cdot)$ specified in the operator $\Lc$.
      
    Rearranging the jump and compensated terms, we can rewrite \eqref{eq:sde_intermediate} as:
      \begin{align*}
        X_t &=X_0 + \int_0^t\int_{\Ab} b\left(s, X_s,\mu_s, a\right) v_{s, X_s}(d a) d s+\int_0^t\int_{\Ab} \sigma\left(s, X_s,\mu_s, a\right) M(d s, d a)\\&\quad\quad\,\,\,\,+\int_0^t\int_{\Ab}\beta(s,X_{s-},\mu_s,a)\tilde{\mathcal{N}}(ds,da)+\int_{0}^{t}m(X_s)dR^c_s\\&\quad\quad\,\,\,\,+\int_{0}^{t}\int_{\Ab} \left[\proj(X_{s-}+\beta(s,X_{s-},\mu_s,a))-(X_{s-}+\beta(s,X_{s-},\mu_s,a))\right]  \mathcal{N}(ds,da),\,\ t\in [0,T].
    \end{align*}
      Let us define the jump component of the reflection process by
      \begin{align*}
          R^d_{t}=\int_{0}^{t}\int_{\Ab} \bigg|\proj(X_{s-}+\beta(s,X_{s-},\mu_s,a))-(X_{s-}+\beta(s,X_{s-},\mu_s,a))\bigg| \mathcal{N}(ds,da).
      \end{align*}
    Then, letting $R_t := R^c_t + R^d_t$, and applying the identity for the projection operator
 \begin{align*}
\proj(X_{s-}+\beta(s,X_{s-},\mu_s,a))&=X_{s-}+\beta(s,X_{s-},\mu_s,a)\\&+|\proj(X_{s-}+\beta(s,X_{s-},\mu_s,a))-(X_{s-}+\beta(s,X_{s-},\mu_s,a))| m(X_s),
 \end{align*}
   we get the reflected SDE for the process $X$ that  \begin{align}\label{representation}
       X_t &=X_0 + \int_0^t\int_{\Ab} b\left(s, X_s,\mu_s, a\right) v_{s, X_s}(d a) d s+\int_0^t\int_{\Ab} \sigma\left(s, X_s,\mu_s, a\right) M(d s, d a)\nonumber\\&\quad\quad\,\,\,\,+\int_0^t\int_{\Ab}\beta(s,X_{s-},\mu_s,a)\tilde{\mathcal{N}}(ds,da)+\int_{0}^{t}m(X_s)dR_s,\,\ t\in [0,T].
  \end{align}
  Condition \eqref{C2}, together with the SDE representation \eqref{representation}, entails that $X$ takes values in $\overline{\mathbb{X}}$. More precisely, condition~\eqref{C2} ensures that the trajectory remains in $\overline{\mathbb{X}}$ for almost every $t\in[0,T]$, thereby implying that the continuous points lie in $\overline{\mathbb{X}}$. Meanwhile, condition~\eqref{representation} guarantees that post-jump states lie in $\overline{\mathbb{X}}$, thus preventing exits at jump times. Hence, the RCLL paths of $X$ remain in $\overline{\mathbb{X}}$ for all $t\in[0,T]$.
  Moreover, it is clear that \(R\) is non-decreasing, RCLL, and satisfies $R_t = \int_0^t \mathbf{1}_{\{\partial \mathbb{X}\}}(X_s) \, dR_s.$
\end{proof}
		\begin{theorem}[Equivalence between LPMFEs  and  MFEs with relaxed control  ]\label{theorem:equi}
			 Assume the conditions of Proposition~\ref{thm:prob-rep-R0} hold. Then, the LP MFG problem and the MFG problem with relaxed control are equivalent. More specifically,
			 \begin{itemize}
			 	\item[(i)] Given an LPMFE $(\nu^{\star}, m^{\star}, \lambda^{\star})$ with $m^{*,\Xb}\in B_2$, there exists an  MFE  with Markovian relaxed control $U^{\star} \in \mathcal{R}(m^{*,\Xb})$ such that
			 	\begin{equation}\label{equivalent}
             \mathcal{J}^{L}_{P}[m^{*,\Xb},\nu^*]\left(\nu^{\star}, m^{\star},\lambda^{\star}\right)= \mathcal{J}^{R}[m^{*,\Xb},\nu^*]\left(U^{\star}\right).
			 	\end{equation}
			 
			 	\item [(ii)] Conversely, let $U^{\star}$ be an MFE with relaxed control such that the mean field flow $\mu^* \in B_2$, where $\mu_t^* := \Pb \circ (X_t)^{-1}$ for a.e. $t \in [0,T]$. Define
			 	\begin{align*}
			 	& m_t^{\star}(B\times C)=\mathbb{E}^{\mathbb{P}}\left[\mathbf{1}_B\left(X_t\right) \Lambda_t(C) \right], \quad B \in \mathcal{B}(\overline{\Xb}),\, C\in\Bc(\Ab),\, t-a.e, \\
			 	&\nu^{\star}=\mathbb{P} \circ\left(X_{T}\right)^{-1}, \lambda^{\star}(D)=\mathbb{E}^{\Pb}\Bigl[\int_0^T \mathbf{1}_{D}(t,X_t)\,dR^c_t\Bigr],
			 	\quad 
			 	D \in \mathcal{B}([0,T]\times\partial\Xb).
			 	\end{align*}
			 	Then \((\nu^{\star}, m^{\star}, \lambda^{\star})\) is an LPMFE  and \eqref{equivalent} holds.
			 \end{itemize}
		\end{theorem}
\begin{proof}
By similar arguments as those in Remarks \ref{remark:jump} and \ref{remark:s_to_r}, for any given \( U \in \mathcal{R}(\mu) \) with $\mu\in B_2$, there exists a triple \( (\nu, m, \lambda) \in \mathcal{D}_{P}(\mu) \) such that 
	\[
  \mathcal{J}^{R}[\mu,\rho]\left(U\right)= \mathcal{J}^{L}_{P}[\mu,\rho]\left(\nu, m,\lambda\right).
	\]
Conversely, Proposition~\ref{thm:prob-rep-R0} ensures that any $(\nu, m, \lambda) \in \Dc_{P}(\mu)$ with $\mu \in B_2$ also induces an element of $\mathcal{R}(\mu)$. Indeed, it only remains to verify the integrability condition \eqref{integ2}. This condition holds, as the inclusion $\Dc_{P}(\mu) \subset \Pc_2(\overline{\Xb}) \times V_2 \times \Mc_{+}([0,T]\times\partial\Xb)$ implies that
\begin{align*}
\Eb^{\Pb}&\Bigg[\int_0^T\!\!\int_{\Ab}(|X_t|^2+|a|^2)\,v_{t,X_t}(da)\,dt
+ |X_T|^2 + R^c_T\Bigg]
\\&= \int_{[0,T]\times\overline{\Xb}\times\Ab} (|x|^2+|a|^2)\,m_t(dx,da)\,dt+ \int_{\overline{\Xb}} |x|^2\,\nu(dx)
+ \lambda([0,T]\times\partial\Xb)
< \infty.
\end{align*}
The equivalence of the two formulations thus follows.
\end{proof}

\section{Existence Result}\label{sect:4}
The next theorem is the main result of this paper, concerning the existence of an LPMFE. 
\begin{theorem}\label{thm:existence}
    Under Assumption \ref{c1}, there exists an LPMFE. Consequently, there also exists an MFE with a Markovian relaxed control.
\end{theorem}

The proof of the above main theorem can be outlined as two main steps: In the first step, we consider the model with bounded coefficients and compact action space to facilitate the proof of the existence result. In the second step under more general model conditions as stated in Assumption \ref{c1}, we utilize the existence result in the first step and an approximation argument to conclude the existence of an LPMFE.

	\subsection{Existence of LPMFEs in the bounded case}
In this subsection, we first establish the existence of LPMFEs under the additional assumptions of boundedness of the coefficients and compactness of the action space $\Ab$. 
\begin{assumption}\label{assumption:2}
The functions $b$, $\sigma$ and $\beta$ are bounded, and the control space $\Ab$ is compact.
\end{assumption}
\begin{remark}\label{remark:nonempty}
Under Assumptions~\ref{c1} and~\ref{assumption:2}, the set $\Dc_{P}(\mu)$ is nonempty for any measurable flow $\mu\in M([0,T],\Pc(\overline{\Xb}))$. 
Indeed, by Remark~\ref{remark:s_to_r}, it suffices to verify that, under a constant control $a_0 \in \Ab$, there exists a strong solution to \eqref{sde_strong}–\eqref{Skorokhod1} satisfying the integrability condition~\eqref{integ1}. 
The existence of such a solution follows from the square integrability of $m^*_0$ together with the Lipschitz continuity and boundedness of the coefficients.
The integrability condition \eqref{integ1} then follows from standard estimates.
\end{remark}
\begin{definition}\label{def:D0}
	We define $\mathcal{D}_0$ as the set of all triples $(\nu, m,\lambda) \in {\Pc_2(\overline{\mathbb{X}})\times	V_2\times \Mc_{+}([0,T]\times\partial\Xb)}$ for which there exists a constant $C>0$, depending only on $m^*_0$, $q$, $\|b\|_{\infty}$,  $\|\sigma\|_{\infty}$, $\|\beta\|_{\infty}$, and $\|\pi\|_{\infty}$, such that the following conditions hold: 
    \begin{itemize}
        \item [(i)] The triple $(\nu, m,\lambda)$ satisfies the following uniform integrability and boundedness conditions:
 \[
        \int_{\overline{\mathbb{X}}}|x|^q \, \nu(\mathrm{d}x) \le C, \quad  \int_{\overline{\mathbb{X}}\times\mathbb{A}} |x|^q m_t(\mathrm{d}x,\mathrm{d}a) \le C,\,t-a.e.,\quad \text{and} \quad \lambda([0,T]\times\partial\mathbb{X}) \le C,
    \]
     where $q$ is the constant specified in  Assumption~\ref{c1}.
        \item [(ii)] For all $h\in[0,T]$, we have  \[
        \int_0^{T-h} d_{\mathrm{BL}}(m^{\mathbb{X}}_{t+h}, m^{\mathbb{X}}_{t}) \, \mathrm{d}t \le C\sqrt{h},
    \]
    where $m^{\mathbb{X}}_t$ denotes the  marginal distribution of  $m_t$ on $\overline{\Xb}$ and $d_{\mathrm{BL}}$ is the bounded Lipschitz metric defined for $\mu^1, \mu^2 \in \mathcal{M}_{+}(\overline{\mathbb{X}})$ by
    \[
       d_{\mathrm{BL}}(\mu^1,\mu^2) := \sup\Bigg\{ \left| \int_{\overline{\mathbb{X}}} \varphi(x) \, (\mu^1-\mu^2)(\mathrm{d}x) \right| : \varphi \in \mathrm{BL}(\overline{\mathbb{X}}),\, \|\varphi\|_{\mathrm{BL}} \le 1 \Bigg\}.
    \]
    \end{itemize}
    \end{definition}
	\begin{lemma} \label{lemma:D0}
    {Suppose that Assumptions~\ref{c1} and  \ref{assumption:2} hold}.
		\begin{itemize}
			 \item[(i)] The set $\mathcal{D}_0$ is nonempty, convex and compact under the topology $\tau_2 \otimes \overline{\tau}_2 \otimes \tau_0$. Moreover, it contains the set $\Dc_{P}(\mu)$ for any measurable flow $\mu\in M([0,T],\Pc(\overline{\Xb}))$.
			
			\item [(ii)] The set $\Dc_{P}(\mu)$ is  convex and compact  under the topology $\tau_2 \otimes \overline{\tau}_2 \otimes \tau_0$.
		\end{itemize}
	\end{lemma}
	\begin{proof}
    (i)  It is straightforward to verify that $\mathcal{D}_0$ is convex. We now show that $\mathcal{D}_0$ contains the set $\mathcal{D}_{P}(\mu)$ for any measurable flow $\mu$. Because $\mathcal{D}_{P}(\mu)$ is nonempty by Remark~\ref{remark:nonempty}, this implies the nonemptyness of $\mathcal{D}_0$.
    
    For any  measurable flow $\mu$,  let us consider a  triple $(\nu, m,\lambda)\in \mathcal{D}_{P}(\mu)$.
    	We claim that { \(\lambda([0,T] \times \partial\Xb) \leq C\) for some constant \(C = C(\|b\|_{\infty}, \|\sigma\|_{\infty},\|\beta\|_{\infty},\|\pi\|_{\infty}) > 0\).} Indeed,  we choose the test function \(u(t,x) = \phi(x)\), where \(\phi \in C_b^2(\overline{\mathbb{X}})\) satisfies  $\nabla\phi(x) = m(x)$ on \(\partial\mathbb{X}\). The existence of such a function is guaranteed by Lemma~\ref{lemma:boundary_construction} when $\Xb$ is bounded. In the case that $\mathbb{X}=(0,\infty)$, the existence of such a function is trivial.
        Substituting this \(u\) into the linear programming constraint \eqref{lpconstraint} yields
  \begin{align*}
      \lambda([0,T]\times\partial\Xb) 
\;\le\;
\int_{\overline{\mathbb{X}}} |\phi(x)|\,\nu(dx)
\;+\;
\int_{\overline{\mathbb{X}}} |\phi(x)|\,m_0^*(dx)
\;+\;
\int_{[0,T]\times\overline{\mathbb{X}}\times\mathbb{A}}
\bigl|\mathcal{L}\phi(t,x,\mu_t,a)\bigr|\,
m_t(dx,da)dt.
  \end{align*}
Note that \(\nu\) and \(m_0^*\) are probability measures, \(\int_0^T m_t(dx,da)dt=T\), and both \(\phi\) and \(\mathcal{L}\phi\) are bounded, the right-hand side is bounded by a constant \(C=C(\|b\|_{\infty}, \|\sigma\|_{\infty},\|\beta\|_{\infty},\|\pi\|_{\infty}) > 0\).

 For $(\nu, m)$, we consider their probabilistic representation provided by Proposition \ref{thm:prob-rep-R0}\footnote{Note that when the coefficients are bounded, the condition $\mu\in B_2$ in Proposition~\ref{thm:prob-rep-R0} is no longer required.}. Applying Lemma~\ref{lemma:A6} (i) and using the boundedness assumption on the coefficients, we obtain the existence of a constant $C>0$, depending only on $m_0^*$, $q$,  $\|b\|_{\infty}$, $\|\sigma\|_{\infty}$, $\|\beta\|_{\infty}$, and $\|\pi\|_{\infty}$, that
 \begin{align*}
      &\int_{\overline{\mathbb{X}}}|x|^q \, \nu(\mathrm{d}x)=\Eb^{\Pb}[|X_T|^q]\leq \Eb^{\Pb}[\sup_{0\leq t\leq T}|X_t|^q]\leq C,\\  
      &\int_{\overline{\mathbb{X}}\times\mathbb{A}} |x|^q m_t(\mathrm{d}x,\mathrm{d}a)=\Eb^{\Pb}[|X_t|^q ]\leq  \Eb^{\Pb}[\sup_{0\leq t\leq T}|X_t|^q]\leq C,\, t-a.e..
 \end{align*}
Let us prove the condition (ii). For $\varphi \in \mathrm{BL}(\overline{\mathbb{X}})$ such that $\|\varphi\|_{\mathrm{BL}} \le 1$, $h\in[0,T]$ and for almost every $t\in [0,T-h]$, we have
\begin{align}\label{m_contin}
    \int_{\overline{\Xb}}\varphi(x)(m^{\mathbb{X}}_{t+h}- m^{\mathbb{X}}_{t})(dx)&=\Eb^{\Pb}\left[\varphi(X_{t+h})-\varphi(X_{t})\right]\leq\Eb^{\Pb}\left[|\varphi(X_{t+h})-\varphi(X_{t})|\right]\nonumber\\&\leq \Eb^{\Pb}\left[|X_{t+h}-X_{t}|\right]\leq \Eb^{\Pb}\left[|X_{t+h}-X_{t}|^2\right]^{\frac{1}{2}}.
\end{align}
Applying Lemma \ref{lemma:A6} (ii), we get $ \Eb^{\Pb}\left[|X_{t+h}-X_{t}|^2\right]\leq Ch$ for some constant $C>0$, depending only on $\|b\|_{\infty}$, $\|\sigma\|_{\infty}$, $\|\beta\|_{\infty}$, and $\|\pi\|_{\infty}$. We then deduce that the existence of a constant $C>0$ such that $\int_0^{T-h} d_{\mathrm{BL}}(m^{\mathbb{X}}_{t+h}, m^{\mathbb{X}}_{t}) \, \mathrm{d}t \le C\sqrt{h}$.

It remains to prove the compactness of $\mathcal{D}_0$. Because the  space $\mathcal{P}_2(\overline{\mathbb{X}}) \times V_2 \times \mathcal{M}_{+}([0,T]\times\partial\mathbb{X})$ is metrizable, it suffices to show that $\mathcal{D}_0$ is sequentially compact. Let $\{(\nu^n, m^n, \lambda^n)\}_{n\ge 1}$ be an arbitrary sequence in $\mathcal{D}_0$.  By definition of $\mathcal{D}_0$, there exists a constant $C>0$ independent of $n$ such that
    \begin{align}\label{uniformlyintegrability}
        \sup_{n\ge1}\int_{\overline{\Xb}}|x|^q\nu^n(dx)<C,
    \end{align}
which implies that \(\{\nu^n\}_{n\geq 1}\) is tight (as the map $x\rightarrow|x|^q$ has compact level sets).  By Prokhorov’s theorem, there exists a subsequence (still denoted by \(\{\nu^n\}\)) and a limit point \( \nu \in  \mathcal{P}\bigl(\overline{\Xb}\bigr)\) such that 
		$ \nu^n \rightarrow  \nu$ under the weak convergence topology $\tau_0$.  Furthermore, the uniform integrability implied by \eqref{uniformlyintegrability} ensures that the convergence also holds in $\tau_2$ (cf. \cite[Theorem 7.12]{Villani2003Topics}).   To show that the limit point $\nu$ also satisfies the moment bound, we note that for any $k \ge 1$,
        \begin{equation}\label{argument}
             \int_{\overline{\mathbb{X}}} (|x|^q \wedge k) \, \nu^n(\mathrm{d}x) \le \int_{\overline{\mathbb{X}}} |x|^q \, \nu^n(\mathrm{d}x) \le C.
        \end{equation}
Taking the limit as $n\to\infty$, the weak convergence implies $\int_{\overline{\mathbb{X}}} (|x|^q \wedge k) \, \nu(\mathrm{d}x) \le C$. The desired bound, $\int_{\overline{\mathbb{X}}} |x|^q \, \nu(\mathrm{d}x) \le C$, then follows from the monotone convergence theorem as $k\to\infty$.

Let us show that we can extract a further subsequence of  $\{m^n\}$ that converges to some $m\in V_2$ in  the stable topology $\overline{\tau}_2$. 
Note that 
\begin{align*}
    \int_{[0,T]\times\overline{\Xb}\times\Ab}m^n_t(dx,da)dt=T,\quad\sup_{n\ge1} \int_{[0,T]\times\overline{\Xb}\times\Ab}|x|^qm^n_t(dx,da)dt\leq C T.
\end{align*}
By \cite[Corollary A.4]{dumitrescu2023linear}, there exists a subsequence, which we do not relabel, and a limit measure $m \in \mathcal{M}^2_{+}([0,T]\times\overline{\mathbb{X}}\times\mathbb{A})$ satisfying 
$\int_{[0,T]\times\overline{\Xb}\times\Ab}m(dt,dx,da)=T$
such that $m^n \rightarrow m$ in the $\tau_2$ topology.
We now show that $m\in V_2$. By the disintegration theorem, this is equivalent to showing that the marginal of $m$ on  $[0,T]$ is the Lebesgue measure.
To this end, we use the relative compactness criterion given in  \cite[Theorem 2 and Extension 1]{rossi2003tightness} for the convergence in measure topology. 
Consider the mapping $H:\Pc(\overline{\Xb})\rightarrow[0,\infty]$ defined by 
\begin{align*}
    H(m)=\int_{\overline{\Xb}}|x|^qm(dx),\quad m\in\Pc(\overline{\Xb}).
\end{align*}
Following \cite[Lemma 2.9]{dumitrescu2023linear}, one can show that $H$ is a coercive integrand in the sense of \cite[Definition 2.8]{dumitrescu2023linear} or (1.7 a,b,c) in \cite{rossi2003tightness}. The integrability condition from Definition \ref{def:D0} (i) implies
\begin{align*}
    \sup_{n\geq1}\int_0^T H(m^{n,\Xb}_t)dt=\sup_{n\geq1}\int_0^T\int_{\overline{\Xb}}|x|^q m_t^{n,\Xb}(dx)dt=\sup_{n\geq1}\int_0^T\int_{\overline{\Xb}}|x|^q m_t^{n}(dx,da)dt\leq CT,
\end{align*}
where $m_t^{n,\Xb}$ denotes the marginal of $m_t^n$ on $\overline{\Xb}$. Now, by Definition \ref{def:D0} (ii), 
\begin{align*}
    \lim\limits_{h\rightarrow0}\sup_{n\geq1}\int_0^{T-h}d_{\mathrm{BL}}(m^{n,\mathbb{X}}_{t+h}, m^{n,\mathbb{X}}_{t}) \, \mathrm{d}t=0.
\end{align*}
According to \cite[Theorem 2 and Extension 1]{rossi2003tightness}, up to subsequence, $\{m^{n,\Xb}\}$ converges to $\tilde{m}\in M([0,T],\Pc(\overline{\Xb}))$ in measure.  Up
to another subsequence, $\{m_t^{n,\Xb}\}$ converges weakly to $\tilde{m}_t$ $t$-a.e. on $[0,T]$. For any bounded continuous function $\varphi\in C_b([0,T]\times\overline{\Xb})$, the dominated convergence theorem ensures that 
\begin{align*}
\lim_{n\rightarrow\infty}\int_{0}^T\int_{\overline{\Xb}}\varphi(t,x)m^{n,\overline{\Xb}}_t(dx)dt = \int_{0}^T\int_{\overline{\Xb}}\varphi(t,x)\tilde{m}_t(dx)dt.
\end{align*}
As $\{m^n\}$ converges to $m$ in the $\tau_2$ topology, we arrive at 
\begin{align*}
    \int_{0}^T\int_{\overline{\Xb}}\varphi(t,x)m(dt,dx,da) = \int_{0}^T\int_{\overline{\Xb}}\varphi(t,x)\tilde{m}_t(dx)dt.
\end{align*}
By choosing a test function $\varphi(t)$ independent of $x$, we conclude that the marginal of $m$ on $[0,T]$ is the Lebesgue measure.  It then follows from \cite[Lemma A.3]{lacker_mean_2015} that the convergence $m^n \to m$ also holds in the stable topology $\overline{\tau}_2$.

It remains to show that $m$ satisfies the conditions of Definition \ref{def:D0}. From the preceding argument, we have a subsequence (which we do not relabel) such that  $\{m_t^{n,\Xb}\}$ converges weakly to $m^{\Xb}_t$ $t$-a.e. As 
\begin{align*}
    \sup_{n\geq1}\int_{\Xb\times\Ab}|x|^qm^n_t(dx,da)= \sup_{n\geq1}\int_{\Xb}|x|^qm^{n,\Xb}_t(dx)\leq C,\quad t-a.e,
\end{align*}
it follows by the same argument as in \eqref{argument} that the limit also satisfies $\int_{\overline{\mathbb{X}}\times\mathbb{A}} |x|^q \, m_t(\mathrm{d}x,\mathrm{d}a) \le C$ for a.e. $t$.  
For condition (ii) of Definition~\ref{def:D0}, we note that the integrand converges pointwise for a.e. $t$ due to the weak convergence of the marginals and the continuity of the bounded Lipschitz metric.
Because the integrand is uniformly bounded, we can apply the dominated convergence theorem to obtain that $\int_0^{T-h} d_{\mathrm{BL}}(m^{\mathbb{X}}_{t+h}, m^{\mathbb{X}}_{t}) \, \mathrm{d}t \le C\sqrt{h}.$

Finally, consider the sequence of measures $\{\lambda^n\}_{n \ge 1}$.  Note that  $\partial\mathbb{X}$ is compact, and thus the product space $[0,T]\times\partial\mathbb{X}$ is also compact. By definition of $\mathcal{D}_0$, 
$\lambda^n([0,T]\times\partial\mathbb{X}) \le C$. Therefore, Prokhorov's theorem implies that there exists a subsequence (still denoted by $\{\lambda^n\}$) and a measure $\lambda \in \mathcal{M}_{+}([0,T]\times\partial\mathbb{X})$ such that $\lambda^n \rightarrow \lambda$ in the weak convergence topology $\tau_0$. Moreover, it is clear that  $\lambda([0,T]\times\partial\mathbb{X}) \le C$.

(ii) The convexity of $\Dc_P(\mu)$ is a direct consequence of the linearity of constraint \eqref{lpconstraint}. To verify the compactness, we first note that the set is relatively compact by the same arguments used in the proof of part (i), and we can pass easily to the limit in the constraint to conclude that $\Dc_P(\mu)$ is closed, which completes the proof.
	\end{proof}
\begin{lemma}\label{lemma:4.5}
   Let $(\nu^n,m^n,\lambda^n)_{n\geq1}\subset\Dc_0$   such that $m^n\rightarrow m$ in $\overline{\tau}_2$, then $m^{n,\Xb}\rightarrow m^{\Xb}$ in the topology of  convergence in measure, $\tilde{\tau}_2$.
\end{lemma}
\begin{proof}
    It is sufficient to show that from any subsequence of $\{m^{n,\mathbb{X}}\}$, we can extract a further subsequence that converges to $m^{\Xb}_t$ in $\tau_2$ $t$-a.e. 
    
   Let $\{m^{n}\}$ be an arbitrary subsequence. As established in the proof of Lemma~\ref{lemma:D0}, the sequence of marginals $\{m^{n,\mathbb{X}}\}$ is relatively compact in the topology of convergence in measure $\tilde{\tau}_0$. Therefore, there exists a further subsequence, which we don not relabel, such that $m^{n,\Xb}$ converges to some $\tilde{m}\in M([0,T],\Pc(\overline{\Xb}))$ in measure. 
    Up to another subsequence (still denoted by the same index),  $m^{n,\Xb}_t$ converges weakly to $\tilde{m}_t$ $t$-a.e. on $[0,T]$. Furthermore,  since $\sup_{n\geq1}\int_{\overline{\Xb}}|x|^qm^{n,\Xb}_t(dx)\leq C$ $t$-a.e., the weak convergence is strengthened to convergence in the $\tau_2$ topology for a.e. $t\in[0,T]$. As $m^n\rightarrow m$ in $\overline{\tau}_2$, we must have $\tilde{m}_t = m^{\mathbb{X}}_t$ $t$-a.e., which completes the proof.
\end{proof}

\begin{definition}\label{Defsetmp} 
Fix a triple $(\overline{\nu},\overline{m},\overline{\lambda})\in \Dc_0$. Let $\Gamma[\overline{\nu},\overline{m}]:\mathcal{D}_0\rightarrow\Rb$ be defined by
	\begin{align*}
		\Gamma[\overline{\nu},\overline{m}](\nu,m,\lambda)= \mathcal{J}^{L}_{P}[\overline{m}^{\Xb},\overline{\nu}](\nu,m,\lambda)
	\end{align*}
	where $\overline{m}^{\Xb}:=(\overline{m}^{\Xb}_t)_{0\leq t\leq T}$ denotes the marginal of $\overline{m}$ on $\overline{\Xb}$ and $\mathcal{J}^{L}_{P}$ is given by Definition~\ref{lpdef}.
    Define the set valued mapping $\mathcal{D}^*:\mathcal{D}_0\rightarrow 2^{\mathcal{D}_0}$ by  
	\[\mathcal{D}^*(\overline{\nu},\overline{m},\overline{\lambda}):=\mathcal{D}_{P}(\overline{m}^{\Xb}).
	\] Also,
	denote the set valued mapping $\Phi: \mathcal{D}_0\rightarrow 2^{\mathcal{D}_0}$ by 
	\begin{equation}\label{fixedpointmapping}
\Phi(\overline{\nu},\overline{m},\overline{\lambda}):=\underset{(\nu,m,\lambda)\in \mathcal{D}^*(\overline{\nu},\overline{m},\overline{\lambda})}{\arg\min}\Gamma[\overline{\nu},\overline{m}](\nu,m,\lambda).
	\end{equation}
\end{definition}
\begin{remark}\label{remark:fix_lp}
Note that the set of LPMFEs coincides with the set of fixed points of \(\Phi\).
\end{remark}

\begin{lemma}\label{lemma:D*}
 {Under Assumptions~\ref{c1}  and  \ref{assumption:2}}, the set-valued mapping $\mathcal{D}^*$ is continuous.
\end{lemma}
\begin{proof}
	\textit{Step-1}. {\sl  We first establish the upper hemicontinuity}. By the Closed Graph Theorem (cf. \cite[Theorem 17.11]{aliprantis_infinite_2006}), it suffices to verify that $\mathcal{D}^\star$ has a closed graph. To this end, let $\left(\nu^n, m^n, \lambda^n\right) \in \mathcal{D}^{\star}\left(\bar{\nu}^n, \bar{m}^n, \bar{\lambda}^n\right)$ be a sequence such that $(\nu^n, m^n, \lambda^n) \rightharpoonup (\nu, m, \lambda)$  and $(\bar{\nu}^n, \bar{m}^n, \bar{\lambda}^n) \rightharpoonup (\bar{\nu}, \bar{m}, \bar{\lambda})$ in the topology $\tau_2 \otimes \overline{\tau}_2 \otimes \tau_0$.  Then, for each \( n \geq 1 \) and any \( u \in C_b^{1,2}([0, T] \times \overline{\Xb}) \), it holds that
	\begin{align}\label{upperhemi}
		\int_{\overline{\Xb}} u(T, x) \nu^n(d x) 
		= \int_{\overline{\Xb}} u(0, x) m^*_0(d x) 
		&+ \int_{[0, T] \times \overline{\Xb} \times \mathbb{A}} \bigl(\partial_t u + \Lc u\bigr)(t, x, \bar{m}^{n,\Xb}_t, a) \, m^n_t(d x, d a) \, d t \\
		& + \int_{[0, T] \times \partial\Xb} \mathcal{A} u(t, x) \, \lambda^n(d t, d x). \nonumber
	\end{align}
	We next pass to the limit in each term.  The convergence of the following terms is straightforward
    \begin{align*}
    		&\int_0^T \int_{\overline{\Xb} \times \Ab} \partial_t u(t, x) \, m^n_t(d x, d a) \, d t \xrightarrow{n \to \infty} \int_0^T \int_{\overline{\Xb} \times \Ab} \partial_t u(t, x) \, m_t(d x, d a) \, d t,\\
            	&\int_{\overline{\Xb}} u(T, x) \, \nu^n(d x) \xrightarrow{n \to \infty} \int_{\overline{\Xb}} u(T, x) \, \nu(d x), \\
	&\int_{[0, T] \times \partial\Xb} m(x)\cdot \partial_{x}u(t, x) \, \lambda^n(d t, d x)
	\xrightarrow{n \to \infty}\int_{[0, T] \times \partial\Xb} m(x)\cdot \partial_{x}u(t, x) \, \lambda(d t, d x),
    \end{align*}
    where the third convergence uses the fact that the inward normal vector field $m$ is $C(\partial\mathbb{X})$ due to the $C^1$-regularity of $\partial\mathbb{X}$.
 We claim that
    \begin{align*}
	&\int_0^T \int_{\overline{\Xb} \times \Ab}
	\Lc u
	\left(t, x, \bar{m}^{n,\Xb}_t, a\right) m^n_t(d x, d a) \, d t 
	\xrightarrow{n \to \infty}
	\int_0^T \int_{\overline{\Xb} \times \Ab}
	\Lc u
	\left(t, x, \bar{m}^{\Xb}_t, a\right) m_t(d x, d a) \, d t.
\end{align*}
To prove this claim, it suffices to show that every subsequence admits a further subsequence along which the convergence holds. We illustrate this with the drift term, as the remaining terms in $\mathcal{L}u$ can be handled similarly.  For any subsequence of $\{\int_0^T \int_{\overline{\Xb} \times \Ab} \left( b\cdot \partial_x u\right)\left(t, x, \bar{m}^{n,\Xb}_t, a\right) m^n_t(d x, d a) dt\}_{n\geq1}$, by Lemma \ref{lemma:4.5}, we can extract a further subsequence (which we do not relabel)  such that $\bar{m}^{n,\Xb}_t$ converges to $\bar{m}^{\Xb}_t$ in $\tau_2$ $t$-a.e.   We now apply Lemma \ref{lemma:technical} (ii), setting the components as $\Theta=[0,T]$, $\Xc=\overline{\Xb}\times\Ab$, $\mathcal{Y}=\Pc_2(\overline{\Xb})$, $\eta(dt)=dt$, $\psi^n(t)=\bar{m}^{n,\Xb}_t$, $\psi(t)=\bar{m}^{\Xb}_t$, $\upsilon^n=m^n$, $\upsilon=m$ and  $\varphi(t,x,a,y)= \left( b\cdot \partial_x u\right)\left(t, x, y, a\right)$.
By the Definition of $\Dc_0$, we can conclude that there exists a constant $C>0$ such that
\begin{align*}
    \sup_{n\geq1}\int_{\overline{\Xb}\times\Ab}(1+|x|^2+|a|^2)m^n_t(dx,da)\leq C,\quad t-a.e..
\end{align*}
Furthermore, for a.e. $t$, $\bar{m}^{n,\mathbb{X}}_t$ belongs to the set $\mathcal{K}=\{\iota\in\mathcal{P}_2(\overline{\mathbb{X}}):\int|x|^q\iota(\mathrm{d}x)\le C\}$, which is a compact subset of $\mathcal{P}_2(\overline{\mathbb{X}})$ in the $\tau_2$ topology, as shown in Lemma~\ref{lemma:D0}. With all conditions satisfied, Lemma~\ref{lemma:technical} (ii) yields the desired convergence for the subsequence, which justifies our claim.

Passing to the limit in all terms of \eqref{upperhemi}, we have
\begin{align*}
	\int_{\overline{\Xb}} u(T, x) \, \nu(d x)
	= \int_{\overline{\Xb}} u(0, x) \, m^*_0(d x)
	&+ \int_{[0, T] \times \overline{\Xb} \times \Ab}
	\left(\partial_t u + \Lc u\right)(t, x, \bar{m}^{\Xb}_t, a) \, m_t(d x, d a) \, d t \\
	& + \int_{[0, T] \times \partial\Xb} \mathcal{A} u(t, x) \, \lambda(d t, d x),
\end{align*}
 which implies $\left(\nu, m, \lambda\right) \in \mathcal{D}^{\star}\left(\bar{\nu}, \bar{m}, \bar{\lambda}\right)$.

	 {   \textit{Step-2}. {\sl We now show the lower hemicontinuity}.  Consider a sequence $\left(\bar{\nu}^n, \bar{m}^n,\bar{\lambda}^n\right)_{n \geq 1} \subset \mathcal{D}_0$ such that $\left(\bar{\nu}^n, \bar{m}^n,\bar{\lambda}^n\right) \rightarrow(\bar{\nu}, \bar{m},\bar{\lambda})$  in the topology $\tau_2 \otimes \overline{\tau}_2 \otimes \tau_0$ and let $(\nu, m,\lambda) \in$ $\mathcal{D}^{\star}(\bar{\nu}, \bar{m},\bar{\lambda})=\mathcal{D}_{P}(\bar{m}^{\Xb})$. We need to show that up to a subsequence, there exists a sequence  $\left(\nu^n, m^n,\lambda^n\right)_{n \geq 1} \subset$ $\mathcal{D}_0$ such that $\left(\nu^n, m^n,\lambda^n\right) \in \mathcal{D}^{\star}\left(\bar{\nu}^n, \bar{m}^n,\bar{\lambda}^n\right)=\mathcal{D}_{P}\left(\bar{m}^{n,\Xb}\right)$ and $\left(\nu^n, m^n,\lambda^n\right) \rightarrow(\nu, m,\lambda)$ in the topology $\tau_2 \otimes \overline{\tau}_2 \otimes \tau_0$. By the uniform integrability conditions in Definition \ref{def:D0}, \cite[Theorem 7.12]{Villani2003Topics} and \cite[Lemma A.3]{lacker_mean_2015}, it suffices to prove this convergence in the topology $\tau_0\otimes \tau_0 \otimes \tau_0$.
     
     Let $v_{t, x}(d a)$ be such that
	$$
	m_t(d x, d a)dt=v_{t, x}(d a) m^{\Xb}_t(d x) d t.
	$$
By Proposition~\ref{thm:prob-rep-R0}, there exists a filtered probability space 
\((\Omega, \mathcal{F}, \mathbb{F}, \mathbb{P})\) supporting 
$d$-dimensional orthogonal continuous $\mathbb{F}$-martingale measures 
\(M = (M^1, \cdots, M^d)\) with intensity \(v_{t, X_{t-}}(da)dt\),  
a random counting measure \(\mathcal{N}\) on \([0,T] \times \mathbb{A}\) with compensator 
\(\pi(t,a)v_{t, X_{t-}}(da)dt\),  
and a pair of $\mathbb{F}$-progressively measurable processes \((X,R)\) such that
\(X\) takes values in \(\overline{\mathbb{X}}\) and
\(R\) is non-decreasing, RCLL, with \(R_0=0\).
Moreover, \((X,R)\) satisfies
\begin{align*}
	dX_t = \left(\int_{\mathbb{A}} b(t, X_t, \bar{m}^{\Xb}_t, a)\, v_{t, X_t}(da)\right) dt 
	&+ \int_{\mathbb{A}} \sigma(t, X_t, \bar{m}^{\Xb}_t, a)\, M(dt,da) \\
	& + \int_{\mathbb{A}} \beta(t, X_{t-}, \bar{m}^{\Xb}_t, a)\, \tilde{\mathcal{N}}(dt,da) 
	+ m(X_t)\, dR_t, 
	\quad \mathbb{P} \circ X_0^{-1} = m_0^*,
\end{align*}
 and $R_t= \int_0^t \mathbf{1}_{\{\partial\mathbb{X}\}}(X_s) \, dR_s$, where \(\tilde{\mathcal{N}}(dt,da) := \mathcal{N}(dt,da) -\pi(t,a) v_{t, X_{t-}}(da)dt\)  
denotes the compensated random measure.  
Finally, the measures \(\nu\), \(m\), and \(\lambda\) have the following representation
\begin{align*}
	&\nu = \mathbb{P} \circ X_T^{-1}, \\
	&m_t(B \times C)
	= \mathbb{E}^{\mathbb{P}}\!\left[\mathbf{1}_{B}(X_{t})\, v_{t, X_{t}}(C)\right]=\mathbb{E}^{\mathbb{P}}\!\left[\mathbf{1}_{B}(X_{t-})\, v_{t, X_{t-}}(C)\right],
	\quad B \in \mathcal{B}(\overline{\mathbb{X}}),\; C \in \mathcal{B}(\mathbb{A}),\; t-a.e., \\
	&\lambda(D)
	= \mathbb{E}^{\mathbb{P}}\!\left[\int_0^T \mathbf{1}_{D}(t, X_{t})\, dR^c_t\right]=\mathbb{E}^{\mathbb{P}}\!\left[\int_0^T \mathbf{1}_{D}(t, X_{t-})\, dR^c_t\right],
	\quad D \in \mathcal{B}([0,T] \times \partial\mathbb{X}).
\end{align*}
On the same filtered probability space, let us define
\[
m_t^n(B \times C)
:= \mathbb{E}^{\mathbb{P}}\!\left[\mathbf{1}_{B}(X^n_t)\, v_{t, X_t}(C)\right], 
\quad 
\nu^n := \mathbb{P} \circ (X^n_T)^{-1}, 
\quad 
\lambda^n(D) := \mathbb{E}^{\mathbb{P}}\!\left[\int_0^T \mathbf{1}_{D}(t, X^n_t)\, dR^{n,c}_t\right],
\]
where \((X^n, R^n)\) denotes the unique strong solution of the reflected SDE  
\begin{align}\label{sdexn}
	X^n_t
	= X_0 
	&+ \int_0^{t} \int_{\mathbb{A}} b\bigl(s, X^n_s, \bar{m}^{n,\mathbb{X}}_s, a\bigr)\, v_{s, X_s}(da)\, ds
	+ \int_0^{t} \int_{\mathbb{A}} \sigma\bigl(s, X^n_s, \bar{m}^{n,\mathbb{X}}_s, a\bigr)\, M(ds, da) \nonumber \\
	& + \int_0^{t} \int_{\mathbb{A}} \beta\bigl(s, X^n_{s-}, \bar{m}^{n,\mathbb{X}}_s, a\bigr)\, {\tilde{\mathcal{N}}(ds, da)}
	+ \int_0^{t} m(X^n_s)\, dR^n_s,\quad X^n_0=X_0.
\end{align}
Note that existence and uniqueness follow by the Lipschitz and boundedness condition on the coefficients and square integrability of $m^*_0$.
Using an argument similar to Remark~\ref{remark:s_to_r}, we have 
\[
\left(\nu^n, m^n, \lambda^n\right) \in \mathcal{D}_{P}\!\left(\bar{m}^{n,\mathbb{X}}\right) 
= \mathcal{D}^{\star}\!\left(\bar{\nu}^n, \bar{m}^n, \bar{\lambda}^n\right).
\]
We now show \(m^n \to m\) in \(\tau_0\). By \cite[Remark 8.3.1 and Exercise 8.10.71]{Bogachev_2007}, it suffices to test convergence against bounded Lipschitz functions. Let \(\phi: [0, T] \times \overline{\mathbb{X}} \times \mathbb{A} \to \mathbb{R}\) be bounded and Lipschitz continuous.
 Then
{\small	\begin{align*}
		&\left| \int_0^T \int_{\overline{\mathbb{X}} \times \mathbb{A}} \phi(t, x, a)\, m_t(dx, da)\,dt 
		- \int_0^T \int_{\overline{\mathbb{X}} \times \mathbb{A}} \phi(t, x, a)\, m_t^n(dx, da)\,dt \right| \\
		&= \left| \mathbb{E}^{\mathbb{P}} \left[ \int_0^T \int_{\mathbb{A}} \phi(t, X_t, a)\, v_{t,X_t}(da)\,dt
		- \int_0^T \int_{\mathbb{A}} \phi(t, X^n_t, a)\, v_{t,X_t}(da)\,dt \right] \right| \\
		&\leq C \left( \mathbb{E}^{\mathbb{P}}\left[ \sup_{0 \leq t \leq T} |X_t - X^n_t|^2 \right] \right)^{1/2}.
	\end{align*}}
	Now, we show 
	\begin{align}\label{convergenceX}
		\mathbb{E}^{\mathbb{P}}\left[ \sup_{0 \leq t \leq T} |X_t - X^n_t|^2 \right] \rightarrow 0, \qquad \text{as } n \rightarrow \infty.
	\end{align}
	To simplify the notation, we introduce the following shorthand: $b_n(t, x, a):=b(t, x, \bar{m}_t^{n,\Xb}, a)$, $b_0(t, x, a):=b(t, x, \bar{m}_t^{\Xb}, a)$, $\sigma_n(t, x, a):=\sigma(t, x, \bar{m}_t^{n,\Xb}, a)$, $\sigma_0(t, x, a):=\sigma(t, x, \bar{m}_t^{\Xb}, a)$, $\beta_n(t, x, a):=\beta(t, x, \bar{m}_t^{n,\Xb}, a)$ and $\beta_0(t, x, a):=\beta(t, x, \bar{m}_t^{\Xb}, a)$. 
    By Lemma \ref{lemma:comparison}, we get 
  {\small \begin{align*}
        \mathbb{E}^{\mathbb{P}}\left[ \sup_{0 \leq s \leq t} |X_s - X^n_s|^2 \right]&\leq C\bigg(\Eb^{\Pb}\bigg[\left(\int_0^t\int_{\Ab}|b_0(s,X_{s},a)-b_n(s,X^n_s,a)|v_{s,X_s}(da)ds\right)^2\\&+\int_0^t\int_{\Ab}|\sigma_0(t,X_t,a)-\sigma_n(t,X^n_t,a)|^2v_{s,X_s}(da)ds\\&+\int_0^t\int_{\Ab}|\beta_0(t,X_{t-},a)-\beta_n(t,X^n_{t-},a)|^2\mathcal{N}(ds,da)\bigg]  \bigg)\\
        &\leq C\bigg(\Eb^{\Pb}\bigg[\int_0^t\int_{\Ab}|b_0(s,X_{s},a)-b_n(s,X^n_s,a)|^2v_{s,X_s}(da)ds\\&+\int_0^t\int_{\Ab}|\sigma_0(t,X_t,a)-\sigma_n(t,X^n_t,a)|^2v_{s,X_s}(da)ds\\&+\int_0^t\int_{\Ab}|\beta_0(t,X_{t-},a)-\beta_n(t,X^n_{t-},a)|^2\pi(t,a)v_{s,X_{s-}}(da)ds\bigg]  \bigg).
        \end{align*}}
    {
The Lipschitz assumption on \(b\) then yields 
{\small	\begin{align*}
		& \int_0^{t} \int_\Ab\left|b_n\left(r, X_r^n, a\right)-b_0\left(r, X_r, a\right)\right|^2  v_{r,X_r}(da) d r \\
		&= \int_0^{t} \int_\Ab\left|b_n\left(r, X_r^n, a\right)-b_n\left(r, X_r, a\right)+b_n\left(r, X_r, a\right)-b_0\left(r, X_r, a\right)\right|^2v_{r,X_r}(da) d r \\
		& \leq C\left[\int_0^{t} \int_\Ab\left|b_n\left(r, X_r^n, a\right)-b_n\left(r, X_r, a\right)\right|^2 v_{r,X_r}(da) d r+\int_0^{t} \int_\Ab\left|b_n\left(r, X_r, a\right)-b_0\left(r, X_r, a\right)\right|^2 v_{r,X_r}(da) d r\right] \\
		& \leq C\left[\int_0^t \sup _{0\leq r \leq s}\left|X_{r}^n-X_{r}\right|^2 d s+\int_0^{t } \int_\Ab\left|b_n\left(r, X_r, a\right)-b_0\left(r, X_r, a\right)\right|^2 v_{r,X_r}(da) d r\right].
	\end{align*}}
Similarly,  we can also obtain 
	{\small\begin{align*}
		&\int_0^{t} \int_\Ab\left|\sigma_n\left(r, X_r^n, a\right)-\sigma_0\left(r, X_r, a\right)\right|^2  v_{r,X_r}(da) d r\\
		&\leq C\left[\int_0^t \sup _{0\leq r \leq s}\left|X_{r}^n-X_{r}\right|^2 d s+\int_0^{t } \int_\Ab\left|\sigma_n\left(r, X_r, a\right)-\sigma_0\left(r, X_r, a\right)\right|^2 v_{r,X_r}(da) d r\right],\\
		&\int_0^{t} \int_\Ab\left|\beta_n\left(r, X_{r-}^n, a\right)-\beta_0\left(r, X_{r-}, a\right)\right|^2\pi(t,a)  v_{r,X_{r-}}(da) d r\\
		&\leq  C\left[\int_0^t \sup _{0\leq r \leq s}\left|X_{r}^n-X_{r}\right|^2 d s+\int_0^{t} \int_\Ab\left|\beta_n\left(r, X_{r-}, a\right)-\beta_0\left(r, X_{r-}, a\right)\right|^2v_{r,X_{r-}}(da) d r\right].
	\end{align*}}
Putting all the pieces together, we get
	$$
	 \mathbb{E}^{\mathbb{P}}\left[ \sup_{0 \leq s \leq t} |X_s - X^n_s|^2 \right] \leq C\left(\int_0^t  \mathbb{E}^{\mathbb{P}}\left[ \sup_{0 \leq r \leq s} |X_r - X^n_r|^2 \right] d s+A_n+B_n+ C_n\right),
	$$
	where
{\small	
	\begin{align*}
		A_n & :=\mathbb{E}^{\mathbb{P}}\left[\int_0^T \int_\Ab\left|b_n\left(r, X_r, a\right)-b_0\left(r, X_r, a\right)\right|^2 v_{r,X_r}(da) d r\right], \\
		B_n & :=\mathbb{E}^{\mathbb{P}}\left[\int_0^T \int_\Ab\left|\sigma_n\left(r, X_r, a\right)-\sigma_0\left(r, X_r, a\right)\right|^2 v_{r,X_r}(da) d r\right],\\
		C_n& :=\mathbb{E}^{\mathbb{P}}\left[\int_0^{T} \int_\Ab\left|\beta_n\left(r, X_{r-}, a\right)-\beta_0\left(r, X_{r-}, a\right)\right|^2 v_{r,X_{r-}}(da) d r\right].
	\end{align*}
	}
Using Gronwall's inequality yields
	$$
	\mathbb{E}^{\mathbb{P}}\left[ \sup_{0 \leq s \leq T} |X_s - X^n_s|^2 \right]\leq C\left(A_n+B_n+C_n\right) e^{C T}.
	$$	
  Let us next show that $A_n \to 0$ as $n \to \infty$.   Similar to the arguments in upper hemicontinuity,  we can show that for all $\omega \in \Omega$,
  \[
  \int_0^T \int_A \bigl| b_n\bigl(r, X_r(\omega), a\bigr) - b_0\bigl(r, X_r(\omega), a\bigr) \bigr|^2 \, v_{r,X_r(\omega)}(da)\,dr 
  \longrightarrow 0 \quad\text{as } n \to \infty.
  \]
 Indeed, fix $\omega \in \Omega$. To show this convergence, it is sufficient to prove that any subsequence has a further subsequence that converges to zero. For any given subsequence, we can extract a further subsequence (which we do not relabel) such that $\bar{m}^{n,\mathbb{X}}_t \to \bar{m}^{\mathbb{X}}_t$ in $\tau_2$ for $t$-a.e. 
We can then apply Lemma~\ref{lemma:technical} (ii) for this fixed $\omega$. We set $\Theta = [0, T]$, $\mathcal{X} = \mathbb{A}$, $\mathcal{Y}=\mathcal{P}_2(\overline{\mathbb{X}})$, $\eta(\mathrm{d}t) = \mathrm{d}t$, $\psi^n(t)=\bar{m}_t^{n,\mathbb{X}}$, $\psi(t)=\bar{m}^{\mathbb{X}}_t$, the measures $\upsilon_t^n(\mathrm{d}a)=\upsilon_t(\mathrm{d}a) = v_{t,X_t(\omega)}(\mathrm{d}a)$, and 
\[
  \varphi(r,a,y) = \left| b\bigl(r, X_r(\omega), y, a\bigr) - b\bigl(r, X_r(\omega), \psi(r), a\bigr) \right|^2.
\]
As in the proof of upper hemicontinuity, all hypotheses of Lemma~\ref{lemma:technical} (ii) are satisfied. The lemma therefore yields the desired convergence for the subsequence.
 Finally, using the dominated convergence theorem, we deduce that $A_n \to 0$. The convergence results for $B_n$ and $C_n$ follow from the same arguments.}
  Hence \eqref{convergenceX} holds, which leads to $m^n \to m$ in $\tau_0$. 
 An analogous argument shows that $\nu^n \to \nu$ in $\tau_0$. 
Finally, we consider the sequence $\{\lambda^n\}$. Note that $(\nu^n, m^n, \lambda^n) \in \mathcal{D}_{P}(\bar{m}^{n,\mathbb{X}})$ and $(\nu, m, \lambda) \in \mathcal{D}_{P}(\bar{m}^{\mathbb{X}})$. From the compactness of $\mathcal{D}_0$ (Lemma~\ref{lemma:D0}), we get that, up to subsequence, $\lambda^n\rightarrow\tilde{\lambda}$ in $\tau_0$ for some measure $\tilde{\lambda}\in\Mc_{+}([0,T]\times\partial\Xb)$. Passing to the limit in \eqref{upperhemi}, we obtain the following identity for every $u \in C_b^{1,2}([0,T] \times \overline{\mathbb{X}})$:
	\[
	\int_{[0,T] \times \partial\mathbb{X}} \partial_x u(t,x) \cdot m(x) \, \tilde{\lambda}(dt,dx)
	=
	\int_{[0,T] \times \partial\mathbb{X}} \partial_x u(t,x) \cdot m(x) \, \lambda(dt,dx).
	\]
   In the case that $\mathbb{X}=(0,\infty)$, this directly implies $\tilde{\lambda} = \lambda$. In the case that $\mathbb{X}$ is a bounded subset of $\mathbb{R}^d$, applying Lemma~\ref{lemma:boundary_equal} leads to the same conclusion. Therefore, we conclude that, up to a subsequence, $\lambda^n \to \lambda$ in $\tau_0$. }
\end{proof}
\begin{theorem}\label{LPthm}
	Under Assumptions ~\ref{c1} and \ref{assumption:2}, there exists an LPMFE.  Consequently, there also exists an  MFE with a Markovian relaxed control.
\end{theorem}

\begin{proof}
	We first prove the existence of an LPMFE. As noted in Remark~\ref{remark:fix_lp}, this is equivalent to finding a fixed point of  
	\(\Phi\colon \mathcal{D}_0 \rightarrow 2^{\mathcal{D}_0}\).  We apply the Kakutani–Fan–Glicksberg fixed point theorem (see, e.g.,~\cite[Corollary 17.55]{aliprantis_infinite_2006}) to \(\Phi\).   
    
    The set $\mathcal{D}_0$ is a nonempty, convex subset of the locally convex Hausdorff space $\mathcal{M}(\overline{\mathbb{X}}) \times \tilde{V}_0 \times \mathcal{M}([0,T]\times\partial\mathbb{X})$, endowed with the topology $\tau_0 \otimes \tau_0 \otimes \tau_0$.  In Lemma~\ref{lemma:D0}, we established that $\mathcal{D}_0$ is compact under the topology $\tau_2 \otimes \overline{\tau}_2 \otimes \tau_0$. The uniform integrability conditions that define $\mathcal{D}_0$ ensure that 
    $\tau_0\otimes\tau_0\otimes\tau_0$ and  $\tau_2 \otimes \overline{\tau}_2 \otimes \tau_0$ both induce the same topology on $\Dc_0$. Therefore, $\mathcal{D}_0$ is also compact under the $\tau_0 \otimes \tau_0 \otimes \tau_0$ topology.
	  Moreover, \(\Phi\) has  convex values.  It remains to show that \(\Phi\) has a closed graph and nonempty values.  By Berge’s Maximum Theorem (cf.~\cite[Theorem 17.31]{aliprantis_infinite_2006}) and the Closed Graph Theorem (cf. \cite[Theorem 17.11]{aliprantis_infinite_2006}), it suffices to verify that
	\[
	((\bar{\nu},\bar{m},\bar{\lambda}),(\nu,m,\lambda)) \;\longmapsto\;
	\Gamma[\bar{\nu},\bar{m}](\nu,m,\lambda)
	\]
	is continuous on \(\operatorname{Gr}(\mathcal{D}^\star)\).  Let 
	\[
	\bigl((\bar{\nu}^n,\bar{m}^n,\bar{\lambda}^n),(\nu^n,m^n,\lambda^n)\bigr)\;\to\;\bigl((\bar{\nu},\bar{m},\bar{\lambda}),(\nu,m,\lambda)\bigr)
	\quad\text{in }\operatorname{Gr}(\mathcal{D}^\star).
	\]
    	Then, as in the proof of upper hemicontinuity in Lemma~\ref{lemma:D*}, we can deduce the following convergence by noting that the growth condition on $\tilde{f}$ from Remark~\ref{tfremark} for the bounded action space case, which holds even when $\beta$ is not bounded, satisfies the requirement of Lemma~\ref{lemma:technical}~(ii):
	\[
	\int_0^T\!\int_{\overline{\mathbb{X}}\times \mathbb{A}}
	\tilde{f}\bigl(t,x,\bar m^{n,\Xb}_t,a\bigr)\,m^n_t(dx,da)\,dt
	\;\longrightarrow\;
	\int_0^T\!\int_{\overline{\mathbb{X}}\times \mathbb{A}}
	\tilde{f}\bigl(t,x,\bar m^{\Xb}_t,a\bigr)\,m_t(dx,da)\,dt.
	\]
	Similarly, using Lemma~\ref{lemma:technical} (i) gives
	\begin{align*}
		&\int_{\overline{\mathbb{X}}} g\bigl(x,\bar{\nu}^n\bigr)\,\nu^n(dx)
		\;\longrightarrow\;
		\int_{\overline{\mathbb{X}}} g\bigl(x,\bar{\nu}\bigr)\,\nu(dx).
	\end{align*}
	Moreover, it is clear that 
	\begin{align*}
		\int_{[0,T]\times\partial\mathbb{X}} h\bigl(t,x\bigr)\,\lambda^n(dt,dx)
		\;\longrightarrow\;
		\int_{[0,T]\times\partial\mathbb{X}} h\bigl(t,x\bigr)\,\lambda(dt,dx).
	\end{align*}
	Hence, we arrive at
	\[
	\Gamma[\bar\nu^n,\bar m^n](\mu^n,m^n,\lambda^n)
	\;\longrightarrow\;
	\Gamma[\bar\mu,\bar m](\mu,m,\lambda),
	\]
	which shows the desired continuity of \(\Gamma\) on \(\operatorname{Gr}(\mathcal{D}^\star)\).  All conditions of the Kakutani–Fan–Glicksberg theorem are then satisfied, and hence \(\Phi\) admits a fixed point, denoted by $(\nu^*,m^*,\lambda^*)$. By the definition of $\Dc_0$, we have $m^{*,\Xb}\in B_2$; hence, by Theorem~\ref{theorem:equi}, $(\nu^*,m^*,\lambda^*)$ induces an MFE with a Markovian relaxed control.
\end{proof}
\subsection{Existence of LPMFEs in the unbounded case}
We now extend the existence result in the previous subsection to the general case when $b$, $\sigma$ and  $\beta$ have linear growth and the action space $\Ab$ is not necessarily bounded.

As preparation, let us introduce some notations. For any $n \ge 1$, let $(b_n, \sigma_n, \beta_n)$ denote the truncated versions of the coefficients $(b, \sigma, \beta)$, defined by pointwise projection onto the closed ball of radius $n$ in their respective spaces. 
Moreover, we denote by $\Ab_n$ the intersection of $\Ab$ with the ball centred at the origin with radius $n$. 
Since $\mathbb{A}$ is closed, there exists $n_0$ such that for all $n \ge n_0$, the set $\mathbb{A}_n$ is nonempty and compact. Consequently, the truncated data $(b_n, \sigma_n, \beta_n, \pi, \tilde{f}, g, h, \mathbb{A}_n)$ satisfies Assumptions \ref{c1}, \ref{assumption:2} and Remark \ref{tfremark} with the same constants $c_1, c_2, c_3$, and $c_4$ independent of $n$. We note that we do not truncate the $\beta$ appearing in $\tilde{f}$; hence, $\tilde{f}$ itself does not depend on $n$.

Let $\mathcal{L}^n$ be the operator defined as $\mathcal{L}$ but with the truncated coefficients $(b_n, \sigma_n, \beta_n)$. For any measurable flow $\mu$, let $\mathcal{D}^n_{P}(\mu)$ be the set defined analogously to $\mathcal{D}_{P}(\mu)$ (see Definition~\ref{def:lpocm}), but using the operator $\mathcal{L}^n$ and requiring that the measure $m$ is supported on $[0,T]\times\overline{\mathbb{X}} \times \mathbb{A}_n$. 
By Theorem~\ref{LPthm}, for each $n \ge n_0$, there exists an LPMFE $(\nu^n, m^n, \lambda^n)$ for the corresponding truncated problem. 
Indeed, although the statement of Theorem~\ref{LPthm} assumes that $\beta$ is bounded, the proof remains valid even when the $\beta$ appearing in the definition of $\tilde{f}$ is not truncated. For each $n$, the solution $(\nu^n, m^n, \lambda^n)$ can be naturally viewed as an element of the  space $\mathcal{P}_2(\overline{\mathbb{X}}) \times V_2 \times \mathcal{M}_{+}([0,T]\times\partial\mathbb{X})$.

The main tool in the proof is the convergence argument inspired by \cite[Section 5]{lacker_mean_2015}, which is employed to show that a limit point of the sequence $\{(\nu^n, m^n, \lambda^n)\}$ from the case with bounded coefficients is an LPMFE in the original problem with general model coefficients, i.e., a fixed point of the mapping $\Phi$ defined in \eqref{fixedpointmapping}.
\begin{lemma}\label{lemma51}
    The sequence $\{(\nu^n, m^n, \lambda^n)\}$ is relatively compact in $\mathcal{P}_2(\overline{\mathbb{X}}) \times V_2 \times \mathcal{M}_{+}([0,T]\times\partial\mathbb{X})$ under the topology $\tau_2 \otimes \overline{\tau}_2 \otimes \tau_0$. Moreover, the sequence satisfies the following uniform estimates: 
    \begin{align}\label{uni1}
     &\sup_{n\geq1}\int_{\overline{\Xb}}|x|^q\nu^n(dx)<\infty,   \quad \sup_{n\geq1}\int_0^T\int_{\overline{\Xb}\times\Ab}(|x|^q+|a|^q)m^n_t(dx,da)dt<\infty,\nonumber\\&\sup_{n\geq1}\lambda^n([0,T]\times\partial\Xb)<\infty,\quad \sup_{n\geq1}\|m^{n,\Xb}\|^q_{T}< \infty,
    \end{align}
    and
    \begin{align}\label{uni2}
           \lim\limits_{h\rightarrow0}\sup_{n\geq1}\int_0^{T-h}d_{\mathrm{BL}}(m^{n,\mathbb{X}}_{t+h}, m^{n,\mathbb{X}}_{t}) \, \mathrm{d}t=0,
    \end{align}
      where $q$ is the constant specified in  Assumption~\ref{c1}.
  Furthermore, any limit point $(\nu^*, m^*, \lambda^*)$ of a convergent subsequence also satisfies these bounds that
  \begin{align}\label{limit_bounds}
    &\int_{\overline{\Xb}}|x|^q\nu^*(dx)<\infty,   \quad 
        \int_0^T\int_{\overline{\Xb}\times\Ab}(|x|^q+|a|^q)m^*_t(dx,da)dt<\infty, \\
        &\lambda^*([0,T]\times\partial\Xb)<\infty,\quad \|m^{*,\Xb}\|^q_{T}< \infty.\nonumber
    \end{align}
\end{lemma}
\begin{proof}
In the following proof, we use $C$ to denote a generic positive constant that may change from line to line.
     Let $v^{n}_{t, x}(d a)$ be such that $m^n_t(d x, d a)dt=v^n_{t, x}(d a) m^{n,\Xb}_t(d x) d t.$ By Proposition~\ref{thm:prob-rep-R0}, the measures \(\nu^n\), \(m^n\), and \(\lambda^n\) have the following representation
{\small\begin{align*}
	&\nu^n = \mathbb{P}^n \circ (X^n_T)^{-1}, \\
	&m^n_t(B \times C)
	= \mathbb{E}^{\mathbb{P}^n}\!\left[\mathbf{1}_{B}(X^n_{t})\, v^{n}_{t, X^n_{t}}(C)\right]=\mathbb{E}^{\mathbb{P}^n}\!\left[\mathbf{1}_{B}(X^n_{t-})\, v^{n}_{t, X^n_{t-}}(C)\right],
	\quad B \in \mathcal{B}(\overline{\mathbb{X}}),\; C \in \mathcal{B}(\mathbb{A}),\; t-a.e., \\
	&\lambda^n(D)
	= \mathbb{E}^{\mathbb{P}^n}\!\left[\int_0^T \mathbf{1}_{D}(t, X^n_{t})\, dR^{n,c}_t\right]=\mathbb{E}^{\mathbb{P}^n}\!\left[\int_0^T \mathbf{1}_{D}(t, X^n_{t-})\, dR^{n,c}_t\right],
	\quad D \in \mathcal{B}([0,T] \times \partial\mathbb{X}),
\end{align*}}
where \((X^n,R^n)\) is defined on some filtered probability space 
\((\Omega^n, \mathcal{F}^n, \mathbb{F}^n, \mathbb{P}^n)\) satisfying
\begin{eqnarray*}
	dX^n_t &=& \left(\int_{\mathbb{A}} b_n(t, X^n_t, m^{n,\Xb}_t, a)\, v^{n}_{t, X^n_t}(da)\right) dt 
	+ \int_{\mathbb{A}} \sigma_n(t, X^n_t, m^{n,\Xb}_t, a)\, M^n(dt,da) \\
	&& + \int_{\mathbb{A}} \beta_n(t, X^n_{t-}, m^{n,\Xb}_t, a)\, \tilde{\mathcal{N}}^n(dt,da) 
	+ m(X_t)\, dR^n_t, 
	\quad \mathbb{P}^n \circ (X^n_0)^{-1} = m_0^*.
\end{eqnarray*}
Noting that the coefficients $(b_n,\sigma_n,\beta_n)$ satisfy Assumptions \ref{c1} and \ref{assumption:2} with same constants, Lemma \ref{lemma:A6} (i) implies that 
\begin{align}\label{mn_estimate}
    \|m^{n,\Xb}\|^2_{T}\leq C\left[1+\Eb^{\Pb^n}\left[\int_0^T\int_{\Ab}|a|^2v^n_{t,X^n_{t}}(da)dt\right]\right]=C\left[1+\int_0^T\int_{\overline{\Xb}\times\Ab}|a|^2m^n_{t}(dx,da)dt\right].
\end{align}
Therefore, the above inequality, in conjunction with \eqref{app1:estimate}, yields
\begin{align}\label{nun_estimate}
    \int_{\overline{\Xb}}|x|^2\nu^n(dx)\leq \Eb^{\Pb^n}[\sup_{0\leq t\leq T}|X^n_t|^2]\leq C\left[1+\int_0^T\int_{\overline{\Xb}\times\Ab}|a|^2m^n_{t}(dx,da)dt\right].
\end{align}
Next, similar to the proof of Lemma \ref{lemma:D0} (i), we choose the test function $u(t,x) = \phi(x)$, where $\phi \in C_b^2(\overline{\mathbb{X}})$ satisfies $\nabla\phi(x) = m(x)$ on $\partial\mathbb{X}$. Substituting this $u$ into the linear programming constraint for the truncated problem yields 
\begin{align*}
    \lambda^n([0,T]\times\partial\mathbb{X}) &\le \int_{\overline{\mathbb{X}}} |\phi(x)|\,\nu^n(\mathrm{d}x) + \int_{\overline{\mathbb{X}}} |\phi(x)|\,m_0^*(\mathrm{d}x)  + \int_0^T \int_{\overline{\mathbb{X}}\times\mathbb{A}_n} |\mathcal{L}_n\phi(t,x,m^{n,\mathbb{X}}_t,a)| \, m^n_t(\mathrm{d}x,\mathrm{d}a)\mathrm{d}t.
\end{align*}
From the growth property of $(b_n,\sigma_n,\beta_n)$, we get
\[
    |\mathcal{L}_n\phi(t,x,\mu,a)| \le C\left(1+|x|+\left(\int_{\overline{\mathbb{X}}}|z|^2\mu(\mathrm{d}z)\right)^{1/2}+|a|\right),
\]
where the constant $C$ depends on $\phi$, $c_3$, and $\|\pi\|_{\infty}$. Applying this leads to the following inequality:
{\small\begin{align}\label{lambdan_estimate}
    \lambda^n([0,T]\times\partial\mathbb{X})
    &\le C\left(1 + \int_0^T\int_{\overline{\mathbb{X}}\times\mathbb{A}} \left(1+|x|+|a|+\sqrt{\|m^{n,\mathbb{X}}\|^2_{T}}\right) m^n_t(\mathrm{d}x,\mathrm{d}a)\mathrm{d}t\right) \\
    &\le C\left(1 + \int_0^T\int_{\overline{\mathbb{X}}\times\mathbb{A}} (1+|x|^2+|a|^2)m^n_t(\mathrm{d}x,\mathrm{d}a)\mathrm{d}t + \|m^{n,\mathbb{X}}\|_{T}^2\right)\nonumber \\
    &\le C\left(1 + \int_0^T\int_{\overline{\mathbb{X}}\times\mathbb{A}} |a|^2m^n_t(\mathrm{d}x,\mathrm{d}a)\mathrm{d}t + \|m^{n,\mathbb{X}}\|_{T}^2\right)\nonumber \\
    &\le C\left(1 + \int_0^T\int_{\overline{\mathbb{X}}\times\mathbb{A}} |a|^2m^n_t(\mathrm{d}x,\mathrm{d}a)\mathrm{d}t\right),\nonumber
\end{align}}
where in the last inequality we have used \eqref{mn_estimate}.

Fix $a_0\in\Ab_{n_0}$. For $n\geq n_0$, consider an admissible strict control $U^n= (\tilde{\Omega}^n, \tilde{\mathcal{F}}^n, \tilde{\mathbb{F}}^n, \tilde{\mathbb{P}}^n, \tilde{W}^n, \tilde{N}^n, \tilde{X}^n, \tilde{R}^n, \alpha)\in \mathcal{S}_n(m^{n,\Xb})$ in the sense of Definition \ref{def:mfe} with data $(b_n,\sigma_n,\beta_n,\Ab_n)$ and  $\tilde{\mathbb{P}}^n(\alpha_t=a_0 \,\text{for}\, a.e. t)=1$. Then,  Lemma \ref{lemma:A6} (i) implies 
\begin{align}\label{barmn_estimate}
    \Eb^{\tilde{\mathbb{P}}^n}[\sup_{0\leq t\leq T}|\tilde{X}^n_t|^2]\leq C(1+\|m^{n,\Xb}\|^2_T+T|a_0|^2)\leq C\left(1+\int_0^T\int_{\overline{\Xb}\times\Ab}|a|^2m^n_{t}(dx,da)dt\right),
    \end{align}
where $C$ is some constant depending only on $T$, $c_3$, $\|\pi\|_{\infty}$, $m^*_0$ and $a_0$. 

Now, we consider $(\bar{\nu}^n,\bar{m}^n,\bar{\lambda}^n)$ given by
{\small\begin{align*}
		&\bar{\nu}^n := \tilde{\mathbb{P}}^n \circ (\tilde{X}^n_T)^{-1}, \\
		&\bar{m}^n_t(B\times C) := \mathbb{E}^{\tilde{\mathbb{P}}^n}\!\left[\mathbf{1}_{B}(\tilde{X}^n_t)\mathbf{1}_{C}(\alpha_t)\right] = \mathbb{E}^{\tilde{\mathbb{P}}^n}\!\left[\mathbf{1}_{B}(\tilde{X}^n_{t-})\mathbf{1}_{C}(\alpha_t)\right], \quad B \in \mathcal{B}(\overline{\mathbb{X}}),\,\,C \in \mathcal{B}(\mathbb{A}),\,\ t-a.e., \\
		&\bar{\lambda}(D) := \mathbb{E}^{\tilde{\mathbb{P}}^n}\!\left[\int_0^T \mathbf{1}_{D}(t,\tilde{X}^n_t)\,d\tilde{R}^{n,c}_t\right] = \mathbb{E}^{\tilde{\mathbb{P}}^n}\!\left[\int_0^T \mathbf{1}_{D}(t,\tilde{X}^n_{t-})\,d\tilde{R}^{n,c}_t\right], \quad D \in \mathcal{B}([0,T]\times\partial\mathbb{X}).
	\end{align*}}
   By construction, $(\bar{\nu}^n,\bar{m}^n,\bar{\lambda}^n)\in \Dc^n_{P}(m^{n,\Xb})$.
Similar to \eqref{lambdan_estimate}, we have
{\small\begin{align}\label{barlambdan_estimate}
    \bar{\lambda}^n([0,T]\times\partial\mathbb{X})
    &\le C\left(1 + \int_0^T\int_{\overline{\mathbb{X}}\times\mathbb{A}} (|x|^2+|a|^2)\bar{m}^n_t(\mathrm{d}x,\mathrm{d}a)\mathrm{d}t + \|m^{n,\mathbb{X}}\|_{T}^2\right) \\
    &\le C\left(1 + \int_0^T\int_{\overline{\mathbb{X}}\times\mathbb{A}} |x|^2\bar{m}^n_t(\mathrm{d}x,\mathrm{d}a)\mathrm{d}t + |a_0|^2T+\|m^{n,\mathbb{X}}\|_{T}^2\right)\nonumber \\
    &\le C\left(1 + \int_0^T\int_{\overline{\mathbb{X}}\times\mathbb{A}} |a|^2m^n_t(\mathrm{d}x,\mathrm{d}a)\mathrm{d}t\right),\nonumber
\end{align}}where in the last inequality we have used \eqref{mn_estimate} and   \eqref{barmn_estimate}. Using the optimality of $(\nu^n,m^n,\lambda^n)$, the upper bounds of $\tilde{f}$, $h$, and $g$, and then \eqref{mn_estimate}, \eqref{nun_estimate}, \eqref{barmn_estimate} and \eqref{barlambdan_estimate} to get 
{\small\begin{align}\label{rhs}
    \Gamma[\nu^n,m^n](\nu^n,m^n,\lambda^n)&\leq  \Gamma[\nu^n,m^n](\bar{\nu}^n,\bar{m}^n,\bar{\lambda}^n)\nonumber\\
    &\leq C\bigg(\int_0^T\int_{\overline{\Xb}\times\Ab}(1+|x|^2+\|m^{n,\Xb}\|^2_{T}+|a_0|^q)\bar{m}^n_t(dx,da)dt+\bar{\lambda}^n([0,T]\times\partial\mathbb{X})\nonumber\\&\quad+\int_{\overline{\Xb}}(1+|x|^2+\int_{\overline{\Xb}}|z|^2\nu^n(dz))\bar{\nu}^n(dx)\bigg)\nonumber\\
    &\leq C\left(1 + \int_0^T\int_{\overline{\mathbb{X}}\times\mathbb{A}} |a|^2m^n_t(\mathrm{d}x,\mathrm{d}a)\mathrm{d}t\right).
\end{align}}
On the other hand, we may use the lower bounds on $\tilde{f}$, $h$, and $g$ along with \eqref{mn_estimate}, \eqref{nun_estimate} and \eqref{lambdan_estimate} to get
{\small\begin{align}\label{lhs}
    \Gamma[\nu^n,m^n](\nu^n,m^n,\lambda^n)&\geq -C\bigg(\int_0^T\int_{\overline{\Xb}\times\Ab}(1+|x|^2+\|m^{n,\Xb}\|^2_{T}+|a|^2)m^n_t(dx,da)dt+\lambda^n([0,T]\times\partial\mathbb{X})\nonumber\\&\quad+\int_{\overline{\Xb}}(1+|x|^2+\int_{\overline{\Xb}}|z|^2\nu^n(dz))\nu^n(dx)\bigg)+c_2\int_0^T\int_{\overline{\Xb}\times\Ab}|a|^qm^n_t(dx,da)dt\nonumber\\
    &\geq -C\left(1 + \int_0^T\int_{\overline{\mathbb{X}}\times\mathbb{A}} |a|^2m^n_t(\mathrm{d}x,\mathrm{d}a)\mathrm{d}t\right)+c_2\int_0^T\int_{\overline{\Xb}\times\Ab}|a|^qm^n_t(dx,da)dt.
\end{align}}Combining \eqref{rhs} and \eqref{lhs} and rearranging terms, we obtain the existence of constants $\kappa_1 \in \mathbb{R}$ and $\kappa_2 > 0$, independent of $n$, such that
\begin{align*}
    \int_0^T\int_{\overline{\Xb}\times\Ab}(|a|^q+\kappa_1|a|^2)m^n_t(dx,da)dt\leq \kappa_2.
\end{align*}
As $q>2$, it holds for all sufficiently large $|a|$ that $|a|^q + \kappa_1 |a|^2 \ge \frac{1}{2}|a|^q$. This implies that
\begin{align*}
    \sup_{n\geq 1}\int_0^T\int_{\overline{\Xb}\times\Ab}|a|^qm^n_t(dx,da)dt< \infty.
\end{align*}
 Combined with  Lemma \ref{lemma:A6} (i) and  \eqref{lambdan_estimate}, we establish \eqref{uni1}. 
For each $n \ge 1$, by reasoning as in the proof of \eqref{m_contin}, we have
\begin{align*}
    \int_0^{T-h}d_{\mathrm{BL}}(m^{n,\mathbb{X}}_{t+h}, m^{n,\mathbb{X}}_{t}) \, \mathrm{d}t\leq T\bigg(\sup_{t\in[0,T]}\Eb^{\Pb^n}\left[|X^n_{(t+h)\wedge T}-X^n_{t}|^2\right]\bigg)^{\frac{1}{2}}.
\end{align*}
Therefore, to establish \eqref{uni2}, it is sufficient to show that 
\begin{align}\label{suff_contin}
\lim\limits_{h\rightarrow0}\sup_{n\geq1}\sup_{t\in[0,T]}\Eb^{\Pb^n}\left[|X^n_{(t+h)\wedge T}-X^n_{t}|^2\right]=0.
\end{align}
To this end, we apply Corollary~\ref{corollary: A5} (ii). For any $t \in [0,T]$, combining with the compensation formula and the growth condition of coefficients yields
  {\small      \begin{align*}
            \mathbb{E}^{\mathbb{P}^n}\left[|X^n_{(t+h)\wedge T}-X^n_{t}|^2\right]&\leq C\bigg(\Eb^{\Pb^n}\bigg[\bigg(\int_t^{(t+h)\wedge T}|b_n(s,X^n_s,m^{n,\Xb}_s,a)|v^n_{s,X^n_s}(da)ds\bigg)^2\\&+\int_t^{(t+h)\wedge T}\int_{\Ab}|\sigma_n(s,X^n_s,m^{n,\Xb}_s,a)|^2v^n_{s,X^n_s}(da)ds\\&+\int_t^{(t+h)\wedge T}\int_{\Ab}|\beta_n(s,X^n_{s-},m^{n,\Xb}_s,a)|^2\pi(s,a)v^n_{s,X^n_{s-}}(da)ds\bigg]\bigg)\\
            &\leq C\bigg(\int_t^{(t+h)\wedge T}\int_{\overline{\Xb}\times\Ab}(1+|x|^2+|a|^2+ \|m^{n,\Xb}\|^2_{T})m^n_t(dx,da)dt\bigg).
        \end{align*}}
Applying \eqref{mn_estimate} and Hölder's inequality with conjugate exponents $q/2$ and $q/(q-2)$, 
{\small      \begin{align*}
            \mathbb{E}^{\mathbb{P}^n}\left[|X^n_{(t+h)\wedge T}-X^n_{t}|^2\right] 
            &\leq C\left(1+\int_0^T\int_{\overline{\Xb}\times\Ab}|a|^2m^n_{t}(dx,da)dt\right)h+Ch^{\frac{q-2}{q}}\bigg[\bigg(\int_0^{T}\int_{\overline{\Xb}\times\Ab}|x|^qm^n_t(dx,da)dt\bigg)^{\frac{2}{q}}\\&+\bigg(\int_0^{T}\int_{\overline{\Xb}\times\Ab}|a|^qm^n_t(dx,da)dt\bigg)^{\frac{2}{q}}\bigg].
        \end{align*}}
 Therefore, the right-hand side vanishes as $h \to 0$, uniformly in $n$ and $t$. This proves \eqref{suff_contin} and hence \eqref{uni2}.

 Finally, with \eqref{uni1} and \eqref{uni2} established, the relative compactness of the sequence $\{(\nu^n, m^n, \lambda^n)\}$ follows from arguments analogous to the proof of Lemma~\ref{lemma:D0} (i). Moreover, by the same reasoning in that proof, any limit point $(\nu^*, m^*, \lambda^*)$ of a convergent subsequence also satisfies the integrability conditions in \eqref{limit_bounds}. The estimate for $\|m^{*,\Xb}\|_T^q$ then follows directly from Lemma~\ref{lemma:A6} (i).
\end{proof}
Let $(\nu^*, m^*, \lambda^*)$ be the limit of a convergent subsequence of $\{(\nu^n, m^n, \lambda^n)\}$ in the topology $\tau_2 \otimes \overline{\tau}_2 \otimes \tau_0$.  As in Lemma~\ref{lemma:4.5}, we can also conclude that $m^{n,\Xb} \to m^{*,\Xb}$ in $\tilde{\tau}_2$.
\begin{lemma}\label{lemma52}
    The limit point $(\nu^*, m^*, \lambda^*)\in \Dc_{P}(m^{*,\Xb})$.
\end{lemma}
\begin{proof}
    Given the integrability conditions established in \eqref{limit_bounds}, it suffices to show that $(\nu^*, m^*, \lambda^*)$ satisfies the linear constraint \eqref{lpconstraint}.
    As $(\nu^n, m^n, \lambda^n)\in\mathcal{D}^n_{P}(m^{n,\Xb})$, for any $u\in C^{1,2}_b([0,T]\times \overline{\mathbb{X}})$, we have
	\begin{align*}
		\int_{\overline{\mathbb{X}}}u(T,x)\,\nu^n(dx)
		= \int_{\overline{\mathbb{X}}}u(0,x)\,m^*_{0}(dx)
		&+ \int_{[0,T]\times \overline{\mathbb{X}}\times\mathbb{A}} \bigl(\partial_t u+\Lc^n u\bigr)(t,x,m^{n,\Xb}_t,a)\,m^n_t(dx,da)dt \\
		& + \int_{[0,T]\times \partial \mathbb{X}}\mathcal{A}u(t,x)\,\lambda^n(dt,dx). \nonumber
	\end{align*}
    Therefore, to prove that $(\nu^*, m^*, \lambda^*)\in \Dc_{P}(m^{*,\Xb})$, it suffices to show that
\begin{align}\label{lemma4.12:limit}
  \lim_{n\rightarrow\infty}\int_{[0,T]\times \overline{\mathbb{X}}\times\mathbb{A}} \Lc^n u(t,x,m^{n,\Xb}_t,a)\,m^n_t(dx,da)dt=\int_{[0,T]\times \overline{\mathbb{X}}\times\mathbb{A}} \Lc u(t,x,m^{*,\Xb}_t,a)\,m^*_t(dx,da)dt.
\end{align}
Note that 
{\small\begin{align*}
   &\bigg |\int_{[0,T]\times \overline{\mathbb{X}}\times\mathbb{A}} \Lc^n u(t,x,m^{n,\Xb}_t,a)\,m^n_t(dx,da)dt-\int_{[0,T]\times \overline{\mathbb{X}}\times\mathbb{A}} \Lc u(t,x,m^{*,\Xb}_t,a)\,m^*_t(dx,da)dt\bigg|\\&\leq   \bigg |\int_{[0,T]\times \overline{\mathbb{X}}\times\mathbb{A}} \Lc^n u(t,x,m^{n,\Xb}_t,a)\,m^n_t(dx,da)dt-\int_{[0,T]\times \overline{\mathbb{X}}\times\mathbb{A}} \Lc u(t,x,m^{n,\Xb}_t,a)\,m^n_t(dx,da)dt\bigg| \\&+\bigg |\int_{[0,T]\times \overline{\mathbb{X}}\times\mathbb{A}} \Lc u(t,x,m^{n,\Xb}_t,a)\,m^n_t(dx,da)dt-\int_{[0,T]\times \overline{\mathbb{X}}\times\mathbb{A}} \Lc u(t,x,m^{*,\Xb}_t,a)\,m^*_t(dx,da)dt\bigg|.
\end{align*}}The second term on the right-hand side vanishes as $n\to\infty$, as in the proof of upper hemicontinuity in Lemma~\ref{lemma:D*}.
 For the first term, we have 
\begin{align*}
&\int_{[0,T]\times \overline{\mathbb{X}}\times\mathbb{A}} \Lc^n u(t,x,m^{n,\Xb}_t,a)\,m^n_t(dx,da)dt-\int_{[0,T]\times \overline{\mathbb{X}}\times\mathbb{A}} \Lc u(t,x,m^{n,\Xb}_t,a)\,m^n_t(dx,da)dt\\&\quad
=\int_{[0,T]\times \overline{\mathbb{X}}\times\mathbb{A}} (b_n(t,x,m^{n,\Xb}_t,a)-b(t,x,m^{n,\Xb}_t,a))\cdot\partial_x u(t,x)\,m^n_t(dx,da)dt\\
&\quad+\int_{[0,T]\times \overline{\mathbb{X}}\times\mathbb{A}} \frac{1}{2}\mathrm{tr}((\sigma_n\sigma^{\top}_n(t,x,m^{n,\Xb}_t,a)-\sigma\sigma^{\top}(t,x,m^{n,\Xb}_t,a))\partial^2_{xx} u(t,x))\,m^n_t(dx,da)dt\\&\quad+
\int_{[0,T]\times \overline{\mathbb{X}}\times\mathbb{A}} [u(t,\proj(x+\beta_n(t,x,m^{n,\Xb}_t,a))-u(t,\proj(x+\beta(t,x,m^{n,\Xb}_t,a))\\&\quad-(\beta_n(t,x,m^{n,\Xb}_t,a)-\beta(t,x,m^{n,\Xb}_t,a))\cdot\partial_x u(t,x)]\pi(t,a)\,m^n_t(dx,da)dt.
\end{align*}
We will show that the third term on the right-hand side of the above equality tends to $0$ as $n\to\infty$, and the convergence of other terms can be established by the same fashion. By construction, we have 
{\small
{\begin{align*}
&\left|\left[u(t,\proj(x+\beta_n(t,x,m^{n,\Xb}_t,a))-u(t,\proj(x+\beta(t,x,m^{n,\Xb}_t,a))-(\beta_n(t,x,m^{n,\Xb}_t,a)-\beta(t,x,m^{n,\Xb}_t,a))\cdot\partial_x u(t,x)\right]\right|\\&\leq C\left[|u(t,\proj(x+\beta_n(t,x,m^{n,\Xb}_t,a))-u(t,\proj(x+\beta(t,x,m^{n,\Xb}_t,a))| +|\beta_n(t,x,m^{n,\Xb}_t,a)-\beta(t,x,m^{n,\Xb}_t,a)|\right]\\&\leq 
    C\left[2+|\beta_n(t,x,m^{n,\Xb}_t,a)-\beta(t,x,m^{n,\Xb}_t,a)|\right]\indicator{|\beta(t,x,m^{n,\Xb}_t,a)|>n|}.
\end{align*}}}Under Assumption~\ref{c1}, the condition $|\beta(t,x,m^{n,\Xb}_t,a)| > n$ implies that $c_3(1+|x|+(\int_{\overline{\Xb}}|z|^2m^{n,\Xb}_t(dz))^{\frac{1}{2}}+|a|)>n$. Consequently,
{\small\begin{align*}
   \bigg| &\int_{[0,T]\times \overline{\mathbb{X}}\times\mathbb{A}}[u(t,\proj(x+\beta_n(t,x,m^{n,\Xb}_t,a))-u(t,\proj(x+\beta(t,x,m^{n,\Xb}_t,a))\\&\qquad\quad-(\beta_n(t,x,m^{n,\Xb}_t,a)-\beta(t,x,m^{n,\Xb}_t,a))\cdot\partial_x u(t,x)]\pi(t,a)\,m^n_t(dx,da)dt\bigg|\\&\leq C\int_{[0,T]\times \overline{\mathbb{X}}\times\mathbb{A}} \left(2+|\beta_n(t,x,m^{n,\Xb}_t,a)-\beta(t,x,m^{n,\Xb}_t,a)|\right)\indicator{|\beta(t,x,m^{n,\Xb}_t,a)|>n|}\,m^n_t(dx,da)dt\\
   &\leq C\int_{[0,T]\times \overline{\mathbb{X}}\times\mathbb{A}}\left(1+|x|+(\int_{\overline{\Xb}}|z|^2m^{n,\Xb}_t(dz))^{\frac{1}{2}}+|a|\right)\indicator{c_3(1+|x|+(\int_{\overline{\Xb}}|z|^2m^{n,\Xb}_t(dz))^{\frac{1}{2}}+|a|)>n}m^n_t(dx,da)dt\\
   &\leq \frac{C}{n}\int_{[0,T]\times \overline{\mathbb{X}}\times\mathbb{A}}\left(1+|x|+(\int_{\overline{\Xb}}|z|^2m^{n,\Xb}_t(dz))^{\frac{1}{2}}+|a|\right)^2m^n_t(dx,da)dt.
\end{align*}}
The uniform integrability property from~\eqref{uni1} ensures that the final term converges to $0$ as $n\to\infty$, which completes the proof.
\end{proof}
\begin{lemma}\label{lemma53}
    For each $(\tilde{\nu},\tilde{m},\tilde{\lambda})\in \Dc_{P}(m^{*,\Xb})$ such that $\Gamma[\nu^*,m^*](\tilde{\nu},\tilde{m},\tilde{\lambda})<\infty$, there exists $(\tilde{\nu}^n,\tilde{m}^n,\tilde{\lambda}^n)\in\Dc^n_{P}(m^{n,\Xb})$ such that 
    \begin{align}\label{lemma53:conver}
        \Gamma[\nu^*,m^*](\tilde{\nu},\tilde{m},\tilde{\lambda})=\lim_{n\rightarrow\infty}\Gamma[\nu^n,m^n](\tilde{\nu}^n,\tilde{m}^n,\tilde{\lambda}^n).
    \end{align}
\end{lemma}
\begin{proof}
    First, the lower bounds of $\tilde{f}$, $h$ and $g$ imply
    \begin{align*}
        \Gamma[\nu^*,m^*](\tilde{\nu},\tilde{m},\tilde{\lambda})&\geq c_2\int_{[0,T]\times \overline{\mathbb{X}}\times\Ab}|a|^q \tilde{m}_t(dx,da)dt -C \bigg(1+\int_{[0,T]\times \overline{\mathbb{X}}\times\Ab}(|x|^2+|a|^2)\tilde{m}_t(dx,da)dt\\
        &+\int_{\overline{\Xb}}|x|^2\tilde{\nu}(dx)+\int_{\overline{\Xb}}|x|^2\nu^*(dx)+\tilde{\lambda}([0,T]\times\partial\Xb)+\|m^{*,\Xb}\|_{T}^2\bigg).
    \end{align*}
    As $\|m^{*,\Xb}\|_{T}^2+\int_{\overline{\Xb}}|x|^2\nu^*(dx)<\infty$ by \eqref{limit_bounds}, and given that $(\tilde{\nu},\tilde{m},\tilde{\lambda})\in \Pc_2(\overline{\mathbb{X}}) \times V_2 \times \Mc_{+}([0,T]\times \partial\mathbb{X})$, the assumption that $\Gamma[\nu^*,m^*](\tilde{\nu},\tilde{m},\tilde{\lambda})<\infty$ implies
    \begin{align}\label{lemma53:bound}
        \int_{[0,T]\times \overline{\mathbb{X}}\times\Ab}|a|^q \tilde{m}_t(dx,da)dt <\infty.
    \end{align}
  Let $\tilde{v}_{t, x}(d a)$ be such that $
	\tilde{m}_t(d x, d a)dt=\tilde{v}_{t, x}(d a) \tilde{m}^{\Xb}_t(d x) d t.$
By Proposition~\ref{thm:prob-rep-R0}, there exists a filtered probability space 
\((\tilde{\Omega}, \tilde{\mathcal{F}}, \tilde{\mathbb{F}}, \tilde{\mathbb{P}})\) supporting $d$-dimensional orthogonal continuous $\tilde{\mathbb{F}}$-martingale measures $M = (M^1, \cdots, M^d)$ with intensity measure $\tilde{v}_{t, X_{t-}}(da)dt$, a random counting measure $\mathcal{N}$ on $[0,T] \times \mathbb{A}$ with compensator $\pi(t,a)\tilde{v}_{t, X_{t-}}(da)dt$,
and a pair of $\tilde{\mathbb{F}}$-progressively measurable processes $(X,R)$ taking values in $\overline{\mathbb{X}} \times \mathbb{R}_+$. 
Here, \(R\) is non-decreasing, RCLL, with \(R_0=0\). 
Moreover, \((X,R)\) satisfies
{\small\begin{eqnarray*}
	dX_t &=& \left(\int_{\mathbb{A}} b(t, X_t, m^{*,\Xb}_t, a)\, \tilde{v}_{t, X_t}(da)\right) dt 
	+ \int_{\mathbb{A}} \sigma(t, X_t, m^{*,\Xb}_t, a)\, M(dt,da) \\
	& &+ \int_{\mathbb{A}} \beta(t, X_{t-}, m^{*,\Xb}_t, a)\, \tilde{\mathcal{N}}(dt,da) 
	+ m(X_t)\, dR_t, 
	\quad \tilde{\mathbb{P}} \circ X_0^{-1} = m_0^*,
\end{eqnarray*}}
and $R_t = \int_0^t \mathbf{1}_{\{\partial\mathbb{X}\}}(X_s)  dR_s$, where \(\tilde{\mathcal{N}}(dt,da) := \mathcal{N}(dt,da) - \pi(t,a)\tilde{v}_{t, X_{t-}}(da)dt\)  
denotes the compensated random measure.  
Finally, the measures \(\tilde{\nu}\), \(\tilde{m}\), and \(\tilde{\lambda}\) have the following representation
\begin{align*}
	&\tilde{\nu} = \tilde{\mathbb{P}} \circ X_T^{-1}, \\
	&\tilde{m}_t(B \times C)
	= \mathbb{E}^{\tilde{\mathbb{P}}}\!\left[\mathbf{1}_{B}(X_{t})\, \tilde{v}_{t, X_{t}}(C)\right]=\mathbb{E}^{\tilde{\mathbb{P}}}\!\left[\mathbf{1}_{B}(X_{t-})\, \tilde{v}_{t, X_{t-}}(C)\right],
	\quad B \in \mathcal{B}(\overline{\mathbb{X}}),\; C \in \mathcal{B}(\mathbb{A}),\; t-a.e., \\
	&\tilde{\lambda}(D)
	= \mathbb{E}^{\tilde{\mathbb{P}}}\!\left[\int_0^T \mathbf{1}_{D}(t, X_{t})\, dR^c_t\right]=\mathbb{E}^{\tilde{\mathbb{P}}}\!\left[\int_0^T \mathbf{1}_{D}(t, X_{t-})\, dR^c_t\right],
	\quad D \in \mathcal{B}([0,T] \times \partial\mathbb{X}).
\end{align*}
Find a measurable mapping $\iota_n:\mathbb{A}\to\mathbb{A}_n$  such that $   \iota_n(\Ab)\subset\Ab_n$ and $\iota_n(a)=a$ for all $a\in \mathbb{A}_n$, so that $\iota_n$ converges pointwise to the indentity. On the same filtered probability space, let $(X^n, R^n)$ denote the unique strong solution to 
\begin{align*}
	X^n_t
	= X_0 
	&+ \int_0^{t} \int_{\mathbb{A}} b_n\bigl(s, X^n_s, m^{n,\mathbb{X}}_s, \iota_n(a)\bigr)\, \tilde{v}_{s, X_s}(da)\, ds
	+ \int_0^{t} \int_{\mathbb{A}} \sigma_n\bigl(s, X^n_s, m^{n,\mathbb{X}}_s, \iota_n(a)\bigr)\, M(ds, da) \nonumber \\
	& + \int_0^{t} \int_{\mathbb{A}} \beta_n\bigl(s, X^n_{s-}, m^{n,\mathbb{X}}_s, \iota_n(a)\bigr)\, {\tilde{\mathcal{N}}(ds, da)}
	+ \int_0^{t} m(X^n_s)\, dR^n_s,\quad X^n_0=X_0.
\end{align*}
For $(t,x)\in [0,T]\times\overline{\mathbb{X}}$, let $\tilde{v}^n_{t,x}$ be the image measure of $\tilde{v}_{t,x}$ under the map $a\mapsto\iota_n(a)$. We then define
\begin{align*}
    \tilde{m}_t^n(B \times C) := \mathbb{E}^{\tilde{\mathbb{P}}}\!\left[\mathbf{1}_{B}(X^n_t)\, \tilde{v}^n_{t, X_t}(C)\right], \quad
    \tilde{\nu}^n := \tilde{\mathbb{P}} \circ (X^n_T)^{-1}, \quad
    \tilde{\lambda}^n(D) := \mathbb{E}^{\tilde{\mathbb{P}}}\!\left[\int_0^T \mathbf{1}_{D}(t, X^n_t)\, dR^{n,c}_t\right].
\end{align*}
By construction, the triple $(\tilde{\nu}^n,\tilde{m}^n,\tilde{\lambda}^n)$ belongs to $\Dc^n_{P}(m^{n,\Xb})$. Moreover, it follows from Lemma \ref{lemma:A6} (i), \eqref{uni1} and \eqref{lemma53:bound} that 
\begin{align}\label{lemma53:xn}
    \Eb^{\tilde{\Pb}}\left[\sup_{0\leq t\leq T}|X^n_t|^q\right]&\leq C\left(1+\|m^{n,\Xb}\|_T^q+\Eb^{\tilde{\Pb}}\left[\int_0^T\int_{\Ab}|a|^q\tilde{v}^n_{t,X_t}(da)dt\right]\right)\\&=C\left(1+\|m^{n,\Xb}\|_T^q+\Eb^{\tilde{\Pb}}\left[\int_0^T\int_{\Ab}|\iota_n(a)|^q\tilde{v}_{t,X_t}(da)dt\right]\right)\nonumber\\
    &=C\left(1+\|m^{n,\Xb}\|_T^q+\int_{[0,T]\times \overline{\mathbb{X}}\times\Ab}|\iota_n(a)|^q \tilde{m}_t(dx,da)dt\right)\nonumber\\&
    \leq C\left(1+\sup_{n\geq1}\|m^{n,\Xb}\|_T^q+\int_{[0,T]\times \overline{\mathbb{X}}\times\Ab}|a|^q \tilde{m}_t(dx,da)dt\right)<\infty.\nonumber
\end{align}
This uniform bound implies the following moment estimates
\begin{align*}
    \int_0^T\int_{\overline{\Xb}\times\Ab}|x|^q+|a|^q \tilde{m}^n_t(dx,da)dt&= \Eb^{\tilde{\Pb}}\left[\int_0^T|X^n_t|^qdt+\int_0^T\int_{\Ab}|\iota_n(a)|^q\tilde{v}_{t,X_t}(da)dt\right]\\&\leq C\left(\sup_{n\geq1}\Eb^{\tilde{\Pb}}\left[\sup_{0\leq t\leq T}|X^n_t|^q\right]+\int_{[0,T]\times \overline{\mathbb{X}}\times\Ab}|a|^q \tilde{m}_t(dx,da)dt\right)<\infty,\\
    \int_{\overline{\Xb}}|x|^q\tilde{\nu}^n(dx)&\leq \sup_{n\geq1}\Eb^{\tilde{\Pb}}\left[\sup_{0\leq t\leq T}|X^n_t|^q\right]<\infty.
\end{align*}
 Similar to \eqref{lambdan_estimate}, we can also derive 
\begin{align*}
    \tilde{\lambda}^n([0,T]\times\partial\Xb)\leq C\left(1+\int_0^T\int_{\overline{\Xb}\times\Ab}(|x|^2+|a|^2)\tilde{m}^n_t(dx,da)dt+\|m^{n,\Xb}\|_{T}^2\right).
\end{align*}
These estimates collectively establish the uniform integrability of the sequence $(\tilde{\nu}^n,\tilde{m}^n,\tilde{\lambda}^n)$, namely
\begin{align}\label{uni3}
    \sup_{n\geq1}\int_{\overline{\Xb}}|x|^q\tilde{\nu}^n(dx)<\infty,\quad\sup_{n\geq1} \int_0^T\int_{\overline{\Xb}\times\Ab}|x|^q+|a|^q \tilde{m}^n_t(dx,da)dt<\infty,\quad \sup_{n\geq1} \tilde{\lambda}^n([0,T]\times\partial\Xb)<\infty.
\end{align}
Next, we show that $(\tilde{\nu}^n,\tilde{m}^n,\tilde{\lambda}^n)$ converges to $(\tilde{\nu},\tilde{m},\tilde{\lambda})$ in the topology $\tau_2 \otimes \overline{\tau}_2 \otimes \tau_0$. In light of the uniform integrability established in~\eqref{uni3}, it is sufficient to prove convergence in  $\tau_0 \otimes \tau_0 \otimes \tau_0$. The argument is analogous to that used to prove the lower hemicontinuity of $\mathcal{D}^*$ in Lemma~\ref{lemma:D*}. We begin by showing that $\tilde{m}^n \to \tilde{m}$ in $\tau_0$. 
To this end, let $\phi: [0, T] \times \overline{\mathbb{X}} \times \mathbb{A} \to \mathbb{R}$ be a bounded and Lipschitz continuous function.
 Then, we have
	\[
	\begin{aligned}
		&\left| \int_0^T \int_{\overline{\mathbb{X}} \times \mathbb{A}} \phi(t, x, a)\, \tilde{m}_t(dx, da)\,dt 
		- \int_0^T \int_{\overline{\mathbb{X}} \times \mathbb{A}} \phi(t, x, a)\, \tilde{m}_t^n(dx, da)\,dt \right| \\
		&= \left| \mathbb{E}^{\tilde{\mathbb{P}}} \left[ \int_0^T \int_{\mathbb{A}} \phi(t, X_t, a)\, \tilde{v}_{t,X_t}(da)\,dt
		- \int_0^T \int_{\mathbb{A}} \phi(t, X^n_t, \iota_n(a))\, \tilde{v}_{t,X_t}(da)\,dt \right]\right| \\
        &\leq C     \mathbb{E}^{\tilde{\mathbb{P}}} \left[\int_0^T|X^n_t-X_t|\,dt + \int_0^T \int_{\mathbb{A}} \left|a-\iota_n(a)\right|\, \tilde{v}_{t,X_t}(da)\,dt
		\right]\\
		&\leq C\bigg( \left( \mathbb{E}^{\tilde{\mathbb{P}}}\left[ \sup_{0 \leq t \leq T} |X_t - X^n_t|^2 \right] \right)^{1/2}+\int_{[0,T]\times\overline{\Xb}\times\Ab}|a-\iota_n(a)|\tilde{m}_t(dx,da)dt\bigg).
	\end{aligned}
	\]
For the second term on the right-hand side of the inequality, the pointwise convergence of $\iota_n(a)$ to $a$ and \eqref{lemma53:bound} allow us to apply the dominated convergence theorem, which implies that this term tends to $0$ as $n\to\infty$. Therefore, it suffices to show that
\begin{align}\label{convergence2}
	\mathbb{E}^{\tilde{\mathbb{P}}}\left[ \sup_{0 \leq t \leq T} |X_t - X^n_t|^2 \right] \to 0 \quad \text{as } n \to \infty,
\end{align}
which can be established by an argument analogous to that of \eqref{convergenceX} and \cite[Lemma 5.3]{lacker_mean_2015}. A similar argument yields the convergence $\tilde{\nu}_n\rightarrow\tilde{\nu}$. The convergence $\tilde{\lambda}_n\rightarrow\tilde{\lambda}$ follows from an argument analogous to that for $\lambda^n$ in the lower hemicontinuity proof of Lemma~\ref{lemma:D*}, combined with the fact that
\begin{align*}
  \lim_{n\rightarrow\infty}\int_{[0,T]\times \overline{\mathbb{X}}\times\mathbb{A}} \Lc^n u(t,x,m^{n,\Xb}_t,a)\,\tilde{m}^n_t(dx,da)dt=\int_{[0,T]\times \overline{\mathbb{X}}\times\mathbb{A}} \Lc u(t,x,m^{*,\Xb}_t,a)\,\tilde{m}_t(dx,da)dt,
\end{align*}
which can be proved similarly to \eqref{lemma4.12:limit}.

We are now ready to prove~\eqref{lemma53:conver}. The convergence of $\int_{\overline{\Xb}}g(x,\nu^n)\tilde{\nu}^n(dx)$ is a direct consequence of Lemma~\ref{lemma:technical} (i), while that of $\int_{[0,T]\times\partial\Xb}h(t,x)\tilde{\lambda}^n(dt,dx)$ follows from the weak convergence of $\tilde{\lambda}^n$. However, the term involving $\tilde{f}$ requires a separate argument, as $\tilde{f}$ does not satisfy the growth condition in Lemma~\ref{lemma:technical} (ii). Note that
\begin{align*}
    \int_{[0,T]\times\overline{\Xb}\times\Ab}\tilde{f}(t,x,m^{n,\Xb}_t,a)\,\tilde{m}_t^n(dx,da)dt=\mathbb{E}^{\tilde{\Pb}}\left[\int_0^T\int_{\Ab}\tilde{f}(t,X^n_t,m^{n,\Xb}_t,\iota_n(a))\,\tilde{v}_{t,X_t}(da)dt\right].
\end{align*}
As the $m^{n,\Xb} \to m^{*,\Xb}$ in $\tilde{\tau}_2$, along a subsequence, $m^{n,\Xb}_t$ converges to $m^{*,\Xb}_t$ in $\tau_2$ $t$-a.e. This result, combined with the continuity of $\tilde{f}$, its growth condition in Remark \ref{tfremark}, \eqref{lemma53:bound}, \eqref{lemma53:xn} and \eqref{convergence2}, and the pointwise convergence of $\iota_n(a)$, justifies the use of the dominated convergence theorem. We can thus conclude that
\begin{align*}
    \lim_{n\to\infty}\int_{[0,T]\times\overline{\Xb}\times\Ab}\tilde{f}(t,x,m^{n,\Xb}_t,a)\,\tilde{m}_t^n(dx,da)dt 
    & = \mathbb{E}^{\tilde{\Pb}}\left[\int_0^T\int_{\Ab}\tilde{f}(t,X_t,m^{*,\Xb}_t,a)\,\tilde{v}_{t,X_t}(da)dt\right] \\
    & = \int_{[0,T]\times\overline{\Xb}\times\Ab}\tilde{f}(t,x,m^{*,\Xb}_t,a)\,\tilde{m}_t(dx,da)dt,
\end{align*}
which completes the proof.


\end{proof}

\begin{proof}[Proof of Theorem \ref{thm:existence}]
    Let $(\tilde{\nu},\tilde{m},\tilde{\lambda})$ be an arbitrary element of $\Dc_{P}(m^{*,\Xb})$. Let $(\tilde{\nu}^n,\tilde{m}^n,\tilde{\lambda}^n)$ be the sequence constructed as in Lemma~\ref{lemma53}. The optimality of $(\nu^n,m^n,\lambda^n)$ for each $n$ implies that 
    \begin{align*}
        \Gamma[\nu^n,m^n](\nu^n,m^n,\lambda^n)\leq \Gamma[\nu^n,m^n](\tilde{\nu}^n,\tilde{m}^n,\tilde{\lambda}^n).
    \end{align*}
    Note that for any $k>0$, we have
    \begin{align*}
        \int_{[0,T]\times \overline{\mathbb{X}}\times\mathbb{A}}& (\tilde{f}\wedge k)(t,x,m^{n,\Xb}_t,a)\,m^n_t(dx,da)\,dt
        + \int_{[0,T]\times \partial\Xb} h(t,x)\,\lambda^n(dt,dx) + \int_{\overline{\mathbb{X}}} g(x,\nu^n)\,\nu^n(dx) \\
        &\leq \Gamma[\nu^n,m^n](\nu^n,m^n,\lambda^n) \leq \Gamma[\nu^n,m^n](\tilde{\nu}^n,\tilde{m}^n,\tilde{\lambda}^n).
    \end{align*}
    The function $\tilde{f}\wedge k$ satisfies the growth condition in Lemma~\ref{lemma:technical} (ii) (see also Remark~\ref{tfremark}). By applying Lemma~\ref{lemma:technical}, Lemma~\ref{lemma53}, and the uniform integrability condition~\eqref{uni1}, we can take the limit as $n\to\infty$ in the above inequality to obtain
    \begin{align*}
         \int_{[0,T]\times \overline{\mathbb{X}}\times\mathbb{A}} (\tilde{f}\wedge k)(t,x,m^{*,\Xb}_t,a)m^*_t(dx,da)dt
        & + \int_{[0,T]\times \partial\Xb} h(t,x)\lambda^*(dt,dx) + \int_{\overline{\mathbb{X}}} g(x,\nu^*)\nu^*(dx) \\&\leq \Gamma[\nu^*,m^*](\tilde{\nu},\tilde{m},\tilde{\lambda}).
    \end{align*}
   In view of the lower bound of $\tilde{f}$ in Remark~\ref{tfremark}  and the estimates in \eqref{limit_bounds}, we may apply the monotone convergence theorem as $k\to\infty$ to obtain
    \begin{equation*}
       \Gamma[\nu^*,m^*](\nu^*,m^*,\lambda^*)\leq \Gamma[\nu^*,m^*](\tilde{\nu},\tilde{m},\tilde{\lambda}).
    \end{equation*}
    Because $(\tilde{\nu},\tilde{m},\tilde{\lambda})$ is arbitrary, this proves the optimality of $(\nu^*,m^*,\lambda^*)$. It follows that $(\nu^*,m^*,\lambda^*)$ is an LPMFE, which, by Theorem~\ref{theorem:equi}, induces an MFE with a Markovian relaxed control.
\end{proof}

\section{A Numerical Example}\label{sect:5}
In this section, we present a simple numerical example to illustrate the application of the linear programming approach in the inventory management in a mean-field competition context. Consider a large population of competing firms, each managing the inventory of a single product. The state process $X_t$ represents the inventory level of a representative firm at time $t$. The inventory is constrained to lie within a warehouse of finite capacity, hence the state space is the bounded domain $\overline{\mathbb{X}} = [0, X_{\text{max}}]$, where $X_{\text{max}}>0$ is the maximum capacity. The firm's control, $\alpha_t \in \mathbb{A} = [0, A_{\max}]$, is its production rate. The dynamics of the inventory level $X_t$ are described by the following reflected SDE with jumps:
\begin{equation*}
	dX_t = (\alpha_t - D(\mu_t))\,dt + \sigma_0\,dW_t - \delta X_{t-}\,d\tilde{N}_t + m(X_t)\,dR_t,
\end{equation*}
where the coefficients are specified as follows:
\begin{itemize}
	\item The drift coefficient is $b(t, x, \mu, \alpha) = \alpha - D(\mu)$. Here, $D(\mu)$ is the market demand rate. We model the mean-field interaction through the demand, positing a market saturation effect:
	\[
	D(\mu_t) = D_0 - c \int_{\overline{\mathbb{X}}} x \,d\mu_t(x),
	\]
	where $D_0 \geq 0$ is the base demand and $c>0$ is a competition parameter. As the average market inventory increases, the demand for each firm's product decreases.
	
	\item A constant diffusion coefficient, $\sigma(t,x,\mu,\alpha) = \sigma_0 > 0$, models random fluctuations in inventory or demand.
	
	\item The jump coefficient is $\beta(x) = -\delta x$ for some fixed $\delta \in (0,1)$. The jumps, driven by a compensated Poisson process $\tilde{N}_t$, represent sudden spoilage events where a fraction $\delta$ of the current inventory is lost.  This choice ensures that if $X_{t-} \in \overline{\mathbb{X}}$, then the post-jump state $X_{t-} + \beta(X_{t-})$ also lies in $\overline{\mathbb{X}}$. Consequently, the reflection process $R_t$ is continuous.
	
	\item The reflection process $R_t$ activates at the boundaries $\partial\mathbb{X} = \{0, X_{\text{max}}\}$. At the lower boundary $x=0$, $R_t$ measures the cumulative unmet demand (lost sales) due to a stockout. At the upper boundary $x=X_{\text{max}}$, $R_t$ measures the cumulative excess inventory that must be discarded due to full capacity.
\end{itemize}
The representative agent's objective is to choose a production strategy $\alpha$ that minimizes the cost functional:
\[
 \mathbb{E}\Biggl[
\int_{0}^{T} \left(\frac{k}{2}\alpha_t^2 + c_h X_t\right)dt
\;+\;
\int_{0}^{T} h(X_t)\,d R_{t}
\;+\;
p X_{T}
\Biggr],
\]
where the cost functional involves the following components:
\begin{itemize}
	\item The running cost is $f(x, \alpha) = \frac{k}{2}\alpha^2 + c_h x$, comprising a quadratic production cost ($k>0$) and a linear inventory holding cost ($c_h > 0$).
	
	\item The reflection cost $h(x) = C_{\text{stockout}}\mathbf{1}_{\{x=0\}} + C_{\text{overflow}}\mathbf{1}_{\{x=X_{\text{max}}\}}$. The constants $C_{\text{stockout}} > 0$ and $C_{\text{overflow}} > 0$ represent the per-unit penalty for lost sales and for discarding excess inventory, respectively. 
	
	 \item The terminal term $g(x) = p x$ represents a cost for inventory disposal or, if $p<0$, a salvage value for the remaining inventory at time $T$.
\end{itemize}

We propose the following algorithm for computing the LPMFE, similar to \cite{dumitrescu2023linear}.
\begin{algorithm}[H] 
    \caption{Linear programming fictitious play (LPFP) algorithm} 
     \textbf{Data:} A number of steps $N$ for the equilibrium approximation; a initial triple $(\bar{\nu}^{(0)}, \bar{m}^{(0)},\bar{\lambda}^{(0)})$;
     
         \textbf{Result:} Approximate LPMFE

         \textbf{For:} $\ell=0,1, \ldots, N-1$ \textbf{do}
    \begin{algorithmic}[1]
            \Statex Compute a best response $(\nu^{(\ell+1)}, m^{(\ell+1)},\lambda^{(\ell+1)})$ to $(\bar{\nu}^{(\ell)}, \bar{m}^{(\ell)},\bar{\lambda}^{(\ell)})$ by solving the linear programming problem
             $$\underset{(\nu,m,\lambda)\in \mathcal{D}_P(\bar{m}^{(\ell,\Xb}))}{\arg\min}\Gamma[\bar{\nu}^{(\ell)}, \bar{m}^{(\ell)}](\nu, m,\lambda).$$
            \Statex Set $(\bar \nu^{(\ell+1)}, \bar m^{(\ell+1)}, \bar\lambda^{(\ell+1)}) := \frac{\ell}{\ell+1}(\bar \nu^{(\ell)}, \bar m^{(\ell)}, \bar \lambda^{(\ell)}) + \frac{1}{\ell+1}(\nu^{(\ell+1)}, m^{(\ell+1)}, \lambda^{(\ell+1)})$
    \end{algorithmic}
    \textbf{end}
\end{algorithm}
To apply the LPFP algorithm, we discretize the underlying model in a way similar to \cite{dumitrescu2023linear,cho2002linear, mendiondo1998approximation}.  More precisely, we consider a time grid $t=i\Delta$ with $\Delta=\frac{T}{n_t}$, for $i\in\{0,1,\cdots, n_{t}\}$, a state grid $ 0= x_0< x_1< \cdots < x_{n_s} = X_{\text{max}}$ with $\delta=\frac{X_{\text{max}}}{n_s}$ and an action grid  $a_0< a_1< \cdots < a_{n_a}$. The initial distribution $m^*_0$ and the initial guess for the triple $(\bar{\nu}, \bar{m}, \bar{\lambda})$ are discretized on this grid structure.

 Without loss of generality, we treat all the coefficients as independent of the mean-field term, because it is fixed during the best-response computation. Then, the generator $\mathcal{L}$  is given by
\begin{align*}
\mathcal{L}u(t, x, a)=\frac{\partial u}{\partial t}(t, x) &+ (b(t, x, a)-\beta(t,x,a)\pi(t,a))\frac{\partial u}{\partial x}(t, x) + \frac{\sigma^2}{2}(t, x, a)\frac{\partial^2 u}{\partial x^2}(t, x), \\&+(u(t,\Pi_{\overline{\mathbb{X}}}\bigl(x + \beta(t,x,a)\bigr)) - u(t,x))\pi(t,a)\quad \forall u\in C^{1, 2}_b([0, T]\times \bar{\mathbb{X}}).
\end{align*}
We set $\mathcal{D}(\hat L)$ as the functions in $C^{1, 2}_b([0, T]\times \bar{\mathbb{X}})$  restricted to the time-state grid.  For $u \in \mathcal{D}(\hat{L})$, we discretized the derivatives as follows:
\begin{align*}
\hat{L}_t u(t_i, x_j, a_k) &= \frac{u(t_{i+1}, x_j) - u(t_i, x_j)}{\Delta}, \\
\hat{L}_x^u u(t_i, x_j, a_k) &= \frac{\max(\tilde{b}(t_i, x_j, a_k), 0)}{\delta} \left(u(t_{i+1}, x_{j+1}) - u(t_{i+1}, x_{j})\right), \\
\hat{L}_x^d u(t_i, x_j, a_k) &= \frac{\min(\tilde{b}(t_i, x_j, a_k), 0)}{\delta} \left(u(t_{i+1}, x_{j}) - u(t_{i+1}, x_{j-1})\right), \\
\hat{L}_{xx} u(t_i, x_j, a_k) &= \frac{\sigma^2(t_i, x_j, a_k)}{2\delta^2} \left(u(t_{i+1}, x_{j+1}) + u(t_{i+1}, x_{j-1}) - 2u(t_{i+1}, x_j)\right),
\end{align*}
where $\tilde{b}(t,x,a) = b(t, x, a) - \beta(t,x,a)\pi(t,a)$. The jump component is discretized using a linear interpolation. Let $x'_j = \Pi_{\bar{\mathbb{X}}}\bigl(x_j + \beta(t_i, x_j, a_k)\bigr)$ be the post-jump position and $l = \lfloor (x'_j - x_0)/\delta \rfloor$ be the index of the grid point immediately to its left. Then,
\begin{align*}
\hat{L}_{J}u(t_i, x_j, a_k) = \pi(t_i,a_k) \left( \frac{x'_j - x_l}{\delta} u(t_{i+1},x_{l+1}) + \frac{x_{l+1} - x'_j}{\delta} u(t_{i+1},x_{l}) - u(t_{i+1},x_{j}) \right).
\end{align*}
The discretized generator $\hat L$ takes the form:
$$\hat L u (t_i, x_j, a_k) = \hat L_t u (t_i, x_j, a_k) + \hat L_x^u u (t_i, x_j, a_k) + \hat L_x^d u (t_i, x_j, a_k) + \hat L_{xx} u (t_i, x_j, a_k)+\hat{L}_{J}(t_i, x_j, a_k).$$
Similarly, the boundary operator $\Ac$ is discretized as:
\[
\hat{A}u(t_i, x_j) = \frac{u(t_{i+1},x_1)-u(t_{i+1},x_0)}{\delta}\indicator{j=0} - \frac{u(t_{i+1},x_{n_s})-u(t_{i+1},x_{n_{s-1}})}{\delta}\indicator{j=n_s}.
\]
The constraint reads as
\begin{align*}
\sum_{j=0}^{n_s} u(t_{n_t}, x_j)\hat\nu(x_j) &- \Delta \sum_{i=0}^{n_t-1}\sum_{j=1}^{n_s-1}\sum_{k=0}^{n_a} \hat{L} u(t_i, x_j, a_k)\hat m(t_i, x_j, a_k) \\
&-\Delta\sum_{i=0}^{n_t-1}\sum_{j \in \{0, n_s\}} \hat{A}u(t_{i},x_j)\hat\lambda(t_i,x_j) = \sum_{j=0}^{n_s}u(t_0, x_j) m^*_0(x_j),
\end{align*}
for $u\in \mathcal{D}(\hat L)$. The set $\mathcal{D}(\hat L)$ is equal to the linear span of the indicators functions
$$\mathbf{1}_{\{(t_i, x_j)\}}, \quad i\in\{0, 1, \cdots, n_t\}, \; j\in\{0, 1, \cdots, n_s\}$$
on the time-state grid. By linearity, it suffices to evaluate the constraint on the set of indicator functions, which yields a total of $(n_t+1) \times (n_s+1)$ linear constraints. Finally, the discretized cost is given by
\[
\sum_{j=0}^{n_s} g(x_j)\hat\nu(x_j) + \Delta \sum_{i=0}^{n_t-1} \sum_{j \in \{0, n_s\}} h(t_i,x_j)\hat{\lambda}(t_i,x_j) + \Delta \sum_{i=0}^{n_t-1}\sum_{j=1}^{n_s-1}\sum_{k=0}^{n_a} \tilde{f}(t_i, x_j, a_k) \hat m(t_i, x_j, a_k).
\]
In order to evaluate the convergence of the algorithm, we compute the distance between the best response $(\nu^{(N)}, m^{(N)}, \lambda^{(N)})$ computed at the current iteration and the measures $(\bar{\nu}^{(N-1)}, \bar{m}^{(N-1)}, \bar{\lambda}^{(N-1)})$ from the previous iteration:
\begin{align*}
\epsilon_N := &\max \bigg\{ \Delta \sum_{i, j, k} |m^{(N)}(t_i, x_j, a_k) - \bar{m}^{(N-1)}(t_i, x_j, a_k)|, \sum_{j} |\nu^{(N)}(x_j) - \bar{\nu}^{(N-1)}(x_j)|, \\
    & \Delta \sum_{i=0}^{n_t-1} \left( |\lambda^{(N)}(t_i,x_0)-\bar{\lambda}^{(N-1)}(t_i, x_0)| + |\lambda^{(N)}(t_i,x_{n_s})-\bar{\lambda}^{(N-1)}(t_i, x_{n_{s}})| \right)
\bigg\}.
\end{align*}
In this example, we take the initial distribution as the law $\mathcal{N}(2, 0.1)$ truncated to $\Xb$, and the parameters are summarized in Table~\ref{tab:params}.
\begin{table}[htbp]
\centering
\caption{Model Parameters for the Numerical Example}
\label{tab:params}
\begin{tabular}{@{}lcccccc@{}} 
\toprule
\textbf{Parameter} & $X_{\text{max}}$ & $A_{\text{max}}$ & $D_0$ & $c$ & $\sigma_0$ & $\delta$ \\
\textbf{Value}     & 4 & 1 & 0.5 & 0.1 & 1 & 0.1 \\
\midrule
\textbf{Parameter} & $\pi$ & $k$ & $c_h$ & $C_{\text{stockout}}$ & $C_{\text{overflow}}$ & $p$ \\
\textbf{Value}     & 0.5 & 2 & 0.1 & 5 & 1 & -0.1 \\
\bottomrule
\end{tabular}
\end{table}
In Figure~\ref{fig:mfg_dist_results} (left graph), we observe the evolution of the distribution of the representative player $m_t$ over time, while Figure~\ref{fig:mfg_dist_results} (right) shows the terminal distribution $\nu$ of the representative player. Figure~\ref{fig:mfg_control_results} plots the average control given by  $(t_i, x_j) \mapsto \sum_{k} a_k v_{t_i, x_j}(a_k)$. As expected, the representative player with a lower inventory level chooses a higher production rate. Finally, Figure~\ref{fig:mfg_Convergence_results} illustrates the convergence of the algorithm.

\begin{figure}[htbp] 
    \centering 
    \includegraphics[width=\textwidth]{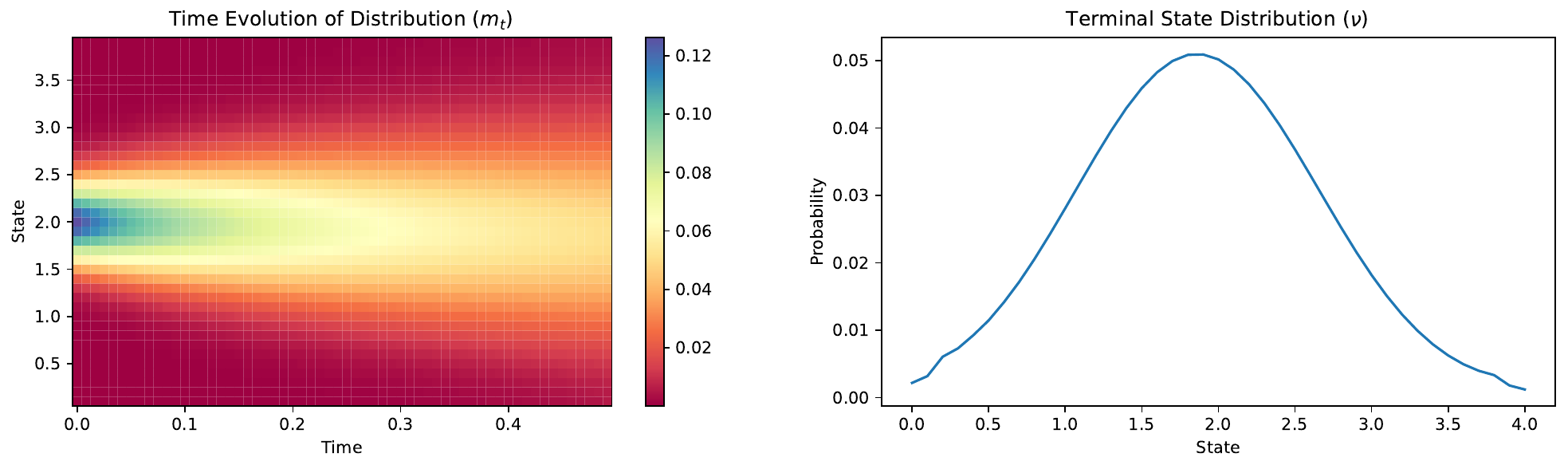}
    \caption{Evolution of the  distribution ($m_t$) and the terminal distribution ($\nu$).}
    \label{fig:mfg_dist_results} 
\end{figure}

\begin{figure}[h] 
    \centering 
    \includegraphics[width=0.45\textwidth]{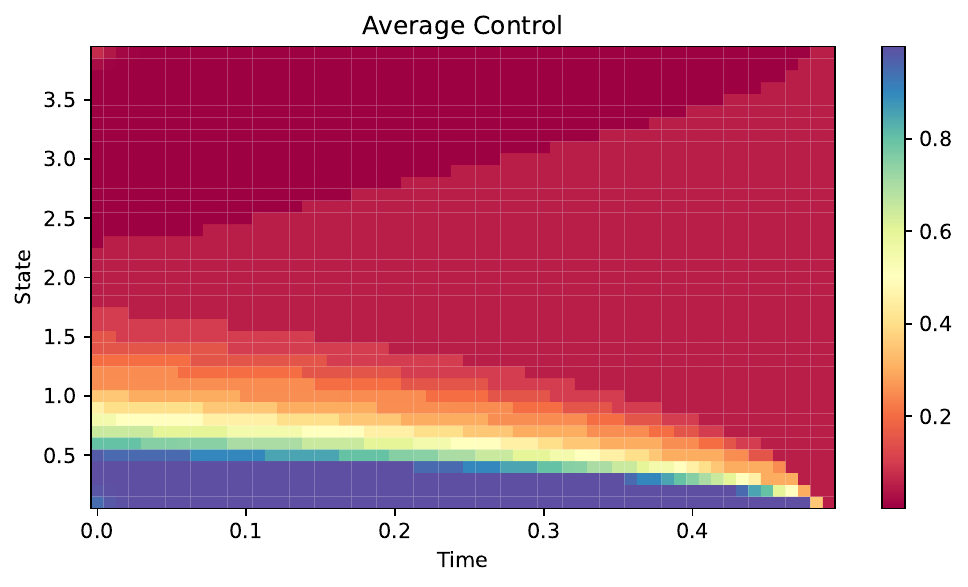}
    \caption{Average control at equilibrium.}
    \label{fig:mfg_control_results} 
\end{figure}

\begin{figure}[h] 
    \centering 
    \includegraphics[width=0.45\textwidth]{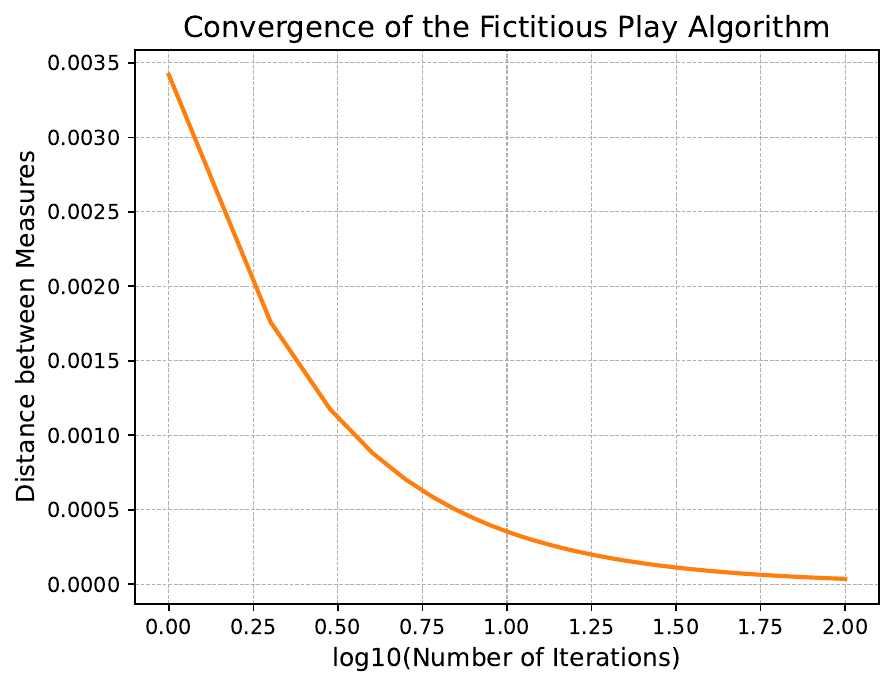}
    \caption{Convergence of the algorithm.}
    \label{fig:mfg_Convergence_results} 
\end{figure}

\appendix		
\section{The Skorokhod Problem for Semimartingales}
This section reviews some results on the Skorokhod problem.  Let $D$ be a domain in $\mathbb{R}^d$ with nonempty interior. Define the set $\mathscr{N}_x$ of inward normal unit vectors at $x \in \partial D$, i.e.,  $\eta \in \mathscr{N}_x$ iff for any $y \in D,\langle x-y, \eta\rangle \leq 0$. 

\begin{definition}
 Let $y \in \mathbb{D}\left([0,T], \mathbb{R}^d\right)$ such that $y_0 \in \bar{D}$. We  say that a pair $(x, k) \in \mathbb{D}\left([0,T], \mathbb{R}^{2 d}\right)$ is a solution of the Skorokhod problem associated with $y$ if
 \begin{itemize}
     \item[(i)] $x_t=y_t+k_t, t \in [0,T]$.
     \item[(ii)] $x_t \in \bar{D}, t \in [0,T]$.
     \item[(iii)] $k$ is a function with bounded variation  such that $k_0=0$ and
$$
k_t=\int_0^t \boldsymbol{n}_s \mathrm{~d}|k|_s, \quad|k|_t=\int_0^t \mathbf{1}_{\left\{x_s \in \partial D\right\}} \mathrm{d}|k|_s,
$$
where $\boldsymbol{n}_s \in \mathscr{N}_{x_s}$ if $x_s \in \partial D$ and $|k|_t$ denotes the total variation of $k$ on $[0, t], t \in [0,T]$.
 \end{itemize}
\end{definition}
\begin{lemma}\label{lemma:newA1}
    Let $y$ and $\tilde{y}$ be elements of $\mathbb{D}\left([0, T], \mathbb{R}^d\right)$ such that $y_0 , \tilde{y}_0\in \bar{D}$ . Let $(x, k)$ and $(\tilde{x}, \tilde{k})$ be the solutions to the Skorokhod problem associated with  $y$ and $\tilde{y}$, respectively. Then, for every $t \in[0,T]$, we have
$$
\left|x_t-\tilde{x}_t\right|^2 \leq\left|y_t-\tilde{y}_t\right|^2+2 \int_0^t\left\langle y_t-y_s+\tilde{y}_s-\tilde{y}_t, \mathrm{d}\left(k_s-\tilde{k}_s\right)\right\rangle.
$$
\end{lemma}
\begin{proof}
    See \cite[Lemma 2.2]{Tanaka1979StochasticDE}.
\end{proof}
Let $(\Omega, \mathcal{F}, \mathbb{F}, \mathbb{P})$ be a filtered probability space satisfying the usual conditions.
\begin{definition}
Let $Y$ be an $\mathbb{F}$-adapted càdlàg process with $Y_0 \in \bar{D}$. A pair of $\mathbb{F}$-adapted processes $(X, K)$ is a solution to the Skorokhod problem associated with $Y$ if, for $\mathbb{P}$-almost every $\omega \in \Omega$, $(X(\omega), K(\omega))$ is a solution to the Skorokhod problem associated with $Y(\omega)$.
\end{definition}
 Let $(X, K)$ be the solution to the Skorokhod problem associated with a semimartingale $Y$ of the form
$$
Y_t = Y_0 + M_t + V_t, \quad t \in [0,T],
$$
where $M$ is an $\mathbb{F}$-adapted local martingale and $V$ is an $\mathbb{F}$-adapted process of bounded variation, satisfying $M_0=V_0=0$. Similarly, let $(X', K')$ be the solution to the Skorokhod problem for another semimartingale $Y'$ with the decomposition
$$
Y'_t = Y'_0 + M'_t + V'_t, \quad t \in [0,T].
$$	
\begin{lemma}\label{lemma:comparison}
    For every $p\geq1$, there exists a constant $C_p>0$ such that for every $\Fb$ stopping time $\tau$,
    \begin{align}\label{inequ:compare}
        \Eb[\sup_{0\leq t\leq\tau}|X_t-X'_t|^{2p}]\leq C_p\Eb\left[|Y_0-Y'_0|^{2p}+[M-M']^p_{\tau}+|V-V'|^{2p}_{\tau}\right],
    \end{align}
    where $[M-M']_t:=\sum_{i=1}^{d} [M^{i}-M^{i'}]_t$ denotes the quadratic variation  and $|V-V'|_{t}:=\sum_{i=1}^{d}|V^{i}-V^{i'}|_{t}$ is the total variation on $[0, t]$.
\end{lemma}
\begin{proof}
    By Lemma \ref{lemma:newA1}, we have 
    $$
\left|X_t-X'_t\right|^2 \leq\left|Y_t-Y'_t\right|^2+2 \int_0^t\left\langle Y_t-Y_s+Y'_s-Y'_t, \mathrm{d}\left(K_s-K'_s\right)\right\rangle.
$$
  Applying the Itô' formula yields
{\small\begin{align*} 
     2 \int_0^t\left\langle Y_t-Y_s+Y'_s-Y'_t, \mathrm{d}\left(K_s-K'_s\right)\right\rangle
   &= 2 \int_0^t \langle K_{s-} - K'_{s-}, d(Y_s - Y'_s) \rangle \\
    &= 2 \int_0^t \langle X_{s-} - X'_{s-}, d(Y_s - Y'_s) \rangle - 2 \int_0^t \langle Y_{s-} - Y'_{s-}, d(Y_s - Y'_s) \rangle\\
    &= 2 \int_0^t \langle X_{s-} - X'_{s-}, d(Y_s - Y'_s) \rangle - |Y_t - Y'_t|^2 + |Y_0 - Y'_0|^2 + [Y - Y']_t.
\end{align*}}
Therefore, it holds that
\begin{align*}
    \left|X_t-X'_t\right|^2 \leq 2 \int_0^t \langle X_{s-} - X'_{s-}, d(Y_s - Y'_s) \rangle  + |Y_0 - Y'_0|^2 + [Y - Y']_t.
\end{align*}
 It then follows that for any $p \ge 1$, there exists a constant $C_p > 0$ such that
\begin{align*}
     \left|X_t-X'_t\right|^{2p} \leq C_p\left( \left|\int_0^t \langle X_{s-} - X'_{s-}, d(Y_s - Y'_s) \rangle\right|^p  + |Y_0 - Y'_0|^{2p} + [Y - Y']^p_t\right).
\end{align*}
The rest of the proof follows the same arguments in \cite[Theorem 1]{Słomiński1994}.
\end{proof}
\begin{corollary}\label{corollary: A5}
\begin{enumerate}
    \item[\textit{(i)}]  For any $p \ge 1$ and any stopping time $\tau$, there exists a constant $C_p>0$ such that
    \begin{equation*}
        \mathbb{E}\left[\sup_{0 \le t \le \tau} |X_t|^{2p}\right] \le C_p \left(1+ \mathbb{E}\left[|Y_0|^{2p} +[M]_{\tau}^p  +  |V|_{\tau}^{2p} \right] \right).
    \end{equation*}
    \item[\textit{(ii)}]  For any $p \ge 1$ and any fixed times $0 \le t \le s \le T$, there exists a constant $C_p > 0$ such that
    \begin{equation*}
        \mathbb{E}\left[ \sup_{t \le u \le s} |X_u - X_t|^{2p} \right] \le C_p \mathbb{E}\left[ ([M]_s - [M]_t)^p + (|V|_s - |V|_t)^{2p} \right].
    \end{equation*}
\end{enumerate}
\end{corollary}
\begin{proof}
(i) Let $x_0 \in \bar{D}$ be an arbitrary fixed point. We apply Lemma~\ref{lemma:comparison} with the constant path $Y'_t \equiv x_0$, for which the unique solution is $(X', K') = (x_0, 0)$, to obtain the bound:
\begin{align*}
      \Eb[\sup_{0\leq t\leq\tau}|X_t-x_0|^{2p}]\leq C_p\Eb\left[|Y_0-x_0|^{2p}+[M]^p_{\tau}+|V|^{2p}_{\tau}\right].
\end{align*}
 The desired estimate then follows from the inequality $|X_t|^{2p} \le C'_p(|X_t - x_0|^{2p} + |x_0|^{2p})$.

(ii) Fix $0 \le t \le s \le T$ and let $Y'_u := Y_{u \wedge t}$. The corresponding reflected process is $X'_u = X_{u \wedge t}$. We apply Lemma~\ref{lemma:comparison} with the stopping time $\tau = s$. The left-hand side of the \eqref{inequ:compare} becomes
\[
    \mathbb{E}\left[ \sup_{0 \le u \le s} |X_u - X_{u \wedge t}|^{2p} \right] = \mathbb{E}\left[ \sup_{t \le u \le s} |X_u - X_t|^{2p} \right].
\]
For the right-hand side, we note that $Y_0 - Y'_0 = 0$ and
\begin{align*}
    [M-M']_s = [M - M_{\cdot \wedge t}]_s = [M]_s - [M]_t, \quad
    |V-V'|_s = |V - V_{\cdot \wedge t}|_s = |V|_s - |V|_t.
\end{align*}
Substituting these into \eqref{inequ:compare} yields the desired result:
\begin{equation*}
    \mathbb{E}\left[ \sup_{t \le u \le s} |X_u - X_t|^{2p} \right] \le C_p \mathbb{E}\left[ ([M]_s - [M]_t)^p + (|V|_s - |V|_t)^{2p} \right]. \qedhere
\end{equation*}
\end{proof}
		\begin{lemma}\label{lemma:A6}
Suppose that Assumption~\ref{c1} holds. For any measurable flow $\mu\in B_q$,  where $q$ is the constant specified in this assumption, let $U\in \mathcal{R}(\mu)$ be an admissible relaxed control in the sense of Definition~\ref{weak_relax}. Then the following statements hold:
\begin{itemize}
    \item[(i)] For any $\bar{q} \in \{2, q\}$, there exists a positive constant $C=C(T,c_3,m_0^*,\bar{q},\|\pi\|_{\infty})$, where $c_3$ is the constant from Assumption~\ref{c1}, such that 
    \begin{equation}\label{app1:estimate}
\mathbb{E}^{\mathbb{P}}\left[\left(\sup_{t\in[0,T]}|X_t|\right)^{\bar{q}}\right] \le C\left(1 + \|\mu\|^{\bar{q}}_{T} + \mathbb{E}^{\mathbb{P}}\left[\int_0^T\int_{\mathbb{A}}|a|^{\bar{q}}\Lambda_t(\mathrm{d}a)\mathrm{d}t\right]\right).
    \end{equation}
   In particular, if $\Pb\circ X^{-1}_t=\mu_t$ for a.e. $t\in[0,T]$, then 
    \begin{align}\label{appA:mu}
        \|\mu\|^{\bar{q}}_T\leq C\left(1 + \mathbb{E}^{\mathbb{P}}\left[\int_0^T\int_{\mathbb{A}}|a|^{\bar{q}}\Lambda_t(\mathrm{d}a)\mathrm{d}t\right]\right).
    \end{align}
     Furthermore, if the coefficients $b$, $\sigma$, and $\beta$ are bounded, then \eqref{app1:estimate} simplifies to $$\mathbb{E}^{\mathbb{P}}\left[\left(\sup_{t\in[0,T]}|X_t|\right)^{\bar{q}}\right] \le C,$$ where $C$ is a constant that depends only on $m_0^*$, $\bar{q}$,  $\|b\|_{\infty}$, $\|\sigma\|_{\infty}$, $\|\beta\|_{\infty}$, and $\|\pi\|_{\infty}$.
    
    \item[(ii)]  If the coefficients $b$, $\sigma$, and $\beta$ are bounded, then for any $h\in[0,T]$ and $t\in[0,T-h]$, there exists a constant $C>0$, depending only on $\|b\|_{\infty}$, $\|\sigma\|_{\infty}$, $\|\beta\|_{\infty}$, and $\|\pi\|_{\infty}$, such that
    \begin{equation}\label{app1:estimate2}
        \mathbb{E}^{\mathbb{P}}\left[|X_{t+h}-X_{t}|^2\right] \le C h.
    \end{equation}
\end{itemize}
\end{lemma}
	\begin{proof}
	    (i) We first focus on the case $\bar{q}=q\geq4$. It follows from Corollary \ref{corollary: A5} (i) that 
{\small\begin{align}\label{appA:1}
\mathbb{E}^{\mathbb{P}}\left[\left(\sup_{s\in[0,t]}|X_s|\right)^{q}\right]\leq &C_{q}\bigg(1+ \int_{\overline{\Xb}}|z|^{q}m_0^{*}(dz)+ \Eb^{\mathbb{P}}\bigg[\left(\int_0^t\int_{\Ab}|b(s,X_s,\mu_s,a)|\Lambda_s(da)ds\right)^q\nonumber\\&+
\left(\int_0^t\int_{\Ab}|\sigma(s,X_s,\mu_s,a)|^2\Lambda_s(da)ds+\int_0^t\int_{\Ab}|\beta(s,X_{s-},\mu_s,a)|^2\mathcal{N}(ds,da)\right)^{\frac{q}{2}} \bigg] \bigg)\nonumber\\
\leq &C_{q}\bigg(1+ \int_{\overline{\Xb}}|z|^{q}m_0^{*}(dz)+ \Eb^{\mathbb{P}}\bigg[\left(\int_0^t\int_{\Ab}|b(s,X_s,\mu_s,a)|\Lambda_s(da)ds\right)^q\\&+
\left(\int_0^t\int_{\Ab}|\sigma(s,X_s,\mu_s,a)|^2\Lambda_s(da)ds\right)^{\frac{q}{2}}+\left(\int_0^t\int_{\Ab}|\beta(s,X_{s-},\mu_s,a)|^2\mathcal{N}(ds,da)\right)^{\frac{q}{2}} \bigg] \bigg).\nonumber
\end{align}}
By Jensen’s inequality and using the growth conditions on $b$, we have for all $t\in[0,T]$:
{\small\begin{align*}
\Eb^{\mathbb{P}}\left[\left(\int_0^t\int_{\Ab}|b(s,X_s,\mu_s,a)|\Lambda_s(da)ds\right)^q\right]&\leq C(T)\Eb^{\mathbb{P}}\left[\int_0^t\int_{\Ab}(\sup_{0\leq u\leq s}|b(u,X_u,\mu_u,a)|)^q\Lambda_s(da)ds\right]\\
&\leq C(T,c_3,q)\Eb^{\Pb}\left[\int_0^t\int_{\Ab}\left(1+(\sup_{0\leq u\leq s}|X_u|)^q+\|\mu\|_s^q+|a|^q\right)\Lambda_s(da)ds\right].
\end{align*}}
Similarly,  for the diffusion term,
{\small\begin{align*}
\Eb^{\mathbb{P}}\left[\left(\int_0^t\int_{\Ab}|\sigma(s,X_s,\mu_s,a)|^2\Lambda_s(da)ds\right)^{\frac{q}{2}}\right]\leq C(T,c_3,q)\Eb^{\Pb}\left[\int_0^t\int_{\Ab}\left(1+(\sup_{0\leq u\leq s}|X_u|)^q+\|\mu\|_s^q+|a|^q\right)\Lambda_s(da)ds\right].
\end{align*}}
Lastly, for the jump term
{\small\begin{align*}
    \Eb^{\mathbb{P}}&
\left[\left(\int_0^t\int_{\Ab}|\beta(s,X_{s-},\mu_s,a)|^2\mathcal{N}(ds,da)\right)^{\frac{q}{2}}\right]\\&=\Eb^{\mathbb{P}}\left[\left(\int_0^t\int_{\Ab}|\beta(s,X_{s-},\mu_s,a)|^2\tilde{\mathcal{N}}(ds,da)+\int_0^t\int_{\Ab}|\beta(s,X_{s-},\mu_s,a)|^2\pi(s,a)\Lambda_s(da)ds\right)^{\frac{q}{2}}\right]\\&\leq C_q\Eb^{\mathbb{P}}\left[\left(\int_0^t\int_{\Ab}|\beta(s,X_{s-},\mu_s,a)|^2\tilde{\mathcal{N}}(ds,da)\right)^{\frac{q}{2}}+\left(\int_0^t\int_{\Ab}|\beta(s,X_{s-},\mu_s,a)|^2\pi(s,a)\Lambda_s(da)ds\right)^{\frac{q}{2}}\right].
\end{align*}}
For the first term on the right-hand side of the above inequality, {we apply a Burkholder-Davis-Gundy type inequality to the stochastic integral involving $\tilde{\mathcal{N}}$ (see, e.g., \cite[Theorem 4.4.23]{Applebaum2009}). This step requires that $q \ge 4$,} which yields
{\small\begin{align*}
\Eb^{\mathbb{P}}\left[\left(\int_0^t\int_{\Ab}|\beta(s,X_{s-},\mu_s,a)|^2\tilde{\mathcal{N}}(ds,da)\right)^{\frac{q}{2}}\right]&\leq C_q\bigg(\Eb^{\mathbb{P}}\left[\left(\int_0^t\int_{\Ab}\sup_{0\leq u\leq  s}|\beta(u,X_{u-},\mu_u,a)|^4\|\pi\|_{\infty}\Lambda_s(da)ds\right)^{\frac{q}{4}}\right]\\&\quad+ \Eb^{\mathbb{P}}\left[\int_0^t\int_{\Ab}\sup_{0\leq u\leq s}|\beta(u,X_{u-},\mu_u,a)|^q\|\pi\|_{\infty}\Lambda_s(da)ds\right] \bigg).
\end{align*}}
For the second term, we have 
{\small\begin{equation*}
\Eb^{\mathbb{P}}\left[\left(\int_0^t\int_{\Ab}|\beta(s,X_{s-},\mu_s,a)|^2\pi(s,a)\Lambda_s(da)ds\right)^{\frac{q}{2}}\right]\leq \Eb^{\mathbb{P}}\left[\left(\int_0^t\int_{\Ab}\sup_{0\leq u\leq  s}|\beta(u,X_{u-},\mu_u,a)|^2\|\pi\|_{\infty}\Lambda_s(da)ds\right)^{\frac{q}{2}}\right].
\end{equation*}}
Therefore, by applying Hölder's inequality, we arrive at
{\small\begin{align*}
    &\Eb^{\mathbb{P}}
\left[\left(\int_0^t\int_{\Ab}|\beta(s,X_{s-},\mu_s,a)|^2\mathcal{N}(ds,da)\right)^{\frac{q}{2}}\right]\\&\leq C(T,q,\|\pi\|_{\infty})\Eb^{\mathbb{P}}\left[\int_0^t\int_{\Ab}\sup_{0\leq u\leq  s}|\beta(u,X_{u-},\mu_u,a)|^q\Lambda_s(da)ds\right]\\
    &\leq C(T,c_3,q,\|\pi\|_{\infty}) \Eb^{\Pb}\left[\int_0^t\int_{\Ab}\left(1+(\sup_{0\leq u\leq s}|X_u|)^q+\|\mu\|_s^q+|a|^q\right)\Lambda_s(da)ds\right].
\end{align*}}
Combining all the previous estimates, we get the existence of a positive constant $C=C(T,c_3,q,\|\pi\|_{\infty})$ such that 
\begin{align*}
    \mathbb{E}^{\mathbb{P}}\left[\left(\sup_{s\in[0,t]}|X_s|\right)^{q}\right]\leq C\left(1+\int_{\overline{\Xb}}|z|^{q}m_0^{*}(dz)+\Eb^{\Pb}\left[\int_0^t\int_{\Ab}\left((\sup_{0\leq u\leq s}|X_u|)^q+\|\mu\|_s^q+|a|^q\right)\Lambda_s(da)ds\right]\right).
\end{align*}
Therefore, \eqref{app1:estimate} follows from the Gronwall’s inequality. Noting that by \eqref{appA:1}, we can also derive 
{\small\begin{align*}
    \sup_{s\in[0,t]} \mathbb{E}^{\mathbb{P}}\left[|X_s|^{q}\right]&\leq \mathbb{E}^{\mathbb{P}}\left[\left(\sup_{s\in[0,t]}|X_s|\right)^{q}\right]\\&
    \leq C\left(1+\int_{\overline{\Xb}}|z|^{q}m_0^{*}(dz)+\int_0^t\left(\sup_{u\in[0,s]}\Eb^{\Pb}[|X_u|^q]+\|\mu\|_s^q\right)ds+\Eb^{\Pb}\left[\int_0^t\int_{\Ab}|a|^q\Lambda_s(da)ds\right]\right).
\end{align*}}
If  $\Pb\circ X^{-1}_t=\mu_t$ for a.e. $t\in[0,T]$, then the above becomes
\begin{align*}
    \|\mu\|^q_t\leq C\left(1+\int_{\overline{\Xb}}|z|^{q}m_0^{*}(dz)+2\int_0^t\|\mu\|_s^qds+\Eb^{\Pb}\left[\int_0^t\int_{\Ab}|a|^q\Lambda_s(da)ds\right]\right).
\end{align*}
The estimate \eqref{appA:mu} now also follows from Gronwall’s inequality.  The final claim for the bounded case follows by a similar argument, replacing the linear growth of the coefficients with the boundedness condition.

For the case of $\bar{q}=2$, we follow the same procedure, noting that the Burkholder-Davis-Gundy type inequality is not required. Instead, the linear growth of $\beta$, together with the moment bound from \eqref{app1:estimate}, ensures that the stochastic integral with respect to the compensated measure $\tilde{\mathcal{N}}$ is a true martingale. This justifies a direct application of the compensation formula, from which we obtain
\begin{align*}
    \mathbb{E}^{\mathbb{P}}\left[ \int_0^t\int_{\mathbb{A}}|\beta(s,X_{s-},\mu_s,a)|^2\,\mathcal{N}(\mathrm{d}s,\mathrm{d}a) \right] 
    = \mathbb{E}^{\mathbb{P}}\left[ \int_0^t\int_{\mathbb{A}}|\beta(s,X_{s-},\mu_s,a)|^2\pi(s,a)\Lambda_s(\mathrm{d}a)\mathrm{d}s \right].
\end{align*}
All the required estimates then follow in a similar fashion.

        (ii) By Corollary \ref{corollary: A5} (ii), we have for any $h\in[0,T]$ and $t\in[0,T-h]$,
      {\small  \begin{align*}
            \mathbb{E}^{\mathbb{P}}\left[|X_{t+h}-X_{t}|^2\right] &\leq C\bigg(\Eb^{\Pb}\bigg[\bigg(\int_t^{t+h}|b(s,X_s,\mu_s,a)|\Lambda_s(da)ds\bigg)^2+\int_t^{t+h}\int_{\Ab}|\sigma(s,X_s,\mu_s,a)|^2\Lambda_s(da)ds\\&\qquad+\int_t^{t+h}\int_{\Ab}|\beta(s,X_{s-},\mu_s,a)|^2\mathcal{N}(ds,da)\bigg]\bigg)\\&=C\bigg(\Eb^{\Pb}\bigg[\bigg(\int_t^{t+h}|b(s,X_s,\mu_s,a)|\Lambda_s(da)ds\bigg)^2+\int_t^{t+h}\int_{\Ab}|\sigma(s,X_s,\mu_s,a)|^2\Lambda_s(da)ds\\&\qquad+\int_t^{t+h}\int_{\Ab}|\beta(s,X_{s-},\mu_s,a)|^2\pi(s,a)\Lambda_s(da)ds\bigg]\bigg).
        \end{align*}}
       Then, if the coefficients $b$, $\sigma$, and $\beta$ are bounded, the inequality \eqref{app1:estimate2} follows.
	\end{proof}

	  \section{Auxiliary Results}
	\label{sec:aux_results}
\begin{lemma}\label{lemma:boundary_construction}
	Let $\mathbb{X} \subset \mathbb{R}^n$ be a bounded domain with $C^2$ boundary $\partial\mathbb{X}$.  
	Then there exists a function $\phi \in C_b^2(\overline{\mathbb{X}})$ such that 
	\[
	\nabla \phi(x) = m(x), \quad x \in \partial\mathbb{X},
	\]
	where $m(x)$ denotes the inward unit normal vector at  $x\in\partial\mathbb{X}$.
\end{lemma}

	\begin{proof}
	Let $d(x)$ be the signed distance function to the boundary $\partial\Xb$, defined by
	\[
	d(x) = \begin{cases} 
		\inf_{y \in \partial\Xb} \|x - y\| & \text{if } x \in \Xb, \\
		-\inf_{y \in \partial\Xb} \|x - y\| & \text{if } x \notin \Xb.
	\end{cases}
	\]
	Because the boundary $\partial\Xb$ is of class $C^2$, it is a standard result that there exists a constant $\delta > 0$ such that $d(x)$ is a $C^2$ function in the tubular neighborhood $N_{\delta} = \{x \in \Rb^d \mid |d(x)| < \delta\}$.  In particular, on the boundary, we have $\nabla d(x)=m(x)$. Let $\eta: \mathbb{R} \to [0, 1]$ be a smooth cutoff function satisfying
		\[
		\eta(t) =
		\begin{cases}
			1 & \text{for } |t| \le \delta/2 ,\\
			0 & \text{for } |t| \ge \delta.
		\end{cases}
		\]
	Let us define the function $\phi: \bar{\mathbb{X}} \to \mathbb{R}$ by $\phi(x) := \eta(d(x)) d(x)$ and verify that $\phi$ has the required properties.
		\begin{enumerate}
			\item \textit{Regularity:} The function $\phi$ is constructed from compositions and products of $C^2$ functions in the region where $\delta/2 < d(x) < \delta$. It smoothly transitions to $\phi(x) = d(x)$ (which is $C^2$) for $0 \leq d(x) \leq \delta/2$, and to $\phi(x) = 0$ (which is $C^\infty$) for $d(x) \geq \delta$. Therefore, $\phi \in C^2(\overline{\mathbb{X}})$. As $\overline{\mathbb{X}}$ is compact, $\phi$ and its first two derivatives are bounded, implying $\phi \in C^2_b(\overline{\mathbb{X}})$.
			
			\item \textit{Boundary Condition:} For any $x_0 \in \partial \mathbb{X}$, we have $d(x_0)=0$, $\eta(0)=1$, and $\eta'(0)=0$. Therefore
			\begin{align*}
				\nabla\phi(x_0) &= \nabla\left[\eta(d(x))d(x)\right]\bigg|_{x=x_0} = \left[ \eta'(d(x_0))d(x_0)\nabla d(x_0) + \eta(d(x_0))\nabla d(x_0) \right] \\
				&= \left[ \eta'(0) \cdot 0 \cdot m(x_0) + \eta(0) \cdot m(x_0) \right] = [0 + 1 \cdot m(x_0)] = m(x_0).
			\end{align*}
		\end{enumerate}
	\end{proof}
\begin{lemma}\label{lemma:boundary_equal}
Let $\mathbb{X}\subset\mathbb{R}^n$ be a bounded domain with $C^3$ boundary $\partial\mathbb{X}$, and let $\lambda_1$ and $\lambda_2$ be finite signed  measures on $[0,T]\times\partial\mathbb{X}$.
	If, for every $u \in C_b^{1,2}([0,T] \times \overline{\mathbb{X}})$, it holds that
	\[
	\int_{[0,T] \times \partial\mathbb{X}} \partial_x u(t,x) \cdot m(x) \, \lambda_1(dt,dx)
	=
	\int_{[0,T] \times \partial\mathbb{X}} \partial_x u(t,x) \cdot m(x) \, \lambda_2(dt,dx),
	\]
	then $\lambda_1 = \lambda_2$.
\end{lemma}
\begin{proof}
It suffices to show that 
\[
V = \left\{ \frac{\partial u}{\partial m}\bigg|_{[0,T]\times\partial \Xb} : u\in C_b^{1,2}([0,T]\times\overline{\Xb}) \right\}
\]
is dense in \(C([0,T]\times \partial\Xb)\).  
We will prove this by showing that \(V\) contains the set \(C^{1,2}([0,T]\times \partial\mathbb{X})\),  
which is known to be dense in \(C([0,T]\times \partial\mathbb{X})\) by the Stone--Weierstrass theorem.

To start, fix \(\varphi\in C^{1,2}([0,T]\times \partial\mathbb{X})\).  
In view that \(\partial\mathbb{X}\) is \(C^3\), there exists a \(\delta>0\) such that the signed distance function
\[
d(x) = 
\begin{cases} 
	\inf_{y \in \partial\Xb} \|x - y\|, & \text{if } x \in \Xb, \\[0.3em]
	-\inf_{y \in \partial\Xb} \|x - y\|, & \text{if } x \notin \Xb
\end{cases}
\]
is \(C^3\) in the tubular neighborhood 
\[
N_\delta := \{ x \in \mathbb{R}^n : |d(x)| < \delta \}.
\]
Moreover, there exists a \emph{closest-point projection mapping} \(p: N_\delta \to \partial\mathbb{X}\) that 
\(p(x) := \operatorname*{argmin}_{y \in \partial\mathbb{X}} |x-y|\), which is \(C^2\) for \(C^3\) boundaries. 
Let \(\eta \in C^\infty(\mathbb{R})\) be a smooth cutoff such that \(\eta(s) = 1\) for \(|s| \le \delta/2\) and \(\eta(s) = 0\) for \(|s| \ge \delta\). We now define our test function \(u\) by
\[
u(t,x) :=
\begin{cases}
	\eta(d(x)) \, \varphi(t,p(x))\, d(x), & x \in N_\delta \cap \overline{\mathbb{X}}, \\[0.3em]
	0, & x \in \overline{\mathbb{X}} \setminus N_\delta.
\end{cases}
\]
It is standard to verify that \(u \in C_b^{1,2}([0,T]\times \overline{\mathbb{X}})\). Moreover,
at a point \((t,x_0) \in [0,T]\times \partial \Xb\), we have
\begin{align*}
	\frac{\partial u}{\partial m}(t,x_0) 
	&= \nabla_x u(t,x_0) \cdot m(x_0) \\
	&= \big( \nabla_x \big(\eta(d(x))  \varphi(t,p(x))\big) \, d(x) 
	+ \eta(d(x))  \varphi(t,p(x)) \nabla_x d(x) \big)\big|_{x=x_0} \cdot m(x_0) \\
	&= \varphi(t,x_0)\, |m(x_0)|^2 = \varphi(t,x_0),
\end{align*}
where we have used that \(d(x_0)=0\) and \(\nabla_x d(x_0) = m(x_0)\). Thus, for any \(\varphi \in C^{1,2}([0,T]\times \partial\mathbb{X})\), there exists a function \(u\) in our test class such that \(\frac{\partial u}{\partial m} = \varphi\). Therefore,
\[
C^{1,2}([0,T]\times \partial\mathbb{X}) \subset V,
\]
and \(V\) is dense in \(C([0,T]\times \partial\mathbb{X})\).

\end{proof}

\begin{lemma}\label{lemma:separable}
	The spaces \(C_b^2(\mathbb{R}^d)\) and \(C_b^1(\mathbb{R}_{+})\) are separable under the topology of bounded pointwise convergence.
\end{lemma}

 \begin{proof}
 	We shall only prove the case \(C_b^2(\mathbb{R}^d)\) and the argument for \(C_b^1(\mathbb{R}_{+})\) is similar. 
 	We will construct a countable set $\mathcal{S}$ and show that for any $\phi \in C^2_b(\mathbb{R}^d)$, there exists a sequence $\{s_k\} \subset \mathcal{S}$ converging to $\phi$ in the sense of  bounded pointwise convergence.
 	
 	Let $\psi: \mathbb{R} \to \mathbb{R}$ be a smooth function with $\psi(t) = 1$ for $|t| \le 1$ and $\psi(t) = 0$ for $|t| \ge 2$. Define a radial cutoff function $\Psi: \mathbb{R}^d \to \mathbb{R}$ by $\Psi(x) = \psi(|x|)$ and its scaled versions $\Psi_m(x) = \Psi(x/m)$ for $m \in \mathbb{N}$.
 	Let $\mathcal{P} = \mathbb{Q}[x_1, \cdots, x_d]$ be the countable set of polynomials in $d$ variables with rational coefficients. We define the countable set by
 	$$ \mathcal{S} = \{ p(x) \cdot \Psi_m(x) \mid p \in \mathcal{P}, m \in \mathbb{N} \}. $$
 	Each element of $\mathcal{S}$ has a compact support and is thus in $C^2_b(\mathbb{R}^d)$. Let $\phi \in C^2_b(\mathbb{R}^d)$ be an arbitrary function. For each $k \in \mathbb{N}$, define a compactly supported function $\phi_k(x) = \phi(x)\Psi_k(x)$. The support of \(\phi_k\) is contained in the compact ball \(K_k := \overline{B(0, 2k)}\).
    
 	 We first show that there exists a sequence of functions \(s_k \in \mathcal{S}\) such that $\|s_k - \phi_k\|_{C^2} < 1/k$.
 To prove this claim, fix $k \in \mathbb{N}$. By the Stone-Weierstrass theorem, we can obtain that for any $\epsilon_k > 0$, there exists $p_k \in \mathcal{P}$ such that $\sup_{x \in K_k}|D^\alpha \phi(x) - D^\alpha p_k(x)| < \epsilon_k$ for all multi-indices $|\alpha| \le 2$. Let $s_k(x) := p_k(x)\Psi_k(x)$, which is an element of $\mathcal{S}$. Then there exists a constant $C$ (depending only on the cutoff function $\Psi$) such that:
 $ \|s_k - \phi_k\|_{C^2} < C\epsilon_k. $
 By choosing $\epsilon_k = 1/(kC)$, we achieve the desired bound $\|s_k - \phi_k\|_{C^2} < 1/k$.   
 	We now show that the sequence \(\{s_k\}_{k=1}^\infty\) converges to \(\phi\) in the topology of bounded pointwise convergence.
 	\begin{enumerate}
 		\item \textit{Pointwise Convergence:} For any fixed $x \in \mathbb{R}^d$ and any multi-index $\alpha$ with $|\alpha| \le 2$, we examine the error $|D^\alpha s_k(x) - D^\alpha \phi(x)|$. Using the triangle inequality:
 		$$ |D^\alpha s_k(x) - D^\alpha \phi(x)| \le |D^\alpha s_k(x) - D^\alpha \phi_k(x)| + |D^\alpha \phi_k(x) - D^\alpha \phi(x)|. $$
 		The first term is bounded by $\|s_k - \phi_k\|_{C^2} < 1/k$, which tends to 0.
 		For the second term, choose an integer $M > |x|$. For all $k > M$, we have $\Psi_k(x) = 1$ and all its derivatives are zero at $x$, which implies $D^\alpha\phi_k(x) = D^\alpha\phi(x)$, so this term is zero for all $k > M$.
 		Thus, the total error tends to 0 as $k \to \infty$.
 		
 		\item \textit{Uniform Boundedness:} 
 		$$ \|s_k\|_{C^2} \le \|s_k - \phi_k\|_{C^2} + \|\phi_k\|_{C^2} < \frac{1}{k} + \|\phi_k\|_{C^2}. $$
 		The sequence of norms $\{\|\phi_k\|_{C^2}\}$ is itself bounded. By the Leibniz rule, $D^\alpha \phi_k$ is a sum of terms like $(D^{\alpha-\beta}\phi)(D^\beta \Psi_k)$. The norms $\|D^\beta \Psi_k\|_\infty$ are bounded, independently of $k$ (as $\|D^\beta \Psi_m\|_\infty \sim 1/m^{|\beta|}$). Therefore, there exists a constant $M_\phi$ such that $\|\phi_k\|_{C^2} \le M_\phi$ for all $k$.
 		This implies that $\|s_k\|_{C^2} < 1/k + M_\phi \le 1 + M_\phi$ for all $k \ge 1$. The sequence $\{s_k\}$ is therefore uniformly bounded in the $C^2$ norm.
 	\end{enumerate}
 	Note that we have constructed a sequence from the countable set $\mathcal{S}$ converging to an arbitrary $\phi \in C^2_b(\mathbb{R}^d)$, it follows that $\mathcal{S}$ is dense and the space is separable.
 \end{proof}

\section{Collected Conditions and Results from the Literature}\label{sec:Condition}
For the reader’s convenience, we collect in this section some conditions and results from the literature that have been used in the main body of the paper. Let 
$\mathcal{L}, \mathcal{A}: \mathcal{D} \subset C_b(E) \to C(E \times U)$ 
for some complete, separable metric spaces $E$ and $U$.  We introduce the following condition from \cite{kurtz_stationary_2001}:
\begin{condition}\label{cond:B1}
\begin{itemize}
\item[(i)] $\mathcal{L}, \mathcal{A}: \mathcal{D} \subset C_b(E) \to C(E \times U)$, with $1 \in \mathcal{D}$, and $\mathcal{L} 1 = 0$, $\mathcal{A} 1 = 0$.
\item[(ii)] There exist $\psi_0, \psi_1 \in C(E \times U)$ with $\psi_0, \psi_1 \ge 1$, and constants $a_f, b_f$ (depending on $f \in \mathcal{D}$), such that
\[
|\mathcal{L} f(x,u)| \le a_f \psi_0(x,u), \qquad 
|\mathcal{A} f(x,u)| \le b_f \psi_1(x,u), \quad \forall (x,u) \in E \times U.
\]
\item[(iii)] Defining 
$(\mathcal{L}^0, \mathcal{A}^0) := \{(f, \psi_0^{-1} \mathcal{L} f, \psi_1^{-1} \mathcal{A} f) : f \in \mathcal{D}\}$,
 $(\mathcal{L}^0, \mathcal{A}^0)$ is separable in the sense that there exists a countable collection $\{g_k\} \subset \mathcal{D}$ such that $(\mathcal{L}^0, \mathcal{A}^0)$ is contained in the bounded, pointwise closure of the linear span of 
$\{(g_k, \psi_0^{-1} \mathcal{L} g_k, \psi_1^{-1} \mathcal{A} g_k)\}$.
\item[(iv)] For each $u \in U$, the operators $\mathcal{L}_u$ and $\mathcal{A}_u$, defined by $\mathcal{L}_u f(x) = \mathcal{L} f(x,u)$ and $\mathcal{A}_u f(x) = \mathcal{A} f(x,u)$, are pre-generators.
\item[(v)] $\mathcal{D}$ is closed under multiplication and separates points.
\end{itemize}
\end{condition}
We also provide a sufficient condition for an operator to be a pre-generator; the proof can be found in Remark 1.1 of \cite{kurtz_stationary_2001}.
\begin{prop}\label{prop:B2}
If $\Lc \subset C_b(E) \times C_b(E)$ and for each $x \in E$, there exists a solution $\nu^x$ of the forward equation for $(\Lc, \delta_x)$ that is right-continuous (in the weak topology) at zero, then $\Lc$ is a pre-generator.
\end{prop}
Finally, we give a technical result established in \cite{dumitrescu2023linear}, which is frequently used in the proofs of our main results.
\begin{lemma}\label{lemma:technical}
\begin{itemize}
    \item[(i)] Let $\mathcal{X}$ and $\mathcal{Y}$ be complete separable metric spaces and let $\varphi: \mathcal{X} \times \mathcal{Y} \rightarrow \mathbb{R}$ be continuous and satisfying the following growth condition: there exist $c \geq 0$ and $\left(x_0, y_0\right) \in \mathcal{X} \times \mathcal{Y}$ such that for all $(x, y) \in \mathcal{X} \times \mathcal{Y}$
$$
|\varphi(x, y)| \leq c\left(1+d_{\mathcal{X}}\left(x, x_0\right)^p+d_{\mathcal{Y}}\left(y, y_0\right)^p\right).
$$
Consider a sequence $\left(\upsilon^n\right)_{n \geq 1} \in \mathcal{M}^p_{+}(\mathcal{X})$ converging to $\upsilon \in \mathcal{M}^p_+(\mathcal{X})$ in $\tau_p$ such that there exists $C>0$ so that
$$
\sup _{n \geq 1} \int_{\mathcal{X}}\left(1+d_{\mathcal{X}}\left(x, x_0\right)^p\right) \upsilon^n(\mathrm{d} x) \leq C.
$$
Consider also a sequence $\left(y^n\right)_{n \geq 1} \in \mathcal{Y}$ converging to $y \in \mathcal{Y}$ such that there exists a compact set $\mathcal{K} \subset \mathcal{Y}$ so that for all $n \geq 1, y^n \in \mathcal{K}$. Then,
$$
\int_{\mathcal{X}} \varphi\left(x, y^n\right) \upsilon^n(\mathrm{d}x) \underset{n \rightarrow \infty}{\longrightarrow} \int_{\mathcal{X}} \varphi(x, y) \upsilon(\mathrm{d} x).
$$
\item[(ii)]  Let $\Theta, \mathcal{X}, \mathcal{Y}$ be complete, separable metric spaces. Let $\eta \in \mathcal{M}^p_{+}(\Theta)$. Let $\varphi: \Theta \times \mathcal{X} \times \mathcal{Y} \rightarrow \mathbb{R}$, be a measurable map and assume that for every $t \in \Theta, \varphi(t, \cdot)$ is continuous. We assume the following growth condition on $\varphi$ : there exists $c \geq 0$ and $\left(t_0, x_0, y_0\right) \in \Theta \times \mathcal{X} \times \mathcal{Y}$
$$
|\varphi(t, x, y)| \leq c\left(1+d_{\Theta}\left(t, t_0\right)^p+d_{\mathcal{X}}\left(x, x_0\right)^p+d_{\mathcal{Y}}\left(y, y_0\right)^p\right).
$$
Suppose that a sequence of measurable functions $\psi^n: \Theta \rightarrow \mathcal{Y}$ converges $\eta$-a.e. in $\Theta$ to a measurable function $\psi: \Theta \rightarrow \mathcal{Y}$ and that $\left(\upsilon_t^n(\mathrm{d} x) \eta(\mathrm{d} t)\right)_{n \geq 1} \subset \mathcal{M}^p_{+}(\Theta \times \mathcal{X})$ converges to $\upsilon_t(\mathrm{d} x) \eta(\mathrm{d} t) \in \mathcal{M}^p_{+}(\Theta \times \mathcal{X})$ in $\bar{\tau}_p$, where $\left(\upsilon^n\right)_{n \geq 1}$ and $\upsilon$ are transition kernels from $\Theta$ to $\mathcal{X}$. Suppose also that there exists a constant $C>0$ such that, $\eta$-a.e.,
$$
\sup _{n \geq 1} \int_{\mathcal{X}}\left(1+d_{\mathcal{X}}\left(x, x_0\right)^p\right) \upsilon_t^n(\mathrm{d}x) \leq C.
$$
Moreover, suppose that there exists a compact set $\mathcal{K} \subset \mathcal{Y}$ such that for all $n \geq 1, \psi^n(t) \in \mathcal{K}$, $\eta$-a.e.. Then,
$$
\int_{\Theta} \int_{\mathcal{X}} \varphi\left(t, x, \psi^n(t)\right) \upsilon_t^n(\mathrm{d} x) \eta(\mathrm{d} t) \underset{n \rightarrow \infty}{\longrightarrow} \int_{\Theta} \int_{\mathcal{X}} \varphi(t, x, \psi(t)) \upsilon_t(\mathrm{d} x) \eta(\mathrm{d} t).
$$
\end{itemize}
\end{lemma}
\vskip 10pt
\ \\
\noindent
\textbf{Acknowledgements}:  Zongxia  Liang  is supported by the National Natural Science Foundation of China under grant no. 12271290. Xiang Yu and Keyu Zhang are supported by the Hong Kong RGC General Research Fund (GRF) under grant no. 15211524, the Hong Kong Polytechnic University research grant under no. P0045654 and the Research Centre for Quantitative Finance at the Hong Kong Polytechnic University under grant no. P0042708. 

	\bibliographystyle{siam}
        {\small
		\bibliography{ref}}
	\end{document}